\documentclass[a4paper,10pt]{article}
\usepackage{amssymb,amsmath,amsthm}
\usepackage{fullpage}
\usepackage{hyperref}
\usepackage{eucal}
\usepackage{mathrsfs}
\usepackage{color}
\usepackage{stackrel}
\usepackage{graphicx}
\newcommand{\ie}{\emph{i.e.}}
\newcommand{\eg}{\emph{e.g.}}
\newcommand{\cf}{\emph{cf.}}

\newcommand{\Real}{\mathbb{R}}
\newcommand{\Com}{\mathbb{C}}
\newcommand{\Nat}{\mathbb{N}}
\newcommand{\Int}{\mathbb{Z}}
\newcommand{\Sphere}{\mathbb{S}}
\newcommand{\sgn}{\mathop{\mathrm{sgn}}\nolimits}

\newcommand{\dist}{\mathrm{dist}}

\newcommand{\sii}{L^2}

\newcommand{\der}{\mathrm{d}}

\renewcommand{\Re}{\mathrm{Re}\,}
\renewcommand{\Im}{\mathrm{Im}\,}
\newtheorem{thm}{Theorem}[section] 
\newtheorem{Theorem}[thm]{Theorem}
\newtheorem{Lemma}[thm]{Lemma}
\newtheorem{Proposition}[thm]{Proposition}
\newtheorem{Corollary}[thm]{Corollary}

\theoremstyle{definition}
\newtheorem{Remark}[thm]{Remark}
\numberwithin{equation}{section}

\def\OMIT#1{}
\usepackage{cancel}
\usepackage[normalem]{ulem}
\definecolor{DarkGreen}{rgb}{0,0.5,0.1} 

\newcommand\soutD{\bgroup\markoverwith
{\textcolor{DarkGreen}{\rule[.5ex]{2pt}{1pt}}}\ULon}
\newcommand\soutP{\bgroup\markoverwith
{\textcolor{blue}{\rule[.5ex]{2pt}{1pt}}}\ULon}

\newcommand{\Hm}[1]{\leavevmode{\marginpar{\tiny%
$\hbox to 0mm{\hspace*{-0.5mm}$\leftarrow$\hss}%
\vcenter{\vrule depth 0.1mm height 0.1mm width \the\marginparwidth}%
\hbox to
0mm{\hss$\rightarrow$\hspace*{-0.5mm}}$\\\relax\raggedright #1}}}

\begin{document}
%
\title{\textbf{\LARGE Quantitative uniform resolvent estimates}}
\author{Piero D'Ancona,$^a$ \
J\'er\'emy Faupin,$^b$ \ and \
David Krej\v{c}i\v{r}{\'\i}k\,$^c$}
\date{\small 
\begin{quote}
\emph{
\begin{itemize}
\item[$a)$]  
Department of Mathematics ``Guido Castelnuovo'', 
University of Rome ``La Sapienza'',
Piazzale Aldo Moro 5, 00185 Rome, Italy;
dancona@mat.uniroma1.it.%
\item[$b)$]
Universit{\'e} de Lorraine, CNRS, IECL, F-57000 Metz, France;
jeremy.faupin@univ-lorraine.fr.%
\item[$c)$] 
Department of Mathematics, Faculty of Nuclear Sciences and 
Physical Engineering, Czech Technical University in Prague, 
Trojanova 13, 12000 Prague 2, Czechia;
david.krejcirik@fjfi.cvut.cz.%
\end{itemize}
}
\end{quote}
{\small 24 June 2026}}
\maketitle
%

%
\begin{abstract}
\noindent
We derive quantitative uniform resolvent estimates
for Schr\"odinger operators on the half-line 
with inverse-square potentials,
which provide a sharp behaviour in the limit of large coupling.
Our approach is based on a matrix representation of
the boundary value of a weighted resolvent.  
  The partial wave decomposition then turns these one-dimensional
  channel estimates into explicit weighted resolvent estimates for
  the Laplacian, its inverse-square potential perturbations
  and for the magnetic Laplacian with an Aharonov--Bohm potential.
  We also obtain exact
  Simon-type identities for the imaginary parts of the weighted
  resolvents of these operators.
\medskip
\begin{itemize}
\item[\textbf{Keywords:}]
Uniform resolvent estimates; Bessel operators; Hankel transform;
inverse square potentials; Aharonov--Bohm potential; 
Kato smoothness.
\item[\textbf{MSC (2020):}]
Primary 35J10, 35P25; Secondary 33C10, 47A10, 47A55, 81Q10.
\end{itemize}
\end{abstract}
%
 
\section{Introduction}

Let $-\Delta$ be the self-adjoint realisation 
of the Laplacian in $L^2(\Real^d)$ with $d\ge 3$.
The resolvent cannot be bounded uniformly near 
its spectrum $[0,\infty)$, since
\begin{equation}\label{sa}
  \|(-\Delta-z)^{-1}\|_{L^2(\Real^d)\to L^2(\Real^d)}
  =
  \frac{1}{\dist(z,[0,\infty))} .
\end{equation}
However, it is also well known that a uniform 
resolvent estimate can be achieved when the Laplacian
is reconsidered as an operator between weighted spaces.
For instance, one has
\begin{equation}\label{uniform}
  \sup_{z\in\Com\setminus[0,\infty)}
  \left\|
    |x|^{-1}(-\Delta-z)^{-1}|x|^{-1}
  \right\|_{L^2(\Real^d)\to L^2(\Real^d)}
  <
  \infty .
\end{equation}
Here the weight $|x|^{-1}$ is scale invariant for the Laplacian 
and it is naturally associated with the classical Hardy inequality.  
Estimate~\eqref{uniform} is 
behind the limiting absorption principle and
Kato smoothness and smoothing estimates 
\cite{Kato_1966,Kato-Yajima_1989,Simon_1992}.  
It has a perturbative
interpretation as well: Kato's abstract theory and the
Birman--Schwinger principle convert such bounds into
spectral stability statements for small
(possibly non-self-adjoint) perturbations 
subordinated to the same critical weight 
\cite{Kato_1966,Frank_2011,FKV,HK2}.

More specifically,
potentials of the form $c\,|x|^{-2}$ 
with $c \in \Com$ 
represent critical perturbations of the Laplacian,
in view of the optimal Hardy inequality
$-\Delta \geq c_d \, |x|^{-2}$ with $c_d := (d-2)^2/4$.
The related critical electromagnetic Helmholtz,
Strichartz and smoothing estimates, 
as well as spectral bounds, have been intensively 
studied in recent years
\cite{Burq-Planchon-Stalker-TahvildarZadeh_2003,
Burq-Planchon-Stalker-TahvildarZadeh_2004,
Barcelo-Vega-Zubeldia_2013,Bui-DAncona-Duong-Li-Ly_2017,
Bouclet-Mizutani_2018,Mizutani-Zhang-Zheng_2020}.  
What is more, even if the classical Hardy inequality
fails in two dimensions, there exist magnetic Hardy inequalities
\cite{Laptev-Weidl_1999,CK}
and the inverse-square potential is particularly relevant 
for the singular Aharonov--Bohm field 
\cite{Pankrashkin-Richard_2011,K13,Fanelli-Zhang-Zheng_2023}.

The purpose of this paper is to make~\eqref{uniform} quantitative
and to identify the mechanism behind the constants. 
Let us now explain the main idea of this paper
and present our main results.

As a first step, we note that
a standard maximum principle reduction and scaling invariance (see Section \ref{subsec:prelim} below for the detailed argument) yield
\begin{equation}\label{uniform-reduc}
\sup_{z\in\Com\setminus[0,\infty)}
  \left\|
    |x|^{-1}(-\Delta-z)^{-1}|x|^{-1}
  \right\|_{L^2(\Real^d)\to L^2(\Real^d)}
=  \left\|
    |x|^{-1}(-\Delta-(1\pm i0))^{-1}|x|^{-1}
  \right\|_{L^2(\Real^d)\to L^2(\Real^d)}
   .
\end{equation}
In this way, the problem is reduced to limiting values
at a single point of the positive real axis.

\subsection{Reduction to one-dimensional Bessel operators}
The main insight of this paper is that
the mechanism to compute the right-hand side of~\eqref{uniform-reduc} 
is one-dimensional. Indeed, after a decomposition into spherical
harmonics, every angular momentum channel is governed by a Bessel
operator
\begin{equation}\label{Bessel}
  h_\nu
  :=
  -\partial_r^2+\frac{\nu^2-1/4}{r^2}
  \qquad\hbox{in}\qquad 
  L^2((0,\infty)) =: \sii
  \,,
\end{equation}
introduced as the Friedrichs extension
of the operator initially defined on
$C_0^\infty((0,\infty))$. 
The central object of this paper
is therefore the weighted boundary value
\begin{equation}\label{channel}
  A_\nu^\pm
  :=
  r^{-1}(h_\nu-(1\pm i0))^{-1}r^{-1}.
\end{equation}

The spectral theory of~$h_\nu$,
including the Hankel transform to diagonalise it, is
classical; see, for example, \cite{Derezinski-Richard_2017}.  
The
new point here is different.  We compute $A_\nu^\pm$ as an explicit
operator matrix in a basis adapted to the Hankel transform.  This
matrix separates the threshold vector, the low energy part, and the
high energy part.  It gives lower bounds, the large order asymptotic law, 
and the sharp imaginary
part discovered by Simon in the Laplacian setting \cite{Simon_1992}.

\subsection{Multidimensional examples beyond the Laplacian}
In this picture, the relevant parameter~$\nu$ 
becomes the smallest effective Bessel order,
which depends on the particular model considered. 

The canonical example 
from~\eqref{uniform-reduc}
is the Laplacian~$-\Delta$ 
in $\sii(\Real^d)$ with $d \geq 3$ defined on $H^2(\Real^d)$.
In this case,
the orders are
\begin{equation*}
  \nu_{\ell,d} := \ell+\frac d2-1,
  \qquad \ell \in \Nat := \{0,1,2,\dots\} .
\end{equation*}

Another example is
the inverse square Hamiltonian 
\begin{equation}\label{inverse}
  H_{d,c} := -\Delta+c\,|x|^{-2}
  \qquad \mbox{in} \qquad
  \sii(\Real^d)
\end{equation}
with $d\ge2$ and real $c>-c_d=-(d-2)^2/4$. 
We introduce it as the Friedrichs extension of this operator
initially defined on $C_0^\infty(\Real^d\setminus\{0\})$.
In this case, replacing~$-\Delta$ by~$H_{d,c}$ in~\eqref{uniform-reduc},
the orders become
\begin{equation*}
  \nu_{\ell,d,c}
  :=
  \sqrt{\nu_{\ell,d}^2+c} .
\end{equation*}

Finally, we consider the two-dimensional magnetic Hamiltonian
\begin{equation}\label{AB.operator}
  -\Delta_\alpha := (-i\nabla-A_\alpha)^2
  \qquad \mbox{in} \qquad
  \sii(\Real^2)
\end{equation}
with the Aharonov--Bohm potential 
$A_\alpha (x) := \alpha \, (-x_2,x_1) \, |x|^{-2}$,
where $\alpha \in \Real$, 
introduced as the Friedrichs extension of this operator
initially defined on $C_0^\infty(\Real^2\setminus\{0\})$.
In this case,
\begin{equation*}
  \nu_{m,\alpha}
  := |m-\alpha|,
  \qquad m \in \Int.
\end{equation*}

The applications below are consequences of this single channel principle.
 
\subsection{Main results for the one-dimensional Bessel operators}
The main methodological novelty of this paper is 
Theorem~\ref{Thm.matrix} below,
in which we reveal a suitable matrix representation of~$A_\nu^\pm$
introduced in~\eqref{channel}.
It would be too cumbersome to present it in the introduction.
Let us only mention that it is based on an orthonormal basis 
adapted to the Hankel transform.
Instead, we collect here the main results about the norm of~$A_\nu^\pm$,
which we establish in this paper 
based particularly on the crucial Theorem~\ref{Thm.matrix}.
\begin{Theorem}[Bessel operators]\label{Thm.new}
Given any positive~$\nu$, let $A_\nu^\pm$ be as in~\eqref{channel}.
Then
\begin{enumerate}
\item[\emph{(i)}] 
$
\displaystyle
  \|\Im A_\nu^\pm\|_{\sii \to \sii} 
  = \frac{\pi}{4\nu} 
$,
\item[\emph{(ii)}] 
$
\displaystyle
  \|A_\nu^\pm\|_{\sii \to \sii}  
  = \frac{1}{\nu^2} 
$
\quad if \quad $0<\nu\le\frac12$,
\item[\emph{(iii)}] 
$
\displaystyle
  \|A_\nu^\pm\|_{L^2\to L^2}
  =
  \frac{1}{\nu}
  +
  O(\nu^{-3/2})
$  
\quad as \quad
$
  \nu\to\infty
$,
\item[\emph{(iv)}] 
$
\displaystyle
  \|A_\nu^\pm\|_{L^2\to L^2}\le\mathcal{C}(\nu)
$,
\quad where
\begin{equation}\label{eq:def-Cnu}
  \mathcal{C}(\nu)
  :=
  \begin{cases}
     1/\nu^{2},
      & 0<\nu\le\frac12, \\
     2/\nu ,
      & \frac12<\nu\le\frac1{\sqrt3}, \\
      1/\sqrt{\nu^{2}-\frac 14},
      & \frac1{\sqrt3}<\nu ,
  \end{cases}
\end{equation}
\end{enumerate}
\end{Theorem}

Part~(i) is an analogue of the Simon-type identity~\cite{Simon_1992}
for the Laplacian, see~\eqref{Simon} below. 
Part~(ii) demonstrates an interesting fact that an exact identity 
can be derived also for the full complex value of~$A_\nu^\pm$, 
provided that the argument~$\nu$ is sufficiently small.
For all the other values of~$\nu$,
we still have explicit lower bounds  
(see Theorem~\ref{Thm.lower}
and Proposition~\ref{prop:lower-bound-Bessel-small-nu}).
For large values of~$\nu$, part~(iii) gives a sharp lower bound. 
Part~(iv) is the main result in the spirit of this paper
to derive quantitative uniform resolvent estimates.
 
\subsection{Main results for the multidimensional examples}
We next state the consequences for the full weighted resolvent.
For the free Laplacian, the Bessel channel estimates imply the
following explicit bound.
\begin{Theorem}[Laplacian]\label{thm:main_Laplacian}
  For every $d\ge3$,
  \begin{equation}\label{eq:main_Laplacian}
    \sup_{z\in\Com\setminus[0,\infty)}
    \left\|
      |x|^{-1}(-\Delta-z)^{-1}|x|^{-1}
    \right\|
    \le
    \begin{cases}
      4, & d=3, \\
      \displaystyle
      \frac{2}{\sqrt{(d-1)(d-3)}}, & d\ge4 .
    \end{cases}
  \end{equation}
\end{Theorem}

As a byproduct, we obtain the following monotonicity property. 
\begin{Proposition}[Dimensional monotonicity]\label{Prop.monotonicity}
  For each $z\in\Com\setminus[0,\infty)$, the maps
  \begin{equation*}
    d\mapsto
    \left\|
      |x|^{-1}(-\Delta-z)^{-1}|x|^{-1}
    \right\|_{L^2(\Real^{2d+1})\to L^2(\Real^{2d+1})}
  \end{equation*}
  and
  \begin{equation*}
    d\mapsto
    \left\|
      |x|^{-1}(-\Delta-z)^{-1}|x|^{-1}
    \right\|_{L^2(\Real^{2d+2})\to L^2(\Real^{2d+2})}
  \end{equation*}
  are nonincreasing on $\{1,2,\dots\}$.
\end{Proposition}

We emphasise that our approach provides not only an explicit constant 
in~\eqref{eq:main_Laplacian}, but also a new proof 
of the very fact~\eqref{uniform}.
Actually, we give two arguments, the alternative being the method of multipliers,
which has been used by other authors, too.

Except for the well-known three-dimensional case,
we are not aware of attempts to quantify 
the uniform resolvent estimate~\eqref{uniform}.
(An exception is the crude estimate obtained 
by Cossetti and one of the present authors in~\cite{CK4}.) 
Be all as it may be, we note that the same bound~\eqref{eq:main_Laplacian}
can be deciphered from the proof of
\cite[Thm.~3]{Burq-Planchon-Stalker-TahvildarZadeh_2004}
(after quantifying the generic constants appearing there),
which is close to our multiplier approach
(see also \cite[Thm.~1]{Burq-Planchon-Stalker-TahvildarZadeh_2003}).
We also note that a similar bound with a worse constant 
appears in \cite[Thm.~4]{Mochizuki_2010a} 
(see also \cite[Thm.~1]{Mochizuki_2010b}).

The constant of Theorem~\ref{thm:main_Laplacian}
is sharp in three dimensions:
\begin{equation}\label{eq:sharpdim3}
  \sup_{z\in\Com\setminus[0,\infty)}
  \left\|
    |x|^{-1}(-\Delta-z)^{-1}|x|^{-1}
  \right\|_{L^2(\Real^3)\to L^2(\Real^3)}
  =
  4 .
\end{equation}
This follows from an explicit lower bound 
that we also derive (Proposition~\ref{pro:opt3D}).
We leave as an open problem whether 
our estimate~\eqref{eq:main_Laplacian} is sharp for $d \geq 4$,
 but we prove that it is sharp in the limit of high dimensions 
 (Proposition \ref{prop:asymptotics-d}):
\begin{equation*}
\sup_{z\in\Com\setminus[0,\infty)}
  \left\|
    |x|^{-1}(-\Delta-z)^{-1}|x|^{-1}
  \right\|_{L^2(\Real^d)\to L^2(\Real^d)}=\frac{2}{d} + O(d^{-3/2})
  \qquad \mbox{as} \qquad d\to\infty.
\end{equation*}

Sharp results can be obtained also for the imaginary part
of the operators. In particular, 
we recover Simon's identity~\cite{Simon_1992}
\begin{equation}\label{Simon} 
  \sup_{z\in\Com\setminus[0,\infty)}
  \left\|
    |x|^{-1}\Im(-\Delta-z)^{-1}|x|^{-1}
  \right\|
  =
  \frac{\pi}{2(d-2)}.
\end{equation}
In our approach, Simon's identity is not isolated but just
a special case of analogous exact identities for
the other critical models.
Indeed, \eqref{Simon} is the first member of a family governed only
by the lowest effective angular momentum, as the following
results indicate.
\begin{Theorem}[Exact imaginary parts]\label{thm:imaginary-critical}
  Given $d\ge2$ and $c>-c_d$, 
  let~$H_{d,c}$ be the inverse square Hamiltonian defined in~\eqref{inverse}.
  Then
  \begin{equation}\label{eq:inverse square-imaginary}
    \sup_{z\in\Com\setminus[0,\infty)}
    \left\|
      |x|^{-1}\Im(H_{d,c}-z)^{-1}|x|^{-1}
    \right\|
    =
    \frac{\pi}{4\nu_{0,d,c}} .
  \end{equation}
  Moreover, for the two-dimensional Aharonov--Bohm Hamiltonian
  $-\Delta_\alpha$ defined in~\eqref{AB.operator}, 
  if $\alpha\notin\Int$, then
  \begin{equation}\label{eq:AB-imaginary}
    \sup_{z\in\Com\setminus[0,\infty)}
    \left\|
      |x|^{-1}\Im(-\Delta_\alpha-z)^{-1}|x|^{-1}
    \right\|
    =
    \frac{\pi}{4\,\dist(\alpha,\Int)} .
  \end{equation}
\end{Theorem}

Coming back to the uniform resolvent estimates,
for the inverse square family we obtain the following result.
\begin{Theorem}[Inverse square potentials]\label{thm:inverse square}
  Given $d\ge2$  and $c>-c_d$, let $H_{d,c}$ be as in~\eqref{inverse}.
  Then
  \begin{equation}\label{eq:inverse square-bound}
    \sup_{z\in\Com\setminus[0,\infty)}
    \left\|
      |x|^{-1}(H_{d,c}-z)^{-1}|x|^{-1}
    \right\|
    \le
    \mathcal{C}(\nu_{0,d,c}),
  \end{equation}
 where $\mathcal{C}(\nu)$ is defined in \eqref{eq:def-Cnu}. 
 If $0<\nu_{0,d,c}\le1/2$, then equality holds in
  \eqref{eq:inverse square-bound}.
\end{Theorem}

Finally, the same channel analysis yields the Aharonov--Bohm
estimate in two dimensions.  Here the lowest effective order is the
distance from the magnetic flux to the integers.
\begin{Theorem}[Aharonov--Bohm potentials]\label{Thm.AB}
  Given $\alpha\notin\Int$, 
  let $-\Delta_\alpha$ be as in~\eqref{AB.operator}.
  Then
  \begin{equation}\label{AB}
    \sup_{z\in\Com\setminus[0,\infty)}
    \left\|
      |x|^{-1}(-\Delta_\alpha-z)^{-1}|x|^{-1}
    \right\|
    =
    \frac{1}{\dist(\alpha,\Int)^2} .
  \end{equation}
\end{Theorem}
\noindent
For $\alpha \in \Int$, $-\Delta_\alpha$ is unitarily equivalent 
to $-\Delta$ in $\sii(\Real^2)$ and there is no uniform resolvent
estimate for the Laplacian in two dimensions.

\subsection{Applications}
\subsubsection*{Kato smoothing and evolution control}
Let $H=-\Delta$ or $H_{d,c}$ or $-\Delta_{\alpha}$.
The exact imaginary part identities 
\eqref{Simon}, \eqref{eq:inverse square-imaginary}
and~\eqref{eq:AB-imaginary}
give sharp smoothness constants in Kato's theorem
\cite{Kato_1966,DAncona_2015}.
Indeed, denoting
\begin{equation*}
  B_H
  :=
  \sup_{z\in\Com\setminus[0,\infty)}
  \left\||x|^{-1}\Im(H-z)^{-1}|x|^{-1}\right\|
\end{equation*}
we have, for every $u \in L^2(\Real^d)$,
\begin{equation*}
  \int_\Real
  \left\||x|^{-1} e^{-itH}u\right\|_{L^2(\Real^d)}^2\,\der t
  \le
  2B_H\|u\|_{L^2(\Real^d)}^2 .
\end{equation*}
In the cases covered by~\eqref{Simon} and
Theorem~\ref{thm:imaginary-critical},
this gives the explicit constants
\begin{equation*}
  2B_{-\Delta}=\frac{\pi}{d-2},
  \qquad
  2B_{H_{d,c}}=\frac{\pi}{2\nu_{0,d,c}},
  \qquad
  2B_{-\Delta_\alpha}
  =
  \frac{\pi}{2\,\dist(\alpha,\Int)} .
\end{equation*}

The full weighted resolvent bounds 
of Theorems~\ref{thm:main_Laplacian},
\ref{thm:inverse square} and~\ref{Thm.AB}
give the corresponding
supersmooth estimates, which are equivalent to the inhomogeneous 
smoothing estimate below~\cite{DAncona_2015}.  
Namely, with
\begin{equation}\label{CH}
  C_H
  :=
  \sup_{z\in\Com\setminus[0,\infty)}
  \left\||x|^{-1}(H-z)^{-1}|x|^{-1}\right\|,
\end{equation}
Kato's retarded estimate yields
\begin{equation*}
  \left\|
    |x|^{-1}\int_0^t e^{-i(t-s)H} F(s)\,\der s
  \right\|_{L^2(\Real;L^2(\Real^d)))}
  \le
  2\pi C_H \||x|F\|_{L^2(\Real;L^2(\Real^d)))} .
\end{equation*}
Here $C_H$ is bounded by the constants in
Theorems~\ref{thm:main_Laplacian}, \ref{thm:inverse square},
and~\ref{Thm.AB}; in the Aharonov--Bohm case (Theorem \ref{Thm.AB}), $C_H$ is exactly
$\dist(\alpha,\Int)^{-2}$.

\subsubsection*{Spectral enclosures and stability}
Let again $H$ be one of the self-adjoint model operators above.
Consider a possibly complex-valued perturbation $V:\Real^d \to \Com$
such that $|x|^2 \, V\in L^\infty(\Real^d)$.  
The Birman--Schwinger principle~\cite{HK2}
gives the following necessary condition 
for an eigenvalue $z\in\Com \setminus [0,\infty)$ 
of a closed realisation~$H_V$ of~$H+V$ to exist:
\begin{equation*}
  1
  \le
  \| |x|^2 V \|_{L^\infty(\Real^d)}
  \left\||x|^{-1}(H-z)^{-1}|x|^{-1}\right\|.
\end{equation*}
Equivalently,
\begin{equation*}
  \sigma_{\rm p}(H_V) \setminus [0,\infty)
  \subset
  \left\{
    z\in\Com \setminus [0,\infty) :
    \left\||x|^{-1}(H-z)^{-1}|x|^{-1}\right\|
    \ge
    \||x|^2V\|_{L^\infty(\Real^d)}^{-1}
  \right\}.
\end{equation*}

In particular,
a sufficient condition for the absence of eigenvalues of~$H$
in $\Com \setminus [0,\infty)$
reads
\begin{equation}\label{smallness}
  C_H \, \||x|^2V\|_{L^\infty(\Real^d)}<1,
\end{equation}
where~$C_H$ is defined in~\eqref{CH}.
What is more, the same smallness condition~\eqref{smallness}
ensures the total absence of eigenvalues
(including eigenvalues in $[0,\infty)$) \cite[Corol.~2]{HK2}.
In fact, $H$~is similar to~$-\Delta$ via bounded
and boundedly invertible transformation~\cite{Kato_1966}.

Theorems~\ref{thm:main_Laplacian}, \ref{thm:inverse square},
and~\ref{Thm.AB} make 
the sufficient condition~\eqref{smallness} quantitative.
For instance, in the Aharonov--Bohm case this smallness condition
reads $\||x|^2V\|_{L^\infty(\Real^d)}<\dist(\alpha,\Int)^2$.

\subsection{Structure of the paper}
The paper is organised as follows.  
In Section~\ref{sec:matrix},
we begin with the matrix representation~$A_\nu^\pm$ 
and its immediate consequences;
in particular, we establish parts~(i) and~(iii) of Theorem~\ref{Thm.new}.  
More subtle one-channel estimates,
including parts~(ii) and~(iv) of Theorem~\ref{Thm.new},
are established in Section~\ref{sec:Bessel}.
The applications to the Laplacian, 
inverse square potentials and the Aharonov--Bohm magnetic fields
are performed in Section~\ref{Sec.all}. 
In the final Section~\ref{Sec.multipliers},
we present an alternative way to produce 
the uniform resolvent estimate for the Laplacian 
(Theorem~\ref{thm:main_Laplacian})
via the method of multipliers.

\section{The matrix computation for one Bessel channel}\label{sec:matrix}
Given any $\nu>0$, recall the definition~\eqref{Bessel}
of the Bessel operator~$h_\nu$ from the introduction.
We are interested in the boundary values~$A_\nu^\pm$
of the weighted resolvent at the positive energy~$1$ 
introduced in~\eqref{channel}.
By a scaling invariance and the maximum principle 
(see Section \ref{subsec:prelim} below), the norm of
$r^{-1}(h_\nu-(\lambda\pm i0))^{-1}r^{-1}$ is independent of
$\lambda>0$ and equals $\sup_{z\in\Com\setminus[0,\infty)}
  \|
    r^{-1}(h_\nu-z)^{-1} r^{-1}
  \|_{L^2\to L^2}$.
  So it is indeed enough to understand~\eqref{channel}.

In this section, we decompose $A^\pm_\nu$ in a suitable
orthonormal basis of~$L^2$, and then deduce from it
both a lower bound and the asymptotic behaviour as $\nu\to\infty$
of $\|A_\nu^\pm\|_{L^2\to L^2}$. 
More specifically, the main result of this section 
is the following theorem.
\begin{Theorem}\label{Thm.matrix}
Let $\nu>0$. There is an orthonormal basis $(e_{\nu,\ell})_{\ell\in\mathbb{Z}}$ of $L^2$ such that $A^\pm_\nu$ decomposes in this basis as follows:
Denoting 
$
  a_{\ell\ell'} 
  := \langle e_{\nu,\ell}, A^\pm_\nu e_{\nu,\ell'} \rangle_{\sii}
$,
\begin{align*}
(a_{\ell\ell'})_{\ell,\ell'\in\mathbb{Z}} = \begin{pmatrix}
\ddots & \vdots & \vdots & \vdots & \vdots & \vdots & \reflectbox{$\ddots$} \\
\dots & a_{-2,-2} & 0 & a_{-2,0} & 0 & 0 & \dots \\
\dots & 0 & a_{-1,-1} & a_{-1,0} & 0 & 0 & \dots \\
\dots & a_{0,-2} & a_{0,-1} & a_{0,0} & a_{0,1} & a_{0,2} & \dots \\
\dots & 0 & 0 & a_{1,0} & a_{1,1} & a_{1,2} & \dots \\
\dots & 0 & 0 & a_{2,0} & a_{2,1} & a_{2,2} & \dots \\
\reflectbox{$\ddots$} & \vdots & \vdots & \vdots & \vdots & \vdots & \ddots
\end{pmatrix}
,
\end{align*}
where for $\ell,\ell'\ge1$,
\begin{align*}
& a_{0,0} = \pm i\frac{\pi}{4\nu} + \frac{1}{4\nu^2}, \\
& a_{-\ell,\ell'} = a_{\ell',-\ell} = 0, \\
& a_{-\ell,-\ell'} =  -\frac{\delta_{\ell\ell'}}{4\ell(\nu+\ell)}, \\
& a_{0,-\ell} = a_{-\ell,0} = \frac{(-1)^{\ell}}{4\ell(\nu+\ell)}\sqrt{1+\frac{2\ell}{\nu}}, \\
& a_{0,\ell} = a_{\ell,0} = -\frac{1}{4\ell(\nu+\ell)}\left[(-1)^{\ell}-\frac{\ell!}{(\nu)_\ell}\right]\sqrt{1+\frac{2\ell}{\nu}} ,\\
& a_{\ell,\ell'} = \frac{1}{4\ell\ell'}\frac{\max(\ell,\ell')! \, \Gamma(\nu+\min(\ell,\ell'))}{(\min(\ell,\ell')-1)! \, \Gamma(\nu+\max(\ell,\ell')+1)}\sqrt{1+\frac{2\ell}{\nu}}\sqrt{1+\frac{2\ell'}{\nu}},
\end{align*}
where $(\nu)_\ell=\nu(\nu+1)\cdots(\nu+\ell-1)$ is the Pochammer symbol.
\end{Theorem}

Theorem~\ref{Thm.matrix} will be proven below.
Before, we discuss its consequences.

\subsection{A lower bound to the norm of
\texorpdfstring{$A_\nu^\pm$}{Apm}}
In what follows, $J_\nu$, $Y_\nu$ are the usual Bessel functions of the first and second kind, respectively.
Moreover, $H_\nu^{(1)}=J_\nu+iY_\nu$ and $H_\nu^{(2)}=J_\nu-iY_\nu$
are the Hankel function of the first and second kind, respectively. 
To shorten the notations, we denote
%
$
p_\nu(r) := r^{-1/2}J_\nu(r),
$
%
and define the distinguished vector 
\begin{equation}\label{distinguished}
  e_{\nu,0} := (2\nu)^{1/2} \, p_\nu 
  = (2\nu)^{1/2}r^{-1/2}J_\nu(r)
  \,.
\end{equation}

First, we are concerned with the imaginary part of~$A_\nu^\pm$,
establishing part~(i) of Theorem~\ref{Thm.new}.
\begin{Proposition}\label{prop:main_Bessel_Imaginary}
Let $\nu>0$. One has
\begin{equation}\label{eq:main_Bessel_Imaginary}
\sup_{z\in\mathbb{C}\setminus[0,\infty)}
\left\| r^{-1} \Im (h_\nu - z)^{-1} r^{-1}\right\|_{L^2 \to L^2}
=\frac{\pi}{4\nu}.
\end{equation}
\end{Proposition}
\begin{proof}
By the maximum principle and scaling 
(see Section~\ref{subsec:prelim} below) it is enough to compute 
the norm of the boundary values at $1\pm i0$.  
The integral kernel of~$h_\nu$ at $1 \pm i0$ reads
\begin{equation}\label{eq:Green_integral_kernel}
\begin{aligned}
  (h_\nu-(1+i0))^{-1}(r,R) 
  &=
 \frac{i\pi}{2} (rR)^{1/2}
\begin{cases}
 J_\nu(r) H_\nu^{(1)}(R) & \mbox{if}\quad  r<R, \\
 J_\nu(R) H_\nu^{(1)}(r) & \mbox{if}\quad  R<r ,
\end{cases}
  \\
   (h_\nu-(1-i0))^{-1}(r,R)
   &=
   -\frac{i\pi}{2} (rR)^{1/2}
   \begin{cases}
   J_\nu(r) H_\nu^{(2)}(R) & \mbox{if}\quad  r<R, \\
   J_\nu(R) H_\nu^{(2)}(r) & \mbox{if}\quad  R<r .
\end{cases}
\end{aligned}   
\end{equation}
Since $J_\nu$, $Y_\nu$ are real-valued,  
one has
\begin{equation*}
  \Im A_\nu^\pm (r,R) = \pm
  \frac{\pi}{2}(rR)^{-1/2}J_\nu(r)J_\nu(R).
\end{equation*}
Thus the imaginary part is the rank-one operator
\begin{equation}\label{rank-one}
  \pm\Im A_\nu^\pm
  =
  \frac{\pi}{2}
  |r^{-1/2}J_\nu\rangle\langle r^{-1/2}J_\nu|  = \frac{\pi}{2}
  |p_\nu\rangle\langle p_\nu| = \frac{\pi}{4\nu} 
  |e_{\nu,0}\rangle\langle e_{\nu,0}| ,
\end{equation}
where we have used  
$\|p_\nu\|_{\sii}^2=\|r^{-1/2}J_\nu\|_{L^2}^2=(2\nu)^{-1}$. 
This proves~\eqref{eq:main_Bessel_Imaginary}.
\end{proof}

As for the real part of~$A_\nu^\pm$, we have the following result.
\begin{Proposition}\label{prop:main_Bessel_Real-2}
Let $\nu>0$.  One has
%
\begin{equation}\label{real.identity}
  \left\langle
    e_{\nu,0},
    \Re A_\nu^\pm e_{\nu,0}
  \right\rangle_{L^2}
  =
  \frac{1}{4\nu^2}.
\end{equation}
\end{Proposition}
\begin{proof}
From~\eqref{eq:Green_integral_kernel} we deduce
\begin{equation}\label{eq:int_kernel_real}
  \Re A_\nu^\pm(r,R) =
 - \frac{\pi}{2} (rR)^{-1/2}
\begin{cases}
 J_\nu(r) Y_\nu(R) & \mbox{if}\quad r<R, \\
 J_\nu(R) Y_\nu(r) & \mbox{if}\quad R<r .
\end{cases}
\end{equation}
Hence
\begin{align*}
&\left\langle
  p_\nu, A_\nu^\pm p_\nu
\right\rangle_{L^2}
=
-\pi
\int_0^\infty
  \frac{J_\nu(r)^2}{r}
  \int_r^\infty
    \frac{Y_\nu(R)J_\nu(R)}{R}\,\mathrm{d}R\,\mathrm{d}r.
\label{eq:bound-im-real2}
\end{align*}
We compute the last integral by the Mellin transform.  First,
\begin{align*}
\mathcal{M}\{r^{-1}J_\nu(r)^2\}(s)
=
\int_0^\infty r^{s-2}J_\nu(r)^2\,\mathrm{d}r 
2^{s-2}
\frac{
  \Gamma(2-s)\Gamma\big(\nu-\frac12+\frac{s}{2}\big)}
{\Gamma\big(\frac32-\frac{s}2\big)^2
 \Gamma\big(\nu+\frac32-\frac{s}2\big)}
\end{align*}
for $1-2\nu<\mathrm{Re}(s)<2$, see \cite[6.574.2]{GR}.
Likewise
\begin{align*}
\mathcal{M}\{r^{-1}J_\nu(r)Y_\nu(r)\}(s)
&=
\int_0^\infty r^{s-2}J_\nu(r)Y_\nu(r)\,\mathrm{d}r \\
&=
-\frac{2^{s-1}}{\pi}
\cos\left(\frac{(s-1)\pi}{2}\right)
\frac{
  \Gamma(\nu-\frac12+\frac{s}2)
  \Gamma(-\frac12+\frac{s}{2})
  \Gamma(2-s)}
{\Gamma(\nu+\frac32-\frac{s}2)
 \Gamma(\frac32-\frac{s}2)}
\end{align*}
for $0<\Re s<1$, see \cite[2.13.15.5]{PBM}.  
Therefore
\begin{align*}
&\mathcal{M}
\left\{
  \int_r^\infty
  \frac{Y_\nu(R)J_\nu(R)}{R}\,\mathrm{d}R
\right\}(s) 
=
-\frac{2^s}{\pi s}
\cos\left(\frac{s\pi}{2}\right)
\frac{
  \Gamma(\nu+\frac{s}2)
  \Gamma(\frac{s}{2})
  \Gamma(1-s)}
{\Gamma(\nu+1-\frac{s}2)\Gamma(1-\frac{s}2)}
\end{align*}
for $-1<\Re s<0$.  
Parseval's formula gives, for $0<c<1$,
\begin{align*}
&-\pi
\int_0^\infty
  \frac{J_\nu(r)^2}{r}
  \int_r^\infty
    \frac{Y_\nu(R)J_\nu(R)}{R}\,\mathrm{d}R\,\mathrm{d}r
=
\frac{1}{2i\pi}
\int_{c-i\infty}^{c+i\infty}
\frac{1}{(1-s)^2}
\frac{1}{
  \big(\nu+\frac12-\frac{s}2\big)
  \big(\nu-\frac12+\frac{s}2\big)}
\,\mathrm{d}s .
\end{align*}
The residue at $s=-2\nu+1$ equals $1/(8\nu^3)$.
This proves the desired identity after recalling 
the relationship~\eqref{distinguished}.
\end{proof}

The imaginary part identity from Proposition \ref{prop:main_Bessel_Imaginary} gives a lower bound for $\|A_\nu^\pm\|_{L^2\to L^2}$.
Indeed, for all $\varphi\in L^2$, one has
\begin{align*}
\|A_\nu^\pm\|_{L^2\to L^2}\ge\big|\langle\varphi, A_\nu^\pm \varphi\rangle_{L^2}\big|&\ge\big|\Im\langle\varphi, A_\nu^\pm \varphi\rangle_{L^2}\big|=\big|\langle\varphi, \big[ \Im A_\nu^\pm \big]\varphi\rangle_{L^2}\big|.
\end{align*}
Taking the supremum over $\varphi\in L^2$ such that $\|\varphi\|_{L^2}=1$ and using Proposition \ref{prop:main_Bessel_Imaginary} gives
\begin{equation*}
\|A_\nu^\pm\|_{L^2\to L^2}\ge\frac{\pi}{4\nu}.
\end{equation*}
Now we are in a position to improve this bound. Indeed,
from the matrix representation 
of Theorem~\ref{Thm.matrix}
one can compute $\|A_\nu^\pm e_{\nu,0}\|_{L^2}$, 
as the following theorem shows.
\begin{Theorem}\label{Thm.lower}
Let $\nu>0$. One has
\begin{align*}
\|A_\nu^\pm\|_{L^2\to L^2}^2\ge\big\|A_\nu^\pm e_{\nu,0}\big\|_{L^2}^2=\frac{\pi^2}{12\nu^2}+\frac{1}{8\nu^4}.
\end{align*}
\end{Theorem}
\begin{proof}
Since $e_{\nu,0}(r)=(2\nu)^{\frac12}r^{-\frac12}J_\nu(r)$ is real-valued, 
we can write
$
  A_\nu^\pm e_{\nu,0}= (\Re A_\nu^\pm) e_{\nu,0} + i (\Im A_\nu^\pm ) e_{\nu,0}
$.
This implies
\begin{align*}
 & \big \| A_\nu^\pm \big\|_{L^{2}\to L^{2}}^{2} \ge \big\|A_\nu^\pm e_{\nu,0}\big\|_{L^2}^2 =  \big\|(\Re A_\nu^\pm) e_{\nu,0}\big\|_{L^2}^{2}+\big\|(\Im A_\nu^\pm ) e_{\nu,0}\big\|_{L^2}^{2} = \big\|(\Re A_\nu^\pm) e_{\nu,0}\big\|_{L^2}^{2}+\frac{\pi^2}{16\nu^2},
\end{align*}
where we have used~\eqref{rank-one} in the last equality. 
It remains to compute $\big\|(\Re A_\nu^\pm) e_{\nu,0}\big\|_{L^2}^{2}$. From the matrix representation of Theorem~\ref{Thm.matrix} and Parseval's identity, it follows that
\begin{align}
\big\|(\Re A_\nu^\pm) e_{\nu,0}\big\|_{L^2}^{2} 
&= \frac{1}{16\nu^4} + \frac{1}{16} \sum_{\ell\ge1} \frac{1}{\ell^2(\nu+\ell)^2}\Big(1+\frac{2\ell}{\nu}\Big)+ \frac{1}{16}\sum_{\ell\ge1} \frac{1}{\ell^2(\nu+\ell)^2}\Big((-1)^{\ell}-\frac{l!}{(\nu)_\ell}\Big)^2\Big(1+\frac{2\ell}{\nu}\Big) \notag \\
&= \frac{1}{16\nu^4} + \frac{1}{8\nu^2} \sum_{\ell\ge1} \Big(\frac{1}{\ell^2}-\frac{1}{(\nu+\ell)^2}\Big)-\frac{1}{8\nu^2} \sum_{\ell\ge1} (-1)^\ell\Big(\frac{1}{\ell^2}-\frac{1}{(\nu+\ell)^2}\Big)\frac{\ell!}{(\nu)_\ell} \notag\\
&\quad+\frac{1}{16\nu^2} \sum_{\ell\ge1} \Big(\frac{1}{\ell^2}-\frac{1}{(\nu+\ell)^2}\Big)\Big(\frac{\ell!}{(\nu)_\ell}\Big)^2. \label{eq:norm-Atildepnu}
\end{align}
We compute the three sums on the right-hand side of~\eqref{eq:norm-Atildepnu}. 
For the first sum, by definition of the trigamma function, we have
\begin{align}\label{eq:norm-Atildepnu2}
\sum_{\ell\ge1} \Big(\frac{1}{\ell^2}-\frac{1}{(\nu+\ell)^2}\Big)=\frac{\pi^2}{6}-\psi^{(1)}(\nu)+\frac{1}{\nu^2}.
\end{align}
To compute the second sum, we define
\begin{align*}
S_\nu := \sum_{\ell\ge1} (-1)^\ell \Big(\frac{1}{\ell^2}-\frac{1}{(\nu+\ell)^2}\Big)\frac{\ell!}{(\nu)_\ell},
\end{align*}
and then compute
\begin{align*}
S_{\nu+1}-S_{\nu}&=\sum_{\ell=1}^\infty(-1)^\ell\frac{\ell!}{(\nu)_{\ell}}\Big[ -\frac{1}{\ell(\nu+\ell)}+\frac{\ell+1}{(\nu+\ell)(\nu+\ell+1)^2}+\frac{1}{(\nu+\ell)^2(\nu+\ell+1)} \Big]=\frac{1}{(\nu+1)^2},
\end{align*}
which yields
\begin{align}\label{eq:norm-Atildepnu3}
S_\nu=\sum_{k\ge0}(S_{\nu+k}-S_{\nu+k+1})=-\sum_{k\ge0}\frac{1}{(\nu+k+1)^2}=\frac{1}{\nu^2}-\psi^{(1)}(\nu).
\end{align}
The third sum in \eqref{eq:norm-Atildepnu} simplifies to
\begin{align}\label{eq:norm-Atildepnu4}
\sum_{\ell\ge1} \Big(\frac{1}{\ell^2}-\frac{1}{(\nu+\ell)^2}\Big)\Big(\frac{\ell!}{(\nu)_\ell}\Big)^2=\sum_{\ell\ge1}\Big(\Big(\frac{(\ell-1)!}{(\nu)_\ell}\Big)^2 - \Big(\frac{\ell!}{(\nu)_{\ell+1}}\Big)^2\Big)=\frac{1}{\nu^2}.
\end{align}
Inserting \eqref{eq:norm-Atildepnu2}--\eqref{eq:norm-Atildepnu4} into \eqref{eq:norm-Atildepnu}, the result follows.
\end{proof}

\subsection{Asymptotics of the norm of
\texorpdfstring{$A_\nu^\pm$}{Apm} 
as \texorpdfstring{$\nu\to\infty$}{nutoinfty}}
It is not clear to us how to compute an exact expression of $\|A_\nu^\pm\|_{L^2\to L^2}$ from the matrix representation of 
Theorem~\ref{Thm.matrix}. However, we can obtain the following asymptotics as $\nu\to\infty$, establishing part~(iii) of Theorem~\ref{Thm.new}.
\begin{Theorem}\label{thm:asymptotics}
Let $\nu > 0$. One has
\begin{align*}
\big\| A_\nu^\pm \big\|_{L^2\to L^2} = \frac{1}{\nu} 
+ O(\nu^{-3/2}) 
  \qquad \mbox{as} \qquad \quad \nu\to\infty.
\end{align*}
\end{Theorem}
\begin{Remark}
We will show more precisely that, for all $\nu>0$,
\begin{align}\label{eq:asymptotics-bound}
\Big| \big\| A_\nu^\pm \big\|_{L^2\to L^2} - \frac{1}{\nu} \Big| \le \frac{\sqrt{3\psi^{(1)}(\nu+1)}}{2\sqrt{2}\nu}+\frac{1}{\nu^2},
\end{align}
where $\psi^{(1)}$ is the trigamma function (in particular $\psi^{(1)}(\nu)\sim\frac{1}{\nu}$ as $\nu\to\infty$), without trying to optimise the remainder term.
\end{Remark}
We prove Theorem \ref{thm:asymptotics} thanks to a suitable matrix representation of the auxiliary operator
\begin{equation}\label{operator.B}
B_\nu^\pm:=A_\nu^\pm-r^{-1}h_\nu^{-1}\mathbf{1}_{h_\nu\ge1}r^{-1}.
\end{equation}
\begin{Remark}\label{rk:ess-spectrum}
The operator $r^{-1}h_\nu^{-1}\mathbf{1}_{h_\nu\ge1}r^{-1}$
is self-adjoint and its spectrum reads
\begin{align}\label{eq:spectrum_aux}
\sigma\big( r^{-1}h_\nu^{-1}\mathbf{1}_{h_\nu\ge1}r^{-1} \big) = \sigma_{\mathrm{ess}}\big( r^{-1}h_\nu^{-1}\mathbf{1}_{h_\nu\ge1}r^{-1} \big ) = [0,\nu^{-2}].
\end{align}
Consequently,
\begin{equation}\label{eq:rk4}
\|r^{-1}h_\nu^{-1}\mathbf{1}_{h_\nu\ge1}r^{-1}\|_{L^2\to L^2}=\nu^{-2}.
\end{equation}
This follows from the integral kernel
\begin{equation*}
 h_\nu^{-1}(r,R) =
 \frac{1}{2\nu} 
\begin{cases}
 r^{\nu+\frac12}R^{-\nu+\frac12} 
 & \mbox{if}\quad \quad r<R \,, \\
 r^{-\nu+\frac12}R^{\nu+\frac12}  
 & \mbox{if}\quad \quad R<r \,.
\end{cases}
\end{equation*}
In more detail,
let $\mathcal{D}:L^2((0,\infty))\to L^2(\mathbb{R})$ be the
unitary operator
$
  (\mathcal{D}\varphi)(x)
  =
  e^{x/2}\varphi(e^x).
$  
Under this transform, $r^{-1}h_\nu^{-1}r^{-1}$ is unitarily
equivalent to the operator with kernel
$(2\nu)^{-1}e^{-\nu|x-y|}$.  After Fourier transform, this becomes
the multiplier $(\xi^2+\nu^2)^{-1}$.  This implies that
\begin{align*}
\sigma\big( r^{-1}h_\nu^{-1}r^{-1} \big) = \sigma_{\mathrm{ess}}\big( r^{-1}h_\nu^{-1}r^{-1} \big ) = [0,\nu^{-2}],
\end{align*}
and that $\|r^{-1}h_\nu^{-1}r^{-1}\|=\nu^{-2}$.  By scaling, the
cutoff operator with $\mathbf{1}_{h_\nu\ge1}$ is unitarily
equivalent to the cutoff at $\varepsilon>0$.  Letting
$\varepsilon\to0$ and using strong convergence, 
we arrive at~\eqref{eq:spectrum_aux}
by a standard spectral convergence results
(see, \eg, \cite[Thm.~VIII.24]{RS1}).
\end{Remark}

Now we are in a position to present 
the matrix representation of~$B_\nu^\pm$.
\begin{Proposition}\label{prop:matrix2}
Let $\nu>0$. In the orthonormal basis $(e_{\nu,\ell})_{\ell\in\mathbb{Z}}$ 
of~$L^2$ from Theorem~\ref{Thm.matrix}, the operator $B^\pm_\nu$ decomposes as follows: 
Denoting 
$
  b_{\ell\ell'} := 
  \langle e_{\nu,\ell}, B_\nu^\pm e_{\nu,\ell'} \rangle
$,
\begin{align*}
(b_{\ell\ell'})_{\ell,\ell'\in\mathbb{Z}} = \begin{pmatrix}
\ddots & \vdots & \vdots & \vdots & \vdots & \vdots & \reflectbox{$\ddots$} \\
\dots & b_{-2,-2} & 0 & b_{-2,0} & 0 & 0 & \dots \\
\dots & 0 & b_{-1,-1} & b_{-1,0} & 0 & 0 & \dots \\
\dots & b_{0,-2} & b_{0,-1} & b_{0,0} & b_{0,1} & b_{0,2} & \dots \\
\dots & 0 & 0 & b_{1,0} & b_{1,1} & 0 & \dots \\
\dots & 0 & 0 & b_{2,0} & 0 & b_{2,2} & \dots \\
\reflectbox{$\ddots$} & \vdots & \vdots & \vdots & \vdots & \vdots & \ddots
\end{pmatrix}
,
\end{align*}
where for $\ell,\ell'\ge1$,
\begin{align*}
& b_{0,0} = \pm i\frac{\pi}{4\nu} , \\
& b_{-\ell,\ell'} = b_{\ell',-\ell} = 0, \\
& b_{-\ell,-\ell'} = -\frac{\delta_{\ell\ell'}}{4\ell(\nu+\ell)}, \\
& b_{0,-\ell} = b_{-\ell,0} = \frac{(-1)^{\ell}}{4\ell(\nu+\ell)}\sqrt{1+\frac{2\ell}{\nu}}, \\
& b_{0,\ell} = b_{\ell,0} = \frac{(-1)^{\ell+1}}{4\ell(\nu+\ell)}\sqrt{1+\frac{2\ell}{\nu}} ,\\
& b_{\ell,\ell'} = \frac{\delta_{\ell\ell'}}{4\ell(\nu+\ell)}.
\end{align*}
\end{Proposition}
\begin{Remark}
Proposition \ref{prop:matrix2} implies by direct computation that, for all $\nu>0$, the operator $B_\nu^\pm$ is Hilbert--Schmidt with norm
\begin{align*}
\big\|B_\nu^\pm\big\|_{\mathrm{HS}}^2&=\sum_{\ell\in\mathbb{Z}}\big\|B_\nu^\pm e_{\nu,\ell}\big\|_{L^2}^2\\
&=\frac{\pi^2}{16\nu^2}+\sum_{\ell\ge1} \frac{1+\frac{2\ell}{\nu}}{8\ell^2(\nu+\ell)^2}+\sum_{\ell\ge1} \frac{1+\frac{\ell}{\nu}}{4\ell^2(\nu+\ell)^2} = \frac{\pi^2}{8\nu^2} - \frac{\psi^{(1)}(\nu+1)}{8\nu^2}-\frac{\psi(\nu+1)+\gamma}{4\nu^3},
\end{align*}
where as before, $\psi^{(1)}$ is the trigamma function, $\psi$ is the digamma function and $\gamma$ is Euler's constant. Combined with Remark \ref{rk:ess-spectrum}, this in turn implies that
\begin{align*}
\sigma_{\mathrm{ess}}\big(A_\nu^\pm\big)=[0,\nu^{-2}],
\end{align*}
for all $\nu>0$.
\end{Remark}

Proposition~\ref{prop:matrix2} will be proven below, 
after the proof of Theorem~\ref{Thm.matrix}. 
In order to deduce Theorem~\ref{thm:asymptotics} 
from the former,
we first identify a matrix $C^\pm$, independent of $\nu$, such that $\| B_\nu^\pm - \nu^{-1}C^\pm\|_{L^2\to L^2}$ is negligible compared to $\nu^{-1}$ as $\nu\to\infty$. 
Let~$C^\pm$ be the operator in~$\sii$,
whose matrix $(c_{\ell,\ell'})_{\ell,\ell'\in\mathbb{Z}}$
with respect to the orthonormal basis $(e_{\nu,\ell})_{\ell\in\mathbb{Z}}$ 
reads as follows:
\begin{align*}
& c_{0,0} = \pm i\frac{\pi}{4} , \\
& c_{-\ell,\ell'} = c_{\ell',-\ell} = 0, \\
& c_{-\ell,-\ell'} = -\frac{\delta_{\ell\ell'}}{4\ell}, \\
& c_{0,-\ell} = c_{-\ell,0} = \frac{(-1)^{\ell}}{4\ell}, \\
& c_{0,\ell} = c_{\ell,0} = \frac{(-1)^{\ell+1}}{4\ell} ,\\
& c_{\ell,\ell'} = \frac{\delta_{\ell\ell'}}{4\ell}.
\end{align*}
\begin{Lemma}\label{lm:asymptotics}
Let $\nu>0$. One has
\begin{align*}
\big\|B_\nu^\pm-\nu^{-1}C^\pm\big\|^2_{\mathrm{HS}}\le\frac{3\psi^{(1)}(\nu+1)}{8\nu^2}.
\end{align*}
\end{Lemma}
\begin{proof}
A direct computation gives
\begin{align*}
&\big\|B_\nu^\pm-\nu^{-1}C^\pm\big\|_{\mathrm{HS}}^2=\big\|(B_\nu^\pm-\nu^{-1}C^\pm)e_{\nu,0}\big\|_{L^2}^2+\sum_{\ell\in\mathbb{Z}\setminus\{0\}}\big\|(B_\nu^\pm-\nu^{-1}C^\pm)e_{\nu,\ell}\big\|_{L^2}^2\\
&=2\sum_{\ell\ge1}\Big(\frac{1}{4\ell(\nu+\ell)}-\frac{1}{4\nu\ell}\Big)^2+4\sum_{\ell\ge1}\Big(\frac{(1+\frac{2\ell}{\nu})^{\frac12}}{4\ell(\nu+\ell)}-\frac{1}{4\nu\ell}\Big)^2\\
&=\frac{1}{8\nu^2}\sum_{\ell\ge1}\frac{1}{(\nu+\ell)^2}+\frac{1}{4\nu^2}\sum_{\ell\ge1}\frac{1}{(\nu+\ell)^2}\frac{(\frac{\ell}{\nu})^2}{((1+\frac{2\ell}{\nu})^\frac12+1+\frac{\ell}{\nu})^2}\\
&\le\frac{3}{8\nu^2}\sum_{\ell\ge1}\frac{1}{(\nu+\ell)^2},
\end{align*}
which proves the result.
\end{proof}
\begin{Lemma}\label{lm:asymptotics2}
We have
\begin{align*}
\big\| C^\pm \big\|_{L^2\to L^2}=1.
\end{align*}
\end{Lemma}
\begin{proof}
First, we observe that the definition of the operator $C^\pm$ reduces to
\begin{align*}
& c_{0,0}=\pm i\frac{\pi}{4} , \qquad  c_{k,k'}=\frac{\delta_{kk'}}{4k}, \qquad c_{0,k}=c_{k,0}=\frac{(-1)^{k+1}}{4k}, 
\end{align*}
for $k,k'\in\mathbb{Z}\setminus\{0\}$. We then introduce the unitary operator 
\begin{align*}
W: L^2((0,\infty)) \ni e_{\nu,k} \mapsto \tilde e_k \in L^2((0,2\pi)), 
\qquad \text{where}\qquad 
\forall k\in\mathbb{Z}, \quad \tilde e_k(t) := (-1)^k(2\pi)^{-\frac12} e^{ikt}.
\end{align*}
In the orthonormal basis $(\tilde e_k)_{k\in\mathbb{Z}}$ of $L^2((0,2\pi))$, the operator 
$ \tilde C := WCW^*
\cong (\tilde c_{\ell,\ell'})_{\ell,\ell'\in\mathbb{Z}}
$ 
decomposes as
\begin{align*}
& \tilde c_{0,0}=\pm i\frac{\pi}{4} , \qquad \tilde c_{k,k'}=\frac{\delta_{kk'}}{4k}, \qquad \tilde c_{0,k}=\tilde c_{k,0}=-\frac{1}{4k}.
\end{align*}
Moreover, its integral kernel is
\begin{align*}
\tilde C^\pm(t,s)&=\frac{1}{2\pi}\sum_{j,k\in\mathbb{Z}}\tilde c_{k,j}e^{-ijs}e^{ikt}\\
&= \frac{1}{2\pi}\Big(\tilde c_{0,0}-\frac14\sum_{k\in\mathbb{Z}\setminus\{0\}}\frac{e^{-iks}}{k}-\frac14\sum_{k\in\mathbb{Z}\setminus\{0\}}\frac{e^{ikt}}{k}+\frac14\sum_{k\in\mathbb{Z}\setminus\{0\}}\frac{e^{ik(t-s)}}{k}\Big).
\end{align*}
Let us consider, for instance, the~$+$ case. Then
$\tilde c_{0,0}=i\frac{\pi}{4}$. Using that $\sum_{k\in\mathbb{Z}\setminus\{0\}}\frac{e^{ikt}}{k}=i(\pi-t)$ for $t\in(0,2\pi)$, we obtain
\begin{align*}
\tilde C^+(t,s)
&= \frac{1}{2\pi}\Big(i\frac{\pi}{4}+\frac{i}4(\pi-s)-\frac{i}4(\pi-t)+\frac{i}4\big((\pi-t+s)\mathbf{1}_{t\ge s}-(\pi-s+t)\mathbf{1}_{t<s}\big)\Big)=\frac{i}{4}\mathbf{1}_{t\ge s}.
\end{align*}
Hence we have shown that, for all $\varphi\in L^2((0,2\pi))$, $\tilde C^+$ is given by
\begin{align*}
\forall t\in (0,2\pi), \qquad (\tilde C^+ \varphi)(t)=\frac{i}{4}\int_0^t\varphi(s)\mathrm{d}s.
\end{align*}
The norm of this Volterra operator is standard to compute.  Let
$\varphi\in L^2((0,2\pi))$ be a normalised eigenstate,
associated to a real eigenvalue $\mu$, of
$(\tilde C^+)^*\tilde C^+$.  We have
\begin{align*}
\forall t\in (0,2\pi), \qquad ((\tilde C^+)^* \tilde C^+ \varphi)(t)=\frac{1}{16} \int_t^{2\pi}\int_0^s\varphi(s') \, \mathrm{d}s'\mathrm{d}s=\mu f(t).
\end{align*}
Hence $\varphi$ is a solution to $-16\mu \varphi''=\varphi$ with the boundary conditions $\varphi'(0)=0$, $\varphi(2\pi)=0$. This implies that the eigenvalues of $(\tilde C^+)^* \tilde C^+$ are $(2n+1)^{-2}$, 
with $n \in \Nat$, 
and therefore $\|\tilde C^+\|_{L^2\to L^2}=1$. Since $W$ is unitary, this also shows that $\|C^+\|_{L^2\to L^2}=1$. One can argue in the same way for $C^-$.
\end{proof}

Now we can conclude the proof of Theorem \ref{thm:asymptotics}.
\begin{proof}[Proof of Theorem \ref{thm:asymptotics}]
It suffices to write, using Lemma \ref{lm:asymptotics2},
\begin{align*}
\Big| \big\| A_\nu^\pm \big\|_{L^2\to L^2} - \frac{1}{\nu} \Big| &= \Big| \big\| A_\nu^\pm \big\|_{L^2\to L^2} - \big\| \nu^{-1}C^\pm\big\|_{L^2\to L^2} \Big| \\
&\le \big\| A_\nu^\pm -B_\nu^\pm + B_\nu^\pm - \nu^{-1} C^\pm \big\|_{L^2\to L^2} \\
&\le \big\|r^{-1}h_\nu^{-1}\mathbf{1}_{h_\nu\ge1}r^{-1}\big\|_{L^2\to L^2} + \big\| B_\nu^\pm - \nu^{-1} C^\pm \big\|_{\mathrm{HS}}, 
\end{align*}
and then apply \eqref{eq:rk4} and Lemma \ref{lm:asymptotics}.
\end{proof}

\subsection{The Hankel transform}
In the remainder of this section, 
we establish Theorem~\ref{Thm.matrix} and Proposition~\ref{prop:matrix2}.
The main ingredient in the proof is
the usual Hankel transform of order~$\nu$:
\begin{align*}
\mathcal{H}_\nu[\varphi](k)
:=\int_0^\infty \sqrt{kr} J_\nu(kr) \varphi(r) \, \mathrm{d} r,
\end{align*}
which extends to a unitary operator on~$L^2$ 
such that $\mathcal{H}_\nu^{-1}=\mathcal{H}_\nu$. It can further be extended to a suitable space of distributions (similarly as the Fourier transform). 

It is well known that the Hankel transform of order~$\nu$ diagonalises 
the Bessel operator~$h_\nu$, in the sense that for any bounded measurable function $f:\mathbb{R}\to\mathbb{R}$ and $\varphi,\psi\in L^2$,
\begin{align}\label{eq:Hankel5a}
\langle\varphi,f(h_\nu)\psi\rangle_{\sii} 
= \langle \mathcal{H}_\nu[\varphi] , f(k^2) \mathcal{H}_\nu[\psi] \rangle_{\sii}.
\end{align}
In particular, for all $\varepsilon>0$ and $\varphi,\psi\in \sii$ 
such that $r^{-1}\varphi,r^{-1}\psi\in L^2$,
\begin{align}\label{eq:Hankel50}
\big\langle\varphi,r^{-1} (h_\nu - (1\pm i\varepsilon))^{-1} r^{-1} \psi\big\rangle_{L^2} =  \int_0^\infty 
\frac{1}{k^2-(1\pm i\varepsilon)} 
\, \overline{\mathcal{H}_\nu[r^{-1}\varphi](k)} 
\, \mathcal{H}_\nu[r^{-1}\psi](k) \, \mathrm{d}k.
\end{align}

Let
\begin{align*}
\hat{\mathcal{Q}}_\nu:=
\big\{ \hat\varphi\in L^2_{\mathrm{loc}}((0,\infty)) \, : \, \hat\varphi' \in L^2((0,\infty)) \big\} = \dot{H}^1_0((0,\infty)),
\end{align*}
equipped with the scalar product defined by
\begin{align*}
\langle\hat\varphi,\hat\psi\rangle_{\hat{\mathcal{Q}}_\nu}:=\langle\hat\varphi,h_\nu\hat\psi\rangle_{L^2},
\end{align*}
for any $\hat\varphi,\hat\psi\in\hat{\mathcal{Q}}_\nu$. 
We introduce the unitary operator
\begin{align*}
\mathcal{U}_\nu : L^2 \ni \psi \mapsto \mathcal{H}_\nu[r^{-1}\psi] \in \hat{\mathcal{Q}}_\nu,
\end{align*}
whose inverse is given by
$\mathcal{U}_\nu^{-1}\hat\psi=r\mathcal{H}_\nu[\hat\psi]$,
for any $\hat\psi\in\hat{\mathcal{Q}}_\nu$.  The facts that
$\mathcal{U}_\nu$ is well defined and unitary follow from the
properties of the Hankel transform, since
\begin{align*}
\langle\varphi,\psi\rangle_{L^2}=\langle r^{-1}\varphi,r^2r^{-1}\psi\rangle_{L^2}=\langle\mathcal{H}_\nu[r^{-1}\varphi],h_\nu\mathcal{H}_\nu[r^{-1}\psi]\rangle_{L^2}=\langle\mathcal{H}_\nu[r^{-1}\varphi],\mathcal{H}_\nu[r^{-1}\psi]\rangle_{\hat{\mathcal{Q}}_\nu}.
\end{align*}

We shall use the classical identity
\begin{equation}\label{eq:hankel_r32Jnu}
  \mathcal{H}_\nu[r^{-3/2}J_\nu](k)
  =
  \frac{1}{2\nu}
  \left(
    k^{\nu+1/2}\mathbf{1}_{(0,1]}(k)
    +
    k^{-\nu+1/2}\mathbf{1}_{(1,\infty)}(k)
  \right),
\end{equation}
which follows from \cite[6.573.2]{GR}.

Let 
\begin{align*}
\hat e_{\nu,0}(k):=(2\nu)^{-\frac12} \big[ k^{\nu+\frac12}\mathbf{1}_{(0,1]}(k)+k^{-\nu+\frac12}\mathbf{1}_{(1,\infty)}(k)\big].
\end{align*}
Note that $\hat e_{\nu,0}=(2\nu)^{\frac12}\mathcal{H}_\nu[r^{-\frac32}J_\nu]=(2\nu)^{\frac12}\mathcal{U}_\nu(r^{-\frac12}J_\nu)$, see \eqref{eq:hankel_r32Jnu}. In particular $\hat e_{\nu,0}\in\hat{\mathcal{Q}}_\nu$ and $\|\hat e_{\nu,0}\|^2_{\hat{\mathcal{Q}}_\nu}=2\nu\|r^{-\frac12}J_\nu\|_{L^2}^2=1$. We have the following orthogonal decomposition of $\hat{\mathcal{Q}}_\nu$ (equipped with $\langle\cdot,\cdot\rangle_{\hat{\mathcal{Q}}_\nu}$):
\begin{align*}
\hat{\mathcal{Q}}_\nu = \hat{\mathcal{Q}}_\nu^0 \oplus \mathrm{Span}(\hat e_{\nu,0}) \oplus \hat{\mathcal{Q}}_\nu^\infty,
\end{align*}
where
\begin{align*}
\hat{\mathcal{Q}}_\nu^0:=\{\hat\varphi\in\hat{\mathcal{Q}}_\nu\, : \, \mathrm{supp}(\hat\varphi)\subset(0,1] \text{ and }\hat\varphi(1)=0\}, \quad \hat{\mathcal{Q}}_\nu^\infty:=\{\hat\varphi\in\hat{\mathcal{Q}}_\nu\, : \, \mathrm{supp}(\hat\varphi)\subset[1,\infty) \text{ and }\hat\varphi(1)=0\}.
\end{align*}
Observe that any $\hat\varphi\in\hat{\mathcal{Q}}_\nu$ decomposes in this direct sum as
\begin{align*}
\hat\varphi(k) = \big( \hat\varphi(k) - \hat\varphi(1)k^{\nu+\frac12} \big)\mathbf{1}_{(0,1)}(k) + \hat\varphi(1)\hat e_{\nu,0}(k) + \big( \hat\varphi(k) - \hat\varphi(1)k^{-\nu+\frac12} \big)\mathbf{1}_{(1,\infty)}(k).
\end{align*}

Using the unitary operator $\mathcal{U}_\nu$, the orthogonal decomposition (with respect to the~$L^2$ scalar product) in~$L^2$ becomes 
 \begin{equation*}
    \sii 
    =
    E_\nu^0
    \oplus
    \operatorname{Span}(e_{\nu,0})
    \oplus
    E_\nu^\infty 
  \end{equation*} 
with
%
$E_\nu^\#:=\mathcal{U}_\nu^{-1}\hat{\mathcal{Q}}_\nu^\#$
%
and 
%
$
e_{\nu,0}:=\mathcal{U}_\nu^{-1}[\hat e_{\nu,0}]=(2\nu)^{\frac12}p_\nu=(2\nu)^{\frac12}r^{-\frac12}J_\nu.
$
From the proof of Proposition~\ref{prop:main_Bessel_Imaginary},
we know
that in this orthogonal decomposition we have
\begin{align*}
\Im A_\nu^\pm 
= \pm \frac{\pi}{4\nu} ( 0 \oplus 1 \oplus 0 ).
\end{align*}
Moreover, Proposition~\ref{prop:main_Bessel_Real-2}
yields the real part of the coefficient~$a_{00}$ 
through the identity~\eqref{real.identity}.

\subsection{The orthonormal basis at low energies}
Our first task is to construct an orthogonal basis of $E_\nu^0$
where the matrix elements of $\Re A_\nu^\pm$ can be computed
explicitly.

For all integer $\ell\ge1$, we introduce the notation 
\begin{align}\label{eq:def_varphinu}
\hat \varphi_{\nu,\ell}(k):= k^{\nu+\frac12}(1-k^2) P_{\ell-1}^{(\nu,1)}(1-2k^2)\mathbf{1}_{(0,1)}(k),
\end{align}
where $P_{\ell-1}^{(\nu,1)}$ is the usual Jacobi polynomial. We will use the following lemma.
\begin{Lemma}\label{lm:hnu-hatPhil0}
Let $\nu>0$. Then, for all $\ell\ge1$, $\hat \varphi_{\nu,\ell} \in \hat{\mathcal{Q}}_\nu^0$. Moreover, for all $0<k<1$,
\begin{align}\label{eq:hnu-hatPhil0}
(h_\nu\hat \varphi_{\nu,\ell})(k)=\frac{4\ell(\ell+\nu)}{1-k^2} \hat \varphi_{\nu,\ell}(k).
\end{align}
\end{Lemma}
\begin{proof}
The definition \eqref{eq:def_varphinu} shows that
$\hat \varphi_{\nu,\ell}\in \hat{\mathcal{Q}}_\nu^0$.
We now prove \eqref{eq:hnu-hatPhil0}.  Let $0<k<1$.
To shorten the notation, set $P:=P_{\ell-1}^{(\nu,1)}$.
We compute, setting $x=1-2k^{2}$,
\begin{align}
& \Big(-\frac{\der^2}{\der k^2}+\frac{\nu^2-\frac14}{k^2} \Big) k^{\nu+\frac12}(1-k^{2})P(1-2k^{2}) \notag \\
&= \Big(\frac{1-x}{2}\Big)^{-\frac12} \Big(-16 \, \Big(\frac{1-x}{2}\Big)^2\frac{\der^2}{\der x^2}+12 \, \Big(\frac{1-x}{2}\Big)\frac{\der}{\der x} + (\nu^2-\mbox{$\frac14$}) \Big) \Big(\frac{1-x}{2}\Big)^{\frac\nu2-\frac14} \Big(\frac{1+x}{2}\Big) P(x) \notag\\
&=2^{-\frac\nu2+\frac34} (1-x)^{-\frac12} \Big(-2(1-x)^2\frac{\der^2}{\der x^2}+3(1-x)\frac{\der}{\der x} + \frac12 (\nu^2-\mbox{$\frac14$}) \Big) (1-x)^{\frac\nu2-\frac14} (1+x) P(x).\label{eq:hnu-hatPsil1}
\end{align}
Now,
\begin{align*}
&\frac{\der}{\der x} \big[ (1-x)^{\frac\nu2-\frac14} (1+x) P(x) \big] =(1-x)^{\frac\nu2-\frac14}(1+x)P'+\big( (1-x)^{\frac\nu2-\frac14} 
- (\mbox{$\frac\nu2-\frac14$})(1-x)^{\frac\nu2-\frac54}(1+x) \big)P
\end{align*}
and
\begin{align*}
\frac{\der^2}{\der x^2} \big[ (1-x)^{\frac\nu2-\frac14} (1+x) P(x) \big] 
&=(1-x)^{\frac\nu2-\frac14}(1+x)P''+\big(2(1-x)^{\frac\nu2-\frac14}
- \mbox{$(\nu-\frac12)$} (1-x)^{\frac\nu2-\frac54}(1+x)\big)P'\\
&\quad+\big(
\mbox{$\frac14(\nu-\frac12)(\nu-\frac52)$}
(1-x)^{\frac\nu2-\frac94}(1+x)
-\mbox{$(\nu-\frac12)$} (1-x)^{\frac\nu2-\frac54}\big)P.
\end{align*}
Thus
\begin{align*}
& \Big(-2(1-x)^2\frac{\der^2}{\der x^2}
+3(1-x)\frac{\der}{\der x} 
+ \mbox{$\frac12 (\nu^2-\frac14)$} \Big) (1-x)^{\frac\nu2-\frac14} (1+x) P(x) \\
&=-2(1-x)^{\frac\nu2+\frac74}(1+x)P''\\
&\quad+\big[-4(1-x)^{\frac\nu2+\frac74}+2(\nu-\mbox{$\frac12$})(1-x)^{\frac\nu2+\frac34}(1+x)+3(1-x)^{\frac\nu2+\frac34}(1+x)\big]P'\\
&\quad+\big[-\mbox{$\frac12(\nu-\frac12)$}
(\nu-\mbox{$\frac52$})(1-x)^{\frac\nu2+\frac74}(1+x)+2(\nu-\mbox{$\frac12$})(1-x)^{\frac\nu2+\frac34} \\
&\qquad\quad+3(1-x)^{\frac\nu2+\frac34} 
- 3(\mbox{$\frac\nu2-\frac14$})(1-x)^{\frac\nu2-\frac14}(1+x)
+\mbox{$\frac12(\nu^2-\frac14)$}(1-x)^{\frac\nu2-\frac14}(1+x)\big]P \\
&=-2(1-x)^{\frac\nu2+\frac74}(1+x)P''\\
&\quad+\big[-4(1-x)^{\frac\nu2+\frac74}+2(\nu+1)(1-x)^{\frac\nu2+\frac34}(1+x)\big]P'\\
&\quad+2(\nu+1)(1-x)^{\frac\nu2+\frac34}P \\
&=-2(1-x)^{\frac\nu2+\frac34} \Big[ (1-x^2) P'' + \big[2 (1-x)  -(\nu+1)(1+x)\big]P'-(\nu+1)P\Big] \\
&=-2(1-x)^{\frac\nu2+\frac34} \Big[ (1-x^2) P'' + \big[1-\nu - ( 3+\nu )x\big]P'-(\nu+1)P\Big] \\
&=-2(1-x)^{\frac\nu2+\frac34} \big[ -(\ell-1)(\ell+\nu+1) - (\nu+1) \big] P,
\end{align*} 
where in the last equality we have used that the Jacobi polynomial $P$ satisfies the differential equation $(1-x^2)y''+(1-\nu-(\nu+3)x)y'=-(\ell-1)(\ell+\nu+1)y$. Inserting the last equality into \eqref{eq:hnu-hatPsil1}, and using that $(\ell-1)(\ell+\nu+1) + (\nu+1)=\ell(\ell+\nu)$ we obtain
\begin{align*}
\big[h_\nu \hat \varphi_{\nu,\ell}\big](k) &=2^{-\frac\nu2+\frac74} (1-x)^{\frac\nu2+\frac14} \ell(\ell+\nu)  P(x)\\
&=4\ell(\ell+\nu)  k^{\nu+\frac12}  P(1-2k^{2})\\
&=4 (1-k^{2})^{-1}\ell(\ell+\nu) \hat \varphi_{\nu,\ell}(k).
\end{align*}
This concludes the proof of the lemma.
\end{proof}

As above we use the notations $p_\nu(r)=r^{-\frac12}J_\nu(r)$ and we set $\tilde p_\nu:=(2\nu)^{\frac12} \, p_\nu=e_{\nu,0}$.
\begin{Lemma}
Let $\nu>0$. The sequence  $(e_{\nu,-\ell})_{\ell\ge1}$ defined by 
\begin{align*}
\forall \ell\ge 1, \qquad e_{\nu,-\ell} := \tilde p_{\nu+2\ell}, 
\end{align*}
is an orthonormal basis of $E_\nu^0$.
\end{Lemma}
\begin{Remark}
We mention that the fact that $(\tilde p_{\nu+2\ell})_{\ell\ge1}$ is  an orthonormal family in~$L^2$ follows, \eg, from~\cite[6.574.2]{GR}. 
Likewise, the fact that $r^{-1}\tilde p_{\nu+2\ell}\in L^2$ 
for any $\ell\ge1$ follows from \cite[6.574.2]{GR} 
(since $\nu>0$ and $\ell\ge1$). These two well-known facts also follow from the following proof. The main point of the lemma is that $(\tilde p_{\nu+2\ell})_{\ell\ge1}$ is a \emph{basis} of $E_\nu^0$.
\end{Remark}
\begin{proof}
We compute $\mathcal{U}_\nu[\tilde p_{\nu+2\ell}]=\mathcal{H}_\nu[r^{-1}\tilde p_{\nu+2\ell}]$. To this end, we write (see \cite[6.574.1]{GR})
\begin{align*}
\mathcal{H}_\nu[r^{-1}p_{\nu+2\ell}](k)=k^{\frac12}\int_0^\infty r^{-1}J_\nu(kr)J_{\nu+2\ell}(r)\mathrm{d}r&=\frac{k^{\nu+\frac12}\Gamma(\nu+\ell)}{2\Gamma(\ell+1)\Gamma(\nu+1)} \mathbf{F}(\nu+\ell,-\ell;\nu+1;k^2)\mathbf{1}_{(0,1)}(k) \\
&=\frac{k^{\nu+\frac12}}{2(\nu+\ell)} P_\ell^{(\nu,-1)}(1-2k^2)\mathbf{1}_{(0,1)}(k),
\end{align*}
where $\mathbf{F}$ and $P_\ell^{(\nu,-1)}$ stand for the usual Hypergeometric function and Jacobi polynomial, respectively. In particular we see that $\mathrm{supp}\, \mathcal{H}_\nu[r^{-1}p_{\nu+2\ell}] \subset (0,1]$ and $\mathcal{H}_\nu[r^{-1}p_{\nu+2\ell}](1)=0$ for any $\ell\ge1$.

Now, using the relation $P_\ell^{(\nu,-1)}(x)=\frac{\nu+\ell}{\ell}\frac{x+1}{2}P_{\ell-1}^{(\nu,1)}(x)$ (see \cite[(4.22.2)]{Szego}), we obtain
\begin{align}\label{eq:Hankel_pnu}
\mathcal{H}_\nu[r^{-1}p_{\nu+2\ell}](k)=\frac{k^{\nu+\frac12}(1-k^2)}{2\ell} P_{\ell-1}^{(\nu,1)}(1-2k^2)\mathbf{1}_{(0,1)}(k)=\frac{\hat \varphi_{\nu,\ell}(k)}{2\ell}.
\end{align}
For any $\ell,\ell'\ge1$, we can then write
\begin{align}
\big\langle\mathcal{U}_\nu p_{\nu+2\ell},\mathcal{U}_\nu p_{\nu+2\ell'}\big\rangle_{\mathcal{Q}_\nu}&=\big\langle\mathcal{H}_\nu[r^{-1}p_{\nu+2\ell}],\mathcal{H}_\nu[r^{-1}p_{\nu+2\ell}']\big\rangle_{\mathcal{Q}_\nu}\notag\\
&=\frac{1}{4\ell\ell'}\int_0^1 \hat \varphi_{\nu,\ell}(k) [h_\nu\hat \varphi_{\nu,\ell'}](k)\mathrm{d}k\notag\\
&=\frac{\ell'+\nu}{\ell}\int_0^1 \hat \varphi_{\nu,\ell}(k) (k^2-1)^{-1}\hat \varphi_{\nu,\ell'}(k)\mathrm{d}k\notag\\
&=\frac{\ell'+\nu}{\ell}\int_0^1 k^{2\nu+1} (k^2-1) P_{\ell-1}^{(\nu,1)}(1-2k^2) P_{\ell'-1}^{(\nu,1)}(1-2k^2) \mathrm{d}k\notag\\
&=\frac{1}{2^{\nu+3}} \frac{\ell'+\nu}{\ell}\int_{-1}^1 (1-x)^\nu(1+x)P_{\ell-1}^{(\nu,1)}(x) P_{\ell'-1}^{(\nu,1)}(x) \mathrm{d}x\notag\\
&=\frac{\delta_{\ell\ell'}}{2(2\ell+\nu)}, \label{eq:orthog_pnu}
\end{align}
where the last equality follows from the well-known orthogonality relation for Jacobi polynomials (see, \eg, \cite[7.391.1]{GR}). Inserting the normalisation $\tilde p_{\nu+2\ell}=\sqrt{2(\nu+2\ell)} p_{\nu+2\ell}$, we deduce that $(\mathcal{U}_\nu \tilde p_{\nu+2\ell})_{\ell\ge1}$ is an orthonormal family in $\hat{\mathcal{Q}}_\nu^0$, and therefore that $(\tilde p_{\nu+2\ell})_{\ell\ge1}$ is an orthonormal family in $E_\nu^0$.

Since in addition it is well known that the family 
of Jacobi polynomials $(P_{\ell-1}^{(\nu,1)})_{\ell\ge1}$ forms an orthogonal basis of $L^2((-1,1))$ for the weight $(1-x)^\nu(1+x)$, we can conclude from the previous computation that $(\tilde p_{\nu+2\ell})_{\ell\ge1}$ is indeed an orthonormal basis of $E_\nu^0$.
\end{proof}
\begin{Lemma}\label{lm:psiell-Re-psiell'}
Let $\nu>0$. Let $\ell,\ell'\ge1$. Then
\begin{align*}
\big\langle e_{\nu,-\ell}, (\Re A_\nu^\pm ) e_{\nu,-\ell'} \big\rangle_{L^2} = -\frac{\delta_{\ell\ell'}}{4\ell(\nu+\ell)}.
\end{align*}
\end{Lemma}
\begin{proof}
For all $\varepsilon>0$, using \eqref{eq:Hankel5a} and \eqref{eq:Hankel_pnu}, we have
\begin{align*}
&\big\langle p_{\nu+2\ell},r^{-1} \Re (h_\nu - (1\pm i\varepsilon))^{-1}  r^{-1} p_{\nu+2\ell'} \big\rangle_{L^2} \\
&= \Re \int_0^1 \frac{1}{k^2-(1\pm i\varepsilon)} \overline{\mathcal{H}_\nu[r^{-1} p_{\nu+2\ell}](k)} \mathcal{H}_\nu[r^{-1} p_{\nu+2\ell'}](k) \mathrm{d}k\\
&= \frac{1}{4\ell\ell'} \Re \int_0^1 \frac{k^{2\nu+1}(1-k^2)^2}{k^2-(1\pm i\varepsilon)} P_{\ell-1}^{(\nu,1)}(1-2k^2) P_{\ell'-1}^{(\nu,1)}(1-2k^2) \mathrm{d}k.
\end{align*}
Letting $\varepsilon\to0$ (using Lebesgue's dominated convergence theorem) 
yields
\begin{align*}
&\lim_{\varepsilon\to0}\big\langle p_{\nu+2\ell},r^{-1} \Re (h_\nu - (1\pm i\varepsilon))^{-1} r^{-1} p_{\nu+2\ell'} \big\rangle_{L^2} \\
&= -\frac{1}{4\ell\ell'} \int_0^1 k^{2\nu+1}(1-k^2) P_{\ell-1}^{(\nu,1)}(1-2k^2) P_{\ell'-1}^{(\nu,1)}(1-2k^2) \mathrm{d}k.
\end{align*}
Using \eqref{eq:orthog_pnu}, we then obtain
\begin{align*}
&\lim_{\varepsilon\to0}\big\langle p_{\nu+2\ell},r^{-1} \Re (h_\nu - (1\pm i\varepsilon))^{-1}  r^{-1} p_{\nu+2\ell'} \big\rangle = -\frac{\delta_{\ell\ell'}}{8\ell(2\ell+\nu)(\ell+\nu)} .
\end{align*}
Inserting the normalisation $e_{\nu,-\ell}=\tilde p_{\nu+2\ell}=\sqrt{2(\nu+2\ell)} p_{\nu+2\ell}$, the result follows.
\end{proof}

We also need to compute the matrix elements of~$\Re A_\nu^\pm$ 
between the basis elements $e_{\nu,-\ell}$, $\ell\ge1$, 
of~$E_\nu^0$ and $e_{\nu,0}$.
\begin{Lemma}\label{lm:orthog-pnu1}
Let $\nu>0$ and $\ell\ge1$. Then
\begin{align*}
\big\langle e_{\nu,\ell}, (\Re A_\nu^\pm ) e_{\nu,0} \big\rangle_{L^2} = \frac{(-1)^\ell}{4\ell(\nu+\ell)}\Big(1+\frac{2\ell}{\nu}\Big)^{\frac12} .
\end{align*}
\end{Lemma}
\begin{proof}
We proceed as in the proof of Lemma \ref{lm:psiell-Re-psiell'}. For all $\varepsilon>0$, using \eqref{eq:Hankel5a} and \eqref{eq:Hankel_pnu}, we have
\begin{align*}
&\big\langle p_{\nu+2\ell},r^{-1} \Re (h_\nu - (1\pm i\varepsilon))^{-1}  r^{-1} p_{\nu} \big\rangle_{L^2} \\
&= \Re \int_0^1 \frac{1}{k^2-(1\pm i\varepsilon)} \overline{\mathcal{H}_\nu[r^{-1} p_{\nu+2\ell}](k)} \mathcal{H}_\nu[r^{-1} p_{\nu}](k) \mathrm{d}k\\
&= \frac{1}{4\nu\ell} \Re \int_0^1 \frac{k^{2\nu+1}(1-k^2)}{k^2-(1\pm i\varepsilon)} P_{\ell-1}^{(\nu,1)}(1-2k^2) \mathrm{d}k.
\end{align*}
Letting $\varepsilon\to0$ 
(using Lebesgue's dominated convergence theorem) yields
\begin{align*}
&\lim_{\varepsilon\to0}\big\langle p_{\nu+2\ell},r^{-1} \Re (h_\nu - (1\pm i\varepsilon))^{-1} r^{-1} p_{\nu} \big\rangle_{L^2} \\
&= -\frac{1}{4\nu\ell} \int_0^1 k^{2\nu+1} P_{\ell-1}^{(\nu,1)}(1-2k^2)  \mathrm{d}k\\
&= -\frac{2^{-\nu-2}}{4\nu\ell} \int_{-1}^1 (1-x)^\nu P_{\ell-1}^{(\nu,1)}(x)  \mathrm{d}x=\frac{(-1)^\ell}{8\nu\ell(\nu+\ell)},
\end{align*}
where the last equality follows from the change of variables $x\mapsto -x$, the identity $P_{\ell-1}^{(\nu,1)}(-x)=(-1)^{\ell-1}P_{\ell-1}^{(1,\nu)}(x)$ and \cite[7.391.4]{GR}. Inserting the normalisation $e_{\nu,-\ell} = \tilde p_{\nu+2\ell}=\sqrt{2(\nu+2\ell)} p_{\nu+2\ell}$, the result follows.
\end{proof}

\subsection{The orthonormal basis at high energies}
Now we turn to the construction of an orthonormal basis 
to decompose~$A_\nu^\pm$ in the high energy region~$E_\nu^\infty$. 
In this region, it is not clear to us how to define a suitable orthonormal basis with an explicit expression, as in the case of the basis $(\tilde p_{\nu+2\ell})_{\ell\ge1}$ of $E_\nu^0$. Nevertheless the Hankel transforms of the basis elements (multiplied by $r^{-1}$) that we construct remain explicit, with a form close to that of the previous section. First, for all integer $\ell\ge1$ and $\nu>0$, we set
\begin{align}\label{eq:def_hatPsil}
\hat\psi_{\nu,\ell}(k):=k^{-\nu+\frac12}(1-k^{-2})P_{\ell-1}^{(\nu,1)}(1-2k^{-2})\mathbf{1}_{(1,\infty)}(k).
\end{align}
\begin{Lemma}\label{lm:hnu-hatPsil}
Let $\nu>0$. Then, for all $\ell\ge1$, $\hat\psi_{\nu,\ell}\in\hat{\mathcal{Q}}_\nu^\infty$. Moreover, for all $k>1$,
\begin{align}\label{eq:hnu-hatPsil}
(h_\nu\hat\psi_{\nu,\ell})(k)=\frac{4\ell(\ell+\nu)}{ k^{2} (k^2-1)} \hat\psi_{\nu,\ell}(k).
\end{align}
\end{Lemma}
\begin{proof}
Let $\ell\ge0$. Since $\nu>0$, the fact that $\hat\psi_{\nu,\ell}\in \hat{\mathcal{Q}}_\nu^\infty$ directly follows from the definition \eqref{eq:def_hatPsil}.

To prove \eqref{eq:hnu-hatPsil}, we argue as in the proof of Lemma \ref{lm:hnu-hatPhil0}. Let $k>1$. To shorten the notations, 
set $P:=P_{\ell-1}^{(\nu,1)}$. We compute, setting $x=1-2k^{-2}$,
\begin{align}
&\big[h_\nu \hat\psi_{\nu,\ell}\big](k) 
= \Big(-\frac{\der^2}{\der k^2}+\frac{\nu^2
-\frac14}{k^2} \Big) k^{-\nu+\frac12}(1-k^{-2})P(1-2k^{-2}) \notag \\
&= \Big(-16\,\Big(\frac{1-x}{2}\Big)^3\frac{\der^2}{\der x^2}+12\,\Big(\frac{1-x}{2}\Big)^2\frac{d}{dx} + \Big(\frac{1-x}{2}\Big) (\nu^2-\mbox{$\frac14$}) \Big) \Big(\frac{1-x}{2}\Big)^{\frac\nu2-\frac14} \Big(\frac{1+x}{2}\Big) P(x) \notag\\
&=2^{-\frac\nu2-\frac34} (1-x) \Big(-2(1-x)^2\frac{\der^2}{\der x^2}+3(1-x)\frac{d}{dx} + \frac12 (\nu^2-\mbox{$\frac14$}) \Big) (1-x)^{\frac\nu2-\frac14} (1+x) P(x) .\label{eq:hnu-hatPsil1-a}
\end{align}
Now, from the computation in the proof of Lemma \ref{lm:hnu-hatPhil0}, we have 
\begin{align*}
& \Big(-2(1-x)^2\frac{\der^2}{\der x^2}+3(1-x)\frac{\der}{\der x} 
+ \frac12 (\nu^2-\mbox{$\frac14$}) \Big) (1-x)^{\frac\nu2-\frac14} (1+x) P(x) \\
&=2\ell(\ell+\nu)(1-x)^{\frac\nu2+\frac34}  P(x).
\end{align*} 
Inserting this equality into \eqref{eq:hnu-hatPsil1-a}, we obtain
\begin{align*}
\big[h_\nu \hat\psi_{\nu,\ell}\big](k) &=2^{-\frac\nu2+\frac14}\ell(\ell+\nu) (1-x)^{\frac\nu2+\frac74} P(x)\\
&=4 \ell(\ell+\nu) k^{-\nu-\frac72}  P(1-2k^{-2})\\
&=4\ell(\ell+\nu) k^{-4} (1-k^{-2})^{-1}k^{-\nu+\frac12} (1-k^{-2})  P(1-2k^{-2}) \\
&=4\ell(\ell+\nu) k^{-2} (k^2-1)^{-1}\hat\psi_{\nu,\ell}(k).
\end{align*}
This proves the lemma.
\end{proof}
Now we set
\begin{align*}
\psi_{\nu,\ell}:=\mathcal{U}_\nu^{-1}\hat\psi_{\nu,\ell}.
\end{align*}
\begin{Lemma}\label{lm:orthog-psi}
Let $\nu>0$. Then for all $\ell\ge1$, $\psi_{\nu,\ell}\in E_\nu^\infty$ and for all $\ell,\ell'\ge1$,
\begin{align*}
\langle \psi_{\nu,\ell'} , \psi_{\nu,\ell}\rangle_{L^2} = \frac{2\ell^2}{2\ell+\nu} \delta_{\ell\ell'}.
\end{align*}
\end{Lemma}
\begin{proof}
By the definition of $\mathcal{U}_\nu$ and $\psi_{\nu,\ell}$,
we have $\psi_{\nu,\ell}\in E_\nu^\infty$ for all
$\ell\ge1$, and
\begin{align}
\langle \psi_{\nu,\ell'} , \psi_{\nu,\ell}\rangle_{L^2} &= \big\langle \hat\psi_{\nu,\ell'} , \hat\psi_{\nu,\ell} \big\rangle_{\hat{\mathcal{Q}}_\nu}=\int_1^\infty \hat\psi_{\nu,\ell'}(k) \big[h_\nu\hat\psi_{\nu,\ell}\big](k)\, \mathrm{d}k.
\end{align}
Therefore, using \eqref{eq:hnu-hatPsil} 
and the change of variables $x=1-2k^{-2}$,
we get
\begin{align}
\langle \psi_{\nu,\ell'} , \psi_{\nu,\ell}\rangle_{L^2} &=4\ell(\ell+\nu) \int_1^\infty k^{-2} (k^2-1)^{-1} \hat\psi_{\nu,\ell'}(k) \hat\psi_{\nu,\ell}(k) \mathrm{d}k \notag\\
&=4\ell(\ell+\nu) 2^{-\nu-3} \int_{-1}^1 (1-x)^\nu (1+x) P^{(\nu,1)}_{\ell'-1}(x) P^{(\nu,1)}_{\ell-1}(x) \mathrm{d}x \notag\\
&=2\ell(\ell+\nu) \frac{\ell}{(2\ell+\nu)(\ell+\nu)} \delta_{\ell\ell'}. \label{eq:orthog-psi}
\end{align}
Here the last equality follows from the well-known identity for Jacobi polynomials (see, \eg, \cite[7.391.1]{GR}).%
\end{proof}
Setting for $\ell\ge1$,
\begin{align*}
e_{\nu,\ell}:=\Big(\frac{2\ell+\nu}{2\ell^2}\Big)^{\frac12}\psi_{\nu,\ell},
\end{align*}
the previous lemma provides an orthonormal basis for $E_\nu^\infty$.
\begin{Lemma}
Let $\nu>0$. 
The sequence
$(e_{\nu,\ell})_{\ell\ge1}$ is an orthonormal basis of $E_\nu^\infty$.
\end{Lemma}
\begin{proof}
Lemma \ref{lm:orthog-psi} shows that $(e_{\nu,\ell})_{\ell\ge1}$ is an orthonormal family in $E_\nu^\infty$. The fact that it is an orthonormal basis follows from the proof of Lemma \ref{lm:orthog-psi} (see \eqref{eq:orthog-psi}) and the well-known property that the family of Jacobi polynomials $(P_{\ell-1}^{(\nu,1)})_{\ell\ge1}$ is an orthonormal basis of $L^2((-1,1))$ for the weight $(1-x)^\nu(1+x)$.
\end{proof}
\
Next we have to compute the matrix elements of~$\Re A_\nu^\pm$ 
in the orthonormal basis of~$E_\nu^\infty$ we just constructed. 
Unfortunately, it does not yield a diagonal matrix 
as in the case of~$E_\nu^0$.
\begin{Lemma}\label{lm:matrix-psi}
Let $\nu>0$. Then for all $\ell,\ell'\ge1$,
\begin{align*}
\big\langle e_{\nu,\ell'},(\Re A_\nu^\pm) e_{\nu,\ell} \big\rangle_{L^2}=\frac{1}{4\ell\ell'}\frac{n!\,\Gamma(\nu+m)}{(m-1)!\,\Gamma(\nu+n+1)}\Big(1+\frac{2\ell}{\nu}\Big)^{\frac12}\Big(1+\frac{2\ell'}{\nu}\Big)^{\frac12},
\end{align*}
where $n:=\max(\ell,\ell')$ and $m:=\min(\ell,\ell')$.
\end{Lemma}
\begin{proof}
Using Lebesgue's dominated convergence theorem and the change of variables $x=1-2k^{-2}$ as before, we have
\begin{align*}
&\big\langle \psi_{\nu,\ell'},r^{-1} \Re (h_\nu - (1\pm i0))^{-1} r^{-1} \psi_{\nu,\ell} \big\rangle_{L^2}\\
&=\lim_{\varepsilon\to0} \int_1^\infty\frac{\hat\psi_{\nu,\ell'}(k)\hat\psi_{\nu,\ell} (k)}{k^2-(1\pm i\varepsilon)} \mathrm{d}k\\
&=\int_1^\infty k^{-2\nu-3}(k^2-1) P_{\ell'-1}^{(\nu,1)}(1-2k^{-2}) P_{\ell-1}^{(\nu,1)}(1-2k^{-2}) \, \mathrm{d}k\\
&=2^{-\nu-2} \int_{-1}^1 (1-x)^{\nu-1} (1+x) P^{(\nu,1)}_{\ell'-1}(x) P^{(\nu,1)}_{\ell-1}(x) \, \mathrm{d}x.
\end{align*}
To evaluate the integral, let $n:=\max(\ell,\ell')$, $m=\min(\ell,\ell')$. If $m=1$, the result follows from \cite[7.391.4]{GR}. If $m\ge2$, write $P^{(\nu,1)}_{m}(x)=P^{(\nu,1)}_{m-1}(1)+(1-x)R_{m-2}(x)$, where $R_{m-2}$ is some polynomial of degree $m-2$. Since $m-2<n-1$, the orthogonality of Jacobi polynomials implies that
\begin{align*}
\int_{-1}^1 (1-x)^{\nu-1} (1+x) P^{(\nu,1)}_{n-1}(x) P^{(\nu,1)}_{m}(x) \, \mathrm{d}x &= \int_{-1}^1 (1-x)^{\nu-1} (1+x) P^{(\nu,1)}_{n-1}(x) P^{(\nu,1)}_{m-1}(1) \, \mathrm{d}x \\
&=\binom{m-1+\nu}{m-1} \int_{-1}^1 (1-x)^{\nu-1} (1+x) P^{(\nu,1)}_{n-1}(x) \, \mathrm{d}x.
\end{align*}
Applying \cite[7.391.4]{GR}, we obtain
\begin{align*}
\big\langle \psi_{\nu,\ell'} ,r^{-1} \Re (h_\nu - (1\pm i0))^{-1} r^{-1}\psi_{\nu,\ell}\big\rangle_{L^2}&=\frac12\binom{m-1+\nu}{m-1}\frac{\Gamma(\nu)\Gamma(n+1)}{\Gamma(\nu+n+1)}=\frac{1}{2\nu}\frac{n!\,\Gamma(\nu+m)}{(m-1)!\,\Gamma(\nu+n+1)}.
\end{align*}
We conclude from the normalisation $e_{\nu,\ell}=\frac{\sqrt{2\ell+\nu}}{\ell\sqrt{2}}\psi_{\nu,\ell}$.
\end{proof}
As in the case of $E_\nu^0$, we finally need to compute the matrix elements of $\Re A_\nu^\pm$ between the basis elements $e_{\nu,\ell}$ of $E_\nu^\infty$ and $e_{\nu,0}=\tilde p_\nu$.
\begin{Lemma}\label{lm:orthog-pnu2}
Let $\nu>0$ and $\ell\ge1$. Then
\begin{align*}
\big\langle e_{\nu,\ell} ,(\Re A_\nu^\pm) e_{\nu,0} \big\rangle_{L^2} = \frac{1}{4\ell(\nu+\ell)}\Big((-1)^{\ell}-\frac{\ell!}{(\nu)_\ell}\Big)\Big(1+\frac{2\ell}{\nu}\Big)^{\frac12} .
\end{align*}
\end{Lemma}
\begin{proof}
We proceed again as in the proof of Lemma \ref{lm:psiell-Re-psiell'}. For all $\varepsilon>0$, using \eqref{eq:Hankel5a} and \eqref{eq:Hankel_pnu}, we have
\begin{align*}
&\big\langle \psi_{\nu,\ell},r^{-1} \Re (h_\nu - (1\pm i\varepsilon))^{-1}  r^{-1} p_{\nu} \big\rangle_{L^2} \\
&= \Re \int_1^\infty \frac{1}{k^2-(1\pm i\varepsilon)} \hat\psi_\nu^{(\ell)}(k) \mathcal{H}_\nu[r^{-1} p_{\nu}](k) \, \mathrm{d}k\\
&= \frac{1}{2\nu} \Re \int_1^\infty \frac{k^{-2\nu+1}(1-k^{-2})}{k^2-(1\pm i\varepsilon)} P_{\ell-1}^{(\nu,1)}(1-2k^{-2}) \, \mathrm{d}k.
\end{align*}
Letting $\varepsilon\to0$ 
(using Lebesgue's dominated convergence theorem) yields
\begin{align*}
&\lim_{\varepsilon\to0} \big\langle \psi_{\nu,\ell},r^{-1} \Re (h_\nu - (1\pm i\varepsilon))^{-1}  r^{-1} p_{\nu} \big\rangle_{L^2} \\
&= \frac{1}{2\nu} \int_1^\infty k^{-2\nu-1} P_{\ell-1}^{(\nu,1)}(1-2k^{-2})  \,\mathrm{d}k\\
&= \frac{2^{-\nu-2}}{\nu} \int_{-1}^1 (1-x)^{\nu-1} P_{\ell-1}^{(\nu,1)}(x)  \,\mathrm{d}x =-\frac{1}{4\nu(\nu+\ell)}\Big((-1)^{\ell}-\frac{\ell!}{(\nu)_\ell}\Big),
\end{align*}
where the last equality follows from \cite[7.391.2]{GR}. 
From the normalisations $e_{\nu,0} = \tilde p_{\nu}=\sqrt{2\nu} p_{\nu}$ and $e_{\nu,\ell}=\frac{\sqrt{2\ell+\nu}}{\ell\sqrt{2}}\psi_{\nu,\ell}$, the result follows.
\end{proof}

\subsection{Proof of Theorem~\ref{Thm.matrix} 
and Proposition~\ref{prop:matrix2}}\label{Sec.Proproofs}
Theorem~\ref{Thm.matrix} is a consequence of the results obtained in the previous sections.
\begin{proof}[Proof of Theorem~\ref{Thm.matrix}]
It suffices to combine Propositions \ref{prop:main_Bessel_Imaginary} and \ref{prop:main_Bessel_Real-2}, together with Lemmas \ref{lm:psiell-Re-psiell'}, \ref{lm:orthog-pnu1}, \ref{lm:matrix-psi} and \ref{lm:orthog-pnu2}.
\end{proof}

Now we turn to the proof of Proposition \ref{prop:matrix2} 
for the operator~$B_\nu^\pm$ defined in~\eqref{operator.B}.
Clearly, if $\ell\le-1$ or $\ell'\le-1$, then
\begin{align*}
\big\langle e_{\nu,\ell},B_\nu^\pm e_{\nu,\ell}\big\rangle_{L^2}=\big\langle e_{\nu,\ell},A_\nu^\pm e_{\nu,\ell}\big\rangle_{L^2}.
\end{align*}
In what follows we compute $\langle e_{\nu,\ell},B_\nu^\pm e_{\nu,\ell}\rangle_{L^2}$ for $\ell\ge0$ and $\ell'\ge0$.
\begin{Lemma}\label{lm:matrix-psi-modified}
Let $\nu>0$. Then for all $\ell,\ell'\ge1$,
\begin{align*}
\big\langle e_{\nu,\ell},(\Re B_\nu^\pm)e_{\nu,\ell}\big\rangle_{L^2}=\frac{\delta_{\ell\ell'}}{4\ell(\nu+\ell)}.
\end{align*}
\end{Lemma}
\begin{proof}
Arguing as in the proof of the Lemma \ref{lm:orthog-pnu2} and using \eqref{eq:hnu-hatPsil} from Lemma \ref{lm:hnu-hatPsil}, we have
\begin{align*}
&\big\langle\psi_{\nu,\ell'},r^{-1} \Re (h_\nu - (1\pm i0))^{-1}h_\nu^{-1}\mathbf{1}_{h_\nu\ge1} r^{-1}\psi_{\nu,\ell}\big\rangle_{L^2}\\
&=\lim_{\varepsilon\to0} \int_1^\infty\frac{\hat\psi_{\nu,\ell'}(k)\hat\psi_{\nu,\ell}(k)}{k^2(k^2-(1\pm i\varepsilon))} \, \mathrm{d}k\\
&=\frac{1}{4\ell(\nu+\ell)} \int_1^\infty \hat\psi_{\nu,\ell'}(k)\big[h_\nu\hat\psi_{\nu,\ell}\big](k) \, \mathrm{d}k = \frac{1}{4\ell(\nu+\ell)}\big\langle\psi_{\nu,\ell'},\psi_{\nu,\ell}\big\rangle_{L^2}.
\end{align*}
The result then follows from Lemma \ref{lm:orthog-psi} and the fact that $e_{\nu,\ell}=\frac{\sqrt{2\ell+\nu}}{\ell\sqrt{2}}\psi_{\nu,\ell}$.
\end{proof}

\begin{Lemma}\label{lm:orthog-pnu3}
Let $\nu>0$ and $\ell\ge1$. Then
\begin{align*}
\big\langle e_{\nu,\ell}, (\Re B_\nu^\pm) e_{\nu,0} \big\rangle_{L^2} = \frac{(-1)^{\ell+1}}{4\ell(\nu+\ell)} \Big(1+\frac{2\ell}{\nu}\Big)^{\frac12} .
\end{align*}
\end{Lemma}
\begin{proof}
The same computations as in the Lemma \ref{lm:orthog-pnu2} give
\begin{align*}
&\lim_{\varepsilon\to0} \big\langle \psi_{\nu,\ell},r^{-1} \Re (h_\nu - (1\pm i\varepsilon))^{-1} h_\nu^{-1} \mathbf{1}_{h_\nu\ge1}  r^{-1} p_{\nu} \big\rangle \\
&= \frac{1}{2\nu} \int_1^\infty k^{-2\nu-3} P_{\ell-1}^{(\nu,1)}(1-2k^{-2})  \, \mathrm{d}k\\
&= \frac{2^{-\nu-3}}{\nu} \int_{-1}^1 (1-x)^{\nu} P_{\ell-1}^{(\nu,1)}(x) \, \der x =\frac{(-1)^{\ell+1}}{4\nu(\nu+\ell)} ,
\end{align*}
where the last equality follows from \cite[7.391.2]{GR}.  From the normalisations $e_{\nu,0}=\sqrt{2\nu} p_{\nu}$ and $e_{\nu,\ell}=\frac{\sqrt{2\ell+\nu}}{\ell\sqrt{2}}\psi_{\nu,\ell}$, the result follows.
\end{proof}
\begin{Lemma}\label{lm:Bnu-e_0}
Let $\nu>0$. Then
\begin{align*}
\big\langle e_{\nu,0} ,r^{-1} h_\nu^{-1}\mathbf{1}_{h_\nu\ge1}  r^{-1} e_{\nu,0} \big\rangle_{L^2} = \frac{1}{4\nu^2}.
\end{align*}
\end{Lemma}
\begin{proof}
Recall that
\begin{align*}
\mathcal{U}_\nu[e_{\nu,0}]=\mathcal{H}_\nu[r^{-1}e_{\nu,0}] (k)=(2\nu)^{-\frac12} \big[ k^{\nu+\frac12}\mathbf{1}_{(0,1]}(k)+k^{-\nu+\frac12}\mathbf{1}_{(1,\infty)}(k)\big].
\end{align*}
This implies
\begin{align*}
\big\langle e_{\nu,0} ,r^{-1} h_\nu^{-1}\mathbf{1}_{h_\nu\ge1}  r^{-1} e_{\nu,0} \big\rangle = \frac{1}{2\nu}\int_1^\infty k^{-2\nu-1} \mathrm{d}k = \frac{1}{4\nu^2}.
\end{align*}
\end{proof}
We can then conclude the proof of Proposition \ref{prop:matrix2}:
\begin{proof}[Proof of Proposition \ref{prop:matrix2}]
It suffices to use the definition \eqref{operator.B} of $B_\nu^\pm$ together with Theorem~\ref{Thm.matrix} and Lemmata~\ref{lm:matrix-psi-modified}, \ref{lm:orthog-pnu3} and \ref{lm:Bnu-e_0}.
\end{proof}

\section{One-channel estimates}\label{sec:Bessel}
%
Based on the matrix representation of~$A_\nu^\pm$
given by Theorem~\ref{Thm.matrix},
we have already proved parts~(i) and~(iii) of Theorem~\ref{Thm.new}, 
as well as Theorem~\ref{Thm.lower} giving a quantitative lower bound 
to the boundary value of the weighted resolvent of~$h_\nu$.
Now we establish additional uniform bounds,
in particular parts~(ii) and~(iv) of Theorem~\ref{Thm.new}.
The main result of this section is the latter,
which we restate here as follows.
\begin{Theorem}\label{thm:main_Bessel}
Let $\nu>0$. One has
\begin{equation}\label{eq:main_Bessel}
  \|A_\nu^\pm\|_{L^2\to L^2}
  =\sup_{z\in\mathbb{C}\setminus[0,\infty)}\left\|r^{-1} (h_\nu - z)^{-1} r^{-1}\right\|_{L^2 \to L^2} 
  \le 
  \mathcal{C}(\nu),
\end{equation}
where $\mathcal{C}(\nu)$ is defined in~\eqref{eq:def-Cnu}. 
\end{Theorem}

The theorem is complemented by uniform lower bounds
established in Section~\ref{Sec.lower}.

\subsection{Basic facts}\label{subsec:basic}
Recall the definition of~$h_\nu$ with $\nu > 0$
given in~\eqref{Bessel}.
It is clear from the Hardy inequality
\begin{equation}\label{Hardy-Bessel}
\forall \psi\in C_0^\infty((0,\infty)), \quad 
\Big\|\frac{\psi}{r}\Big\|_{\sii}^2 
\le 4 \, \|\partial_r\psi\|_{\sii}^2.
\end{equation}
that~$h_\nu$ is non-negative. 
Moreover, its form domain is given by
%
$\mathcal{Q}(h_\nu)=H^1_0((0,\infty))$.
%

If $\nu\ge1$, then the operator domain of~$h_\nu$ is given by
\begin{equation*}
\mathcal{D}(h_\nu)=\Big\{\psi\in L^2 \,:\,
-\partial^2_r\psi+\frac{\nu^2-\frac{1}{4}}{r^2}\psi \in L^2 \Big\},
\end{equation*}
where $-\partial^2_r\psi+(\nu^2-\frac14)r^{-2}\psi$ should be understood in the sense of distributions.
Actually, $h_\nu$ is essentially selfadjoint 
on $C_0^\infty((0,\infty))$ provided that $\nu\ge1$.

If $0<\nu<1$, then
\begin{equation*}
\mathcal{D}(h_\nu)=\Big\{\psi\in L^2 \,:\,-\partial^2_r\psi+\frac{\nu^2-\frac{1}{4}}{r^2}\psi \in L^2 \text{ and } \lim_{r\to0} \mathcal{W}(\psi,\psi_0)(r)=0 \Big\},
\end{equation*}
where $\mathcal{W}$ stands for the Wronskian and $\psi_0(r):=r^{\frac12+\nu}$.
In this case ($0<\nu<1$), 
the operator initially defined on $C_0^\infty((0,\infty))$ 
has several self-adjoint extensions. 

In any case, \ie\ for all $\nu>0$, 
an element of $\mathcal{D}(h_\nu)$ is a continuous function that behaves like $cr^{\frac12+\nu}$ near $0$, for some $c\in\Real$. 
Moreover, for all $\nu>0$, 
the spectrum of $h_\nu$ is $\sigma(h_\nu)=[0,\infty)$ and $h_\nu$ is homogeneous, in the sense that defining, for $\tau>0$, the unitary operator $u_\tau$ in~$L^2$, by
\begin{equation}\label{eq:def-utau}
  (u_\tau \psi)(r) := \tau^{1/2} \, \psi(\tau r) 
  \,,
\end{equation}
we have
\begin{equation}\label{eq:utau-hnu}
u_\tau h_\nu u_\tau^{-1}=\tau^{-2}h_\nu.
\end{equation}
For more details on the family of operators $h_\nu$, we refer to \cite{Derezinski-Richard_2017} and references therein.

\subsection{Preliminaries}\label{subsec:prelim}
The Hardy inequality~\eqref{Hardy-Bessel} 
has the following direct consequences.
\begin{Lemma}\label{Lem:csq_Hardy}
Let $\nu>0$. Then, in the sense of quadratic forms in~$\sii$
defined on $H^1((0,\infty))$,
\begin{equation}\label{eq:csq_Hardy-Bessel0}
r^{-2}\le\nu^{-2} \, h_\nu
\qquad\mbox{and}\qquad 
-\partial_r^2\le\max(1,(4\nu^2)^{-1}) \, h_\nu.
\end{equation}
\end{Lemma}
\begin{proof}
Note that the Hardy inequality~\eqref{Hardy-Bessel} gives $-\partial_r^2-\frac1{4r^2}\ge0$. To prove the first inequality in \eqref{eq:csq_Hardy-Bessel0}, we then write
\begin{equation*}
\frac{1}{r^2}\le\frac{1}{\nu^2}\Big(\frac{\nu^2}{r^2}-\partial_{r}^2-\frac{1}{4r^2}\Big)=\frac{h_\nu}{\nu^2}.
\end{equation*}
To prove the second inequality in \eqref{eq:csq_Hardy-Bessel0}, we distinguish two cases. If $\nu^2-\frac14\ge0$, the inequality $-\partial_r^2\le h_\nu$ is obvious. If $\nu^2-\frac14<0$, it suffices to write, using $r^{-2}\le\nu^{-2}h_\nu$,
\begin{equation*}
-\partial_r^2=h_\nu+\frac{\frac14-\nu^2}{r^2}\le h_\nu+\frac{\frac14-\nu^2}{\nu^2}h_\nu=(4\nu^2)^{-1}h_\nu.
\end{equation*}
This proves the lemma.
\end{proof}

The Hardy inequality~\eqref{Hardy-Bessel} gives the following easy bound for the weighted Bessel resolvent away from the spectrum $[0,\infty)$.
\begin{Lemma}\label{lm:bounded-op-Bessel}
Let $\nu>0$ and $z\in\mathbb{C}\setminus[0,\infty)$. Then
\begin{equation*}
\left\| r^{-1} (h_\nu - z)^{-1} r^{-1}  \right\|_{\sii \to \sii} \le \nu^{-2} \Big( 1 + \frac{|z|}{\dist(z,[0,\infty))} \Big).
 \end{equation*}
\end{Lemma}
\begin{proof}
Estimate~\eqref{eq:csq_Hardy-Bessel0} implies that
\begin{equation}\label{eq:csq_Hardy-Bessel}
\big\|r^{-1}h_\nu^{-\frac12}\big\|_{\sii \to \sii}=\big\|h_\nu^{-\frac12}r^{-1}\big\|_{\sii \to \sii}\le\nu^{-1}.
\end{equation}
Together with the estimate
\begin{equation*}
\left\| h_\nu (h_\nu - z)^{-1} \right\|_{\sii \to \sii} \le 1 + |z|\left\|(h_\nu - z)^{-1}\right\|_{\sii \to \sii} \le 1 + \frac{|z|}{\dist(z,[0,\infty))},
\end{equation*}
this proves the lemma.
\end{proof}

For $\Re z\le0$, the same Hardy argument gives the following
sharper uniform estimate.
\begin{Lemma}\label{lem:Hardy-Bessel}
Let $\nu>0$ and $z\in\mathbb{C}$ be such that $\Re z\le0$, $z\neq0$. Then
\begin{equation}
\left\| r^{-1} (h_\nu - z)^{-1} r^{-1}  \right\|_{\sii \to \sii} \le \nu^{-2}.
 \end{equation}
\end{Lemma}
\begin{proof}
Using \eqref{eq:csq_Hardy-Bessel0} and $\Re z\le0$, we have, for all $\psi\in\mathcal{D}(h_\nu)$,
\begin{equation*} 
\begin{aligned} 
  \nu^2\Big\| \frac{\psi}{r}\Big \|_{\sii}^2
  \leq
  \langle\psi,h_\nu\psi\rangle_{\sii}  - \Re z \, \|\psi\|_{\sii}^2
  = \Re\langle\psi,(h_\nu-z)\psi\rangle_{\sii} 
  \leq \Big\|\frac{\psi}{r}\Big\|_{\sii} \, 
  \big\|r (h_\nu-z)\psi\big\|_{\sii}
  \,,
\end{aligned}
\end{equation*}
which yields
\begin{equation*}
\Big\|\frac{\psi}{r}\Big\|_{\sii}\le\nu^{-2}\big\|r (h_\nu-z)\psi\big\|_{\sii}.
\end{equation*}
Let $\varphi\in C_0^\infty((0,\infty))$. Applying the previous inequality to $\psi=(h_\nu-z)^{-1}r^{-1}\varphi$ (which belongs to $\mathcal{D}(h_\nu)$), we obtain 
\begin{equation*}
\left\| r^{-1} (h_\nu - z)^{-1} r^{-1} \varphi \right\|_{\sii}
\le \nu^{-2} \|\varphi\|_{L^2}.
\end{equation*}
Since $C_0^\infty((0,\infty))$ is dense in~$\sii $, this concludes the proof.
\end{proof}

The next proposition is the standard maximum principle reduction
from the resolvent set to the boundary rays of a sector around the
positive half--line.  To keep the one channel argument
self-contained, we spell out the proof after recording the two
scalar input lemmata.

\begin{Proposition}\label{prop:epsilon-Bessel}
Let $\nu>0$. The map
\begin{equation*}
(0,\infty)\ni\varepsilon\mapsto\sup_{|\Im z|=\varepsilon \Re z}\left\| r^{-1} (h_\nu - z)^{-1} r^{-1} \right\|_{\sii \to \sii}
\end{equation*}
is non-increasing and we have
\begin{align*}
&\sup_{z\in\mathbb{C}\setminus[0,\infty)}
\left\|
  r^{-1} (h_\nu - z)^{-1} r^{-1} \right\|_{\sii \to \sii} =\lim_{\varepsilon\to0} \sup_{|\Im z|=\varepsilon \Re z}\left\|
  r^{-1} (h_\nu - z)^{-1} r^{-1}
  \right\|_{\sii \to \sii}.
\end{align*}
\end{Proposition}
For $\varepsilon>0$, we write
\begin{equation*}
  S_\varepsilon
  :=
  \mathbb{C}\setminus
  \{z\in\mathbb{C}\,:\,|\Im z|\le\varepsilon\Re z\}.
\end{equation*}
The following two lemmata employ an
explicit expression of the integral kernel of $(h_\nu-z)^{-1}$.

\begin{Lemma}\label{Lem:decay-Bessel}  
Let $\nu>0$ and let $\phi,\psi\in C_0^\infty((0,\infty))$. Then 
\begin{equation*}
  \lim_{\stackrel[z \not\in [0,\infty)]{}{|z| \to \infty}}
  \left\langle
  \phi, (h_\nu - z)^{-1} \psi 
  \right\rangle_{L^2} = 0.
\end{equation*}
\end{Lemma}
\begin{proof}
For $z\in\mathbb{C}\setminus[0,\infty)$, let $(r,R)\mapsto k_{\nu,z}(r,R)$ denote the integral kernel of 
$(h_\nu- z)^{-1}$. Then $k_{\nu,z}$ satisfies, for all $r,R>0$,
\begin{equation}\label{eq:bound_kernel}
|k_{\nu,z}(r,R)|\le \frac{C_\nu}{|z|^{1/2}}\min(1,|z|^{\frac14}r^{\frac12})\min(1,|z|^{\frac14}R^{\frac12}),
\end{equation}
for some constant $C_\nu>0$, see, \eg, \cite[Prop. 4.3]{Derezinski-Richard_2017}. One deduces from this that, for any $\phi,\psi\in\mathrm{C}_0^\infty((0,\infty))$ and $|z|$ large enough,
\begin{equation*}
\left|\left\langle\phi,(h_\nu-z)^{-1}\psi\right\rangle\right|\le\frac{C_\nu}{|z|^{1/2}}\|\phi\|_{L^\infty}\|\psi\|_{L^\infty},
\end{equation*}
which proves the lemma.
\end{proof}
\begin{Lemma}\label{Lem:zero-Bessel}  
Let $\varepsilon>0$, $\nu>0$, and let
$\phi,\psi\in C_0^\infty((0,\infty))$. Then 
\begin{equation*}
  \lim_{\stackrel[z \in S_\varepsilon]{}{z \to 0}}
  \left\langle
  \phi, (h_\nu - z)^{-1} \psi 
  \right\rangle_{L^2} =   \left\langle
  \phi, h_\nu^{-1} \psi 
  \right\rangle_{L^2}.
\end{equation*}
\end{Lemma}
\begin{proof}
Estimate~\eqref{eq:bound_kernel} supplies the required domination,
so Lebesgue's dominated convergence theorem applies.
\end{proof}
\begin{proof}[Proof of Proposition \ref{prop:epsilon-Bessel}]
Since $S_{\varepsilon_2}\subset S_{\varepsilon_1}$ whenever
$0<\varepsilon_1<\varepsilon_2$, the map
\begin{equation*}
  \varepsilon\mapsto
  \sup_{z\in S_\varepsilon}
  \left\|r^{-1}(h_\nu-z)^{-1}r^{-1}\right\|_{\sii\to\sii}
\end{equation*}
is non-increasing.  Also every fixed
$z_0\in\mathbb{C}\setminus[0,\infty)$ belongs to
$S_\varepsilon$ for all sufficiently small $\varepsilon>0$.
Hence
\begin{equation*}
  \sup_{z\in\mathbb{C}\setminus[0,\infty)}
  \left\|r^{-1}(h_\nu-z)^{-1}r^{-1}\right\|_{\sii\to\sii}
  =
  \lim_{\varepsilon\to0}
  \sup_{z\in S_\varepsilon}
  \left\|r^{-1}(h_\nu-z)^{-1}r^{-1}\right\|_{\sii\to\sii}.
\end{equation*}

It remains to replace $S_\varepsilon$ by its boundary.  Let
$\phi,\psi\in C_0^\infty((0,\infty))$.  The function
\begin{equation*}
  z\mapsto
  \left\langle
    \phi,r^{-1}(h_\nu-z)^{-1}r^{-1}\psi
  \right\rangle_{L^2}
\end{equation*}
is analytic on $S_\varepsilon$ and continuous up to the boundary;
continuity at the vertex $z=0$ follows from
Lemma~\ref{Lem:zero-Bessel}, applied to $r^{-1}\phi$ and
$r^{-1}\psi$.  By Lemma~\ref{Lem:decay-Bessel}, applied to the
same two test functions, this scalar function tends to zero as
$|z|\to\infty$.  The maximum modulus principle on truncated
sectors, followed by the limits at zero and infinity, gives
\begin{equation*}
  \sup_{z\in S_\varepsilon}
  \left|
  \left\langle
    \phi,r^{-1}(h_\nu-z)^{-1}r^{-1}\psi
  \right\rangle_{L^2}
  \right|
  =
  \sup_{|\Im z|=\varepsilon\Re z}
  \left|
  \left\langle
    \phi,r^{-1}(h_\nu-z)^{-1}r^{-1}\psi
  \right\rangle_{L^2}
  \right|.
\end{equation*}
The density of $C_0^\infty((0,\infty))$ in~$L^2$
then gives the corresponding equality for the operator norms, and
the proposition follows.
\end{proof}
\begin{Remark}\label{rk:scaling-hnu}
It follows from scaling invariance that the norm $\| r^{-1} (h_\nu - z)^{-1} r^{-1}\|_{\sii\to\sii}$ is constant along the ray $e^{i \arg(z)} (0,\infty)$. More precisely, recalling the definition \eqref{eq:def-utau} of the unitary transform $u_\tau$ on $\sii$, 
we have, by \eqref{eq:utau-hnu},
\begin{align*} 
  \left\| r^{-1} (h_\nu -  z)^{-1} r^{-1}\right\|_{\sii\to \sii} &= \left\|
  U_\tau
  r^{-1} (h_\nu - z)^{-1} r^{-1}
  U_\tau^{-1}
  \right\|_{\sii\to \sii} 
 \notag  \\
  &=
  \left\|
  r^{-1} (h_\nu - \tau^2 z)^{-1} r^{-1}
  \right\|_{\sii \to \sii}
  \,,
\end{align*}
for any $\tau>0$. Together with Proposition \ref{prop:epsilon-Bessel}, this justifies the equality in the statement of Theorem \ref{thm:main_Bessel}, namely
\begin{equation*}
\|A_\nu^\pm\|_{L^2\to L^2}=\sup_{z\in\mathbb{C}\setminus[0,\infty)}\left\|r^{-1} (h_\nu - z)^{-1} r^{-1}\right\|_{L^2 \to L^2}.
\end{equation*}
\end{Remark}

\subsection{Proof of Theorem~\ref{thm:main_Bessel}}
Now we are in a position to prove Theorem~\ref{thm:main_Bessel}.
The proof is based on the method of multipliers. An identity obtained from this method in the context of the Bessel operators reads as follows.
\begin{Lemma}\label{lm:multipliers-Bessel}
Let $\nu>0$ and $z\in\mathbb{C}\setminus[0,\infty)$ be such that $\Re z>0$. Let $f\in C_0^\infty((0,\infty))$ and let $\psi :=(h_\nu-z)^{-1}f$. Then
\begin{align}
&\|\partial_r\psi^-\|_{L^2}^2+\frac{|z_2|}{\sqrt{z_1}}\int_0^\infty r\big|\partial_r\psi^-\big|^2\notag\\
&=\Re\int_0^\infty\Big( \frac{|z_2|}{\sqrt{z_1}}rf^-\overline{\psi^-} - \frac{|z_2|}{\sqrt{z_1}}(\nu^2-\mbox{$\frac14$})\frac{|\psi^-|^2}{r}  
+ 2 rf^- \overline{\partial_r\psi^-}-2\frac{\nu^2-\frac14}{r}\psi^-\overline{\partial_r\psi^-}\Big), \label{eq:identity_multipliers-Bessel}
\end{align}
where 
\begin{equation*}
z_1:=\Re z, \quad z_2:= \Im z, \quad \psi^-(r):=e^{-i\sqrt{z_1}\mathrm{sgn}(z_2)r}\psi(r),
\end{equation*}
and likewise for $f^-$.
\end{Lemma}
\begin{proof}
The identity can be deduced from~\cite[Rem.~3.1]{CK2}.
For completeness, and because the half--line setting involves both
a singular Bessel potential and the boundary point $r=0$, we give
the short derivation in the present context.  

The weak formulation of $(h_\nu-z)\psi = f$ reads
\begin{equation}\label{i0B}
  \int_0^\infty \bar\phi' \psi' = z \int_0^\infty \bar\phi \psi
  + \int_0^\infty \bar\phi \tilde{f}
  \qquad\mbox{with}\qquad  
  \tilde{f} := f - \frac{\nu^2-\frac{1}{4}}{r^2} \, \psi
  \,,
\end{equation}
where~$\phi$ is any function from the form domain of~$h_\nu$.
The test functions used below are admissible because~$\psi$ is
smooth and exponentially decaying at infinity (this follows, \eg,
from the explicit expression of the integral kernel of
$(h_\nu-z)^{-1}$), while
$\psi(r)\sim c r^{\frac12+\nu}$ as $r\to0$ for some constant~$c$.
This asymptotic also eliminates all boundary terms at $r=0$.

Choosing $\phi:=r\psi'+\frac12\psi$ in~\eqref{i0B}, taking twice
the real part and integrating by parts gives
\begin{equation*}
  2\int_0^\infty |\psi'|^2
  =
  -2z_2\,\Im\int_0^\infty \psi\,r\overline{\psi'}
  +2\Re\int_0^\infty \tilde f\,r\overline{\psi'}
  +\Re\int_0^\infty \tilde f\,\overline\psi .
\end{equation*}
Choosing $\phi:=\psi$ and taking the real part gives
\begin{equation*}
  \int_0^\infty |\psi'|^2
  =
  z_1\int_0^\infty |\psi|^2
  +\Re\int_0^\infty \tilde f\,\overline\psi .
\end{equation*}
Finally, choosing $\phi:=r\psi$ and taking respectively the real
and imaginary parts gives
\begin{equation*}
  \int_0^\infty r|\psi'|^2
  =
  z_1\int_0^\infty r|\psi|^2
  +\Re\int_0^\infty r\tilde f\,\overline\psi
\end{equation*}
and
\begin{equation*}
  \Im\int_0^\infty \overline\psi\,\psi'
  =
  z_2\int_0^\infty r|\psi|^2
  +\Im\int_0^\infty r\overline\psi\,\tilde f .
\end{equation*}
No singular term appears on the left-hand side of the third
identity.

We now take the first identity minus the second one, add
$|z_2|z_1^{-1/2}$ times the third one, and subtract
$2z_1^{1/2}\sgn(z_2)$ times the fourth one.  Substituting
$\tilde f=f-(\nu^2-\frac14)r^{-2}\psi$ and using
\begin{equation*}
  \partial_r\psi^-
  =
  e^{-i\sqrt{z_1}\sgn(z_2)r}
  \big(\psi'-i\sqrt{z_1}\sgn(z_2)\psi\big),
  \qquad
  f^-
  =
  e^{-i\sqrt{z_1}\sgn(z_2)r}f,
\end{equation*}
we arrive at the desired identity~\eqref{eq:identity_multipliers-Bessel}.
\end{proof}
\begin{proof}[Proof of Theorem \ref{thm:main_Bessel}]
Let $z\in\mathbb{C}\setminus[0,\infty)$. If $z_1=\Re z\le0$ (and $z\neq0$), we know from Lemma~\ref{lem:Hardy-Bessel} that
\begin{equation*}
\left\| r^{-1} (h_\nu - z)^{-1} r^{-1}  \right\|_{\sii \to \sii} \le \nu^{-2}.
 \end{equation*}
Now we assume that $z_1>0$. 
Let~$f$ and~$\psi$ be as in Lemma~\ref{lm:multipliers-Bessel}.
We see that the
second term on the left-hand side of
\eqref{eq:identity_multipliers-Bessel} is non-negative.  The first
two terms on the right-hand side are controlled by rough Cauchy
bounds, since they only affect the auxiliary constants.  When
$\nu^2-\frac14\ge0$, the second term on the right-hand side of
\eqref{eq:identity_multipliers-Bessel} is non-positive and can be
dropped.  More precisely, using the Cauchy--Schwarz inequality and
\begin{equation*}
|z_2| \, \|\psi\|_{\sii}^2
=\Big|\Im\int_0^\infty\overline{\psi} f\Big|\le\Big\|\frac{\psi}{r}\Big\|_{L^2}\|rf\|_{L^2},
\end{equation*}
we obtain
\begin{align}
&\Big | \int_0^\infty\Big( \frac{|z_2|}{\sqrt{z_1}}rf^-\overline{\psi^-} - \frac{|z_2|}{\sqrt{z_1}}(\nu^2-\mbox{$\frac14$})\frac{|\psi^-|^2}{r} \Big) \Big |  \notag\\
&\le \sqrt{\frac{|z_2|}{z_1}}\Big(\|rf\|_{L^2}^{\frac32}\Big\|\frac{\psi}{r}\Big\|^{\frac12}_{L^2}+\Big|\nu^2-\mbox{$\frac14$}\Big| \Big\|\frac{\psi}{r}\Big\|^{\frac32}_{L^2}\|rf\|_{L^2}^{\frac12}\Big)\notag\\
&\le \sqrt{\frac{|z_2|}{z_1}}\Big(\mbox{$\frac34$}\|rf\|^2_{L^2}
+\mbox{$\frac14$}\Big\|\frac{\psi}{r}\Big\|^{2}_{L^2}+\Big|\nu^2-\mbox{$\frac14$}\Big| \Big( \mbox{$\frac34$} \Big\|\frac{\psi}{r}\Big\|^{2}_{L^2}+\frac14\|rf\|^2_{L^2}\Big)\Big)\notag\\
&= \sqrt{\frac{|z_2|}{z_1}}\Big(\Big(\mbox{$\frac34$}+\mbox{$\frac14$}\Big|\nu^2-\mbox{$\frac14$}\Big|\Big)\|rf\|^2_{L^2}+\Big(\mbox{$\frac14$}+\mbox{$\frac34$}\Big|\nu^2-\mbox{$\frac14$}\Big|\Big)\Big\|\frac{\psi}{r}\Big\|^{2}_{L^2}\Big) \notag\\
&\le \sqrt{\frac{|z_2|}{z_1}}\Big(\Big(\mbox{$\frac34$}+\mbox{$\frac14$}\Big|\nu^2-\mbox{$\frac14$}\Big|\Big)\|rf\|^2_{L^2}+\nu^{-2}\Big(\mbox{$\frac14$}+\mbox{$\frac34$}\Big|\nu^2-\mbox{$\frac14$}\Big|\Big)\big\|h_\nu^{\frac12}\psi^-\big\|^{2}_{L^2}\Big), \label{eq:identity_multipliers-Bessel1}
\end{align}
where we used $\|r^{-1}\psi\|_{L^2}=\|r^{-1}\psi^-\|_{L^2}$ and Lemma \ref{Lem:csq_Hardy} in the last inequality. For the last term on the right-hand side of \eqref{eq:identity_multipliers-Bessel}, we integrate by parts, obtaining
\begin{equation}
2\Re\int_0^\infty\frac{\nu^2-\frac14}{r}\psi^-\overline{\partial_r\psi^-}=\int_0^\infty\frac{\nu^2-\frac14}{r}\partial_r|\psi|^2=\Big(\nu^2-\frac14\Big)\int_0^\infty\frac{|\psi|^2}{r^2}. \label{eq:identity_multipliers-Bessel2}
\end{equation}
Inserting \eqref{eq:identity_multipliers-Bessel1} and \eqref{eq:identity_multipliers-Bessel2} into \eqref{eq:identity_multipliers-Bessel}, and estimating the third term on the right-hand side of \eqref{eq:identity_multipliers-Bessel} by the Cauchy--Schwarz inequality, we obtain
\begin{align*}
&\Big(1-\sqrt{\frac{|z_2|}{z_1}}\nu^{-2}\Big(\mbox{$\frac14$}+\mbox{$\frac34$}\Big|\nu^2-\mbox{$\frac14$}\Big|\Big)\Big) \big\|h_\nu^{\frac12}\psi^-\big\|_{L^2}^2\notag\\
&\le\sqrt{\frac{|z_2|}{z_1}}\Big(\Big(\mbox{$\frac34$}+\mbox{$\frac14$}\Big|\nu^2-\mbox{$\frac14$}\Big|\Big)\|rf\|^2_{L^2}+2\|rf\|_{L^2}\|\partial_r\psi^-\|_{L^2}\Big).
\end{align*}
To simplify the notations, we rewrite the previous inequality as
\begin{equation}
\Big(1-\sqrt{\frac{|z_2|}{z_1}}C_{1,\nu}\Big) \big\|h_\nu^{\frac12}\psi^-\big\|_{L^2}^2-2\|rf\|_{L^2}\big\|\partial_r\psi^-\big\|_{L^2}-\sqrt{\frac{|z_2|}{z_1}}C_{2,\nu}\|rf\|^2_{L^2}\le0, \label{eq:central_estimate_Bessel}
\end{equation}
where $C_{j,\nu}$ are the positive explicit constants depending on $\nu$ following from the equations above.

Assume first that $\nu\le\frac{1}{\sqrt{3}}$. From \eqref{eq:central_estimate_Bessel} and Lemma \ref{Lem:csq_Hardy}, we obtain
\begin{equation*}
\Big(1-\sqrt{\frac{|z_2|}{z_1}}C_{1,\nu}\Big) \big\|h_\nu^{\frac12}\psi^-\big\|_{L^2}^2-2\max(1,(2\nu)^{-1})\|rf\|_{L^2}\big\|h_\nu^{\frac12}\psi^-\big\|_{L^2}-\sqrt{\frac{|z_2|}{z_1}}C_{2,\nu}\|rf\|^2_{L^2}\le0.
\end{equation*}
Computing the roots of the previous quadratic expression in
$\|h_\nu^{\frac12}\psi^-\|_{L^2}$, this yields
\begin{align*}
\big\|h_\nu^{\frac12}\psi^-\big\|_{L^2}\le\frac{2\max(1,(2\nu)^{-1})+\Big[4\max(1,(2\nu)^{-2})+4\Big(1-\sqrt{\frac{|z_2|}{z_1}}C_{1,\nu})\Big)\sqrt{\frac{|z_2|}{z_1}}C_{2,\nu}\Big]^\frac12}{2\Big(1-\sqrt{\frac{|z_2|}{z_1}}C_{1,\nu}\Big)}\|rf\|_{L^2}.
\end{align*}
Using again Lemma \ref{Lem:csq_Hardy}, we arrive at
\begin{align}\label{eq:main_estim1-Bessel}
\Big\|\frac{\psi}{r}\Big\|_{L^2}\le\frac{2\max(1,(2\nu)^{-1})+\Big[2\max(1,(4\nu)^{-2})+4\Big(1-\sqrt{\frac{|z_2|}{z_1}}C_{1,\nu})\Big)\sqrt{\frac{|z_2|}{z_1}}C_{2,\nu}\Big]^\frac12}{2\nu\Big(1-\sqrt{\frac{|z_2|}{z_1}}C_{1,\nu}\Big)}\|rf\|_{L^2}.
\end{align}

This inequality holds for any $\psi$ of the form $\psi=(h_\nu-z)^{-1}f$ with $f\in C_0^\infty((0,\infty))$.  Applying it to $\psi=(h_\nu-z)^{-1}r^{-1}\varphi$ with $\varphi\in C_0^\infty((0,\infty))$ (and recalling from Lemma \ref{lm:bounded-op-Bessel} that $r^{-1}(h_\nu-z)^{-1}r^{-1}$ extends to a bounded operator on $L^2((0,\infty))$), we obtain
\begin{align}
&\big\|r^{-1}(h_\nu-z)^{-1}r^{-1}\varphi\big\|_{L^2}\notag\\
&\le\frac{2\max(1,(2\nu)^{-1})+\Big[4\max(1,(2\nu)^{-2})+4\Big(1-\sqrt{\frac{|z_2|}{z_1}}C_{1,\nu})\Big)\sqrt{\frac{|z_2|}{z_1}}C_{2,\nu}\Big]^\frac12}{2\nu\Big(1-\sqrt{\frac{|z_2|}{z_1}}C_{1,\nu}\Big)}\|\varphi\|_{L^2} \label{eq:main_estim2-Bessel}
\end{align}
for any $\varphi\in C_0^\infty((0,\infty))$. Since $C_0^\infty((0,\infty))$ is dense in~$L^2$, estimate \eqref{eq:main_estim2-Bessel} leads to
\begin{align}\label{eq:main_estim3-Bessel}
&\big\|r^{-1}(h_\nu-z)^{-1}r^{-1}\big\|_{L^2\to L^2}\notag\\
&\le\frac{2\max(1,(2\nu)^{-1})+\Big[4\max(1,(2\nu)^{-2})+4\Big(1-\sqrt{\frac{|z_2|}{z_1}}C_{1,\nu})\Big)\sqrt{\frac{|z_2|}{z_1}}C_{2,\nu}\Big]^\frac12}{2\nu\Big(1-\sqrt{\frac{|z_2|}{z_1}}C_{1,\nu}\Big)}.
\end{align}
Applying now Proposition \ref{prop:epsilon-Bessel}, we deduce that
\begin{align}\label{eq:main_estim4-Bessel}
&\sup_{z_1>0,\, z_2\neq0}\big\|r^{-1}(h_\nu-z)^{-1}r^{-1}\big\|\notag\\
&\le\lim_{\varepsilon\to0}\sup_{z_2=\varepsilon z_1} \frac{2\max(1,(2\nu)^{-1})+\Big[4\max(1,(2\nu)^{-2})+4\Big(1-\sqrt{\varepsilon}\, C_{1,\nu})\Big)\sqrt{\varepsilon}\, C_{2,\nu}\Big]^\frac12}{2\nu\Big(1-\sqrt{\varepsilon}\, C_{1,\nu}\Big)}\notag\\
&=\max\Big(\frac{1}{\nu^2},\frac{2}{\nu}\Big).
\end{align}
This proves \eqref{eq:main_Bessel} in the case where $\nu\le\frac{1}{\sqrt{3}}$.

Suppose now that $\nu>\frac{1}{\sqrt{3}}$. Going back to \eqref{eq:central_estimate_Bessel}, we have
\begin{align*}
&\Big(1-\sqrt{\frac{|z_2|}{z_1}}C_{1,\nu}\Big)^2 \big\|h_\nu^{\frac12}\psi^-\big\|_{L^2}^4\notag\\
&\le\Big(2\|rf\|_{L^2}\big\|\partial_r\psi^-\big\|_{L^2}-\sqrt{\frac{|z_2|}{z_1}}C_{2,\nu}\|rf\|^2_{L^2}\Big)^2\notag\\
&\le4\|rf\|^2_{L^2}\big\|\partial_r\psi^-\big\|^2_{L^2}+\frac{|z_2|}{z_1}C_{2,\nu}^2\|rf\|^4_{L^2}\notag\\
&=4\|rf\|^2_{L^2}\big\|h_\nu^{\frac12}\psi^-\big\|^2_{L^2}-4\Big(\nu^2-\mbox{$\frac14$}\Big)\Big\|\frac{\psi}{r}\Big\|_{L^2}^2\|rf\|_{L^2}^2+\frac{|z_2|}{z_1}C_{2,\nu}^2\|rf\|^4_{L^2}.
\end{align*}
Computing the discriminant of the previous quadratic expression in $\|h_\nu^{\frac12}\psi^-\|_{L^2}^2$, we obtain
\begin{align*}
16\|rf\|^4_{L^2}-4\Big(1-\sqrt{\frac{|z_2|}{z_1}}C_{1,\nu}\Big)^2
\Big(4\Big(\nu^2-\mbox{$\frac14$}\Big)\Big\|\frac{\psi}{r}\Big\|_{L^2}^2\|rf\|_{L^2}^2-\frac{|z_2|}{z_1}C_{2,\nu}^2\|rf\|^4_{L^2}\Big)\ge0,
\end{align*}
which implies
\begin{align}\label{eq:main_estim5-Bessel}
16\Big(1-\sqrt{\frac{|z_2|}{z_1}}C_{1,\nu}\Big)^2\Big(\nu^2-\mbox{$\frac14$}\Big)\Big\|\frac{\psi}{r}\Big\|_{L^2}^2\le\Big(16+4\Big(1-\sqrt{\frac{|z_2|}{z_1}}C_{1,\nu}\Big)^2\frac{|z_2|}{z_1}C_{2,\nu}^2\Big)\|rf\|^2_{L^2}.
\end{align}
We can then proceed exactly as before, using \eqref{eq:main_estim5-Bessel} instead of \eqref{eq:main_estim1-Bessel}, which leads to \eqref{eq:main_Bessel} in the case where $\nu>\frac{1}{\sqrt{3}}$. This concludes the proof of the theorem.
\end{proof}

\subsection{One-channel lower bounds}\label{Sec.lower}
We record the channel lower bounds that complement the upper
estimate in Theorem~\ref{thm:main_Bessel} and the lower bound in Theorem~\ref{Thm.lower}. Using a logarithmic trial state near the origin, the latter can be improved for small values of $\nu$, giving, for any $z\in\mathbb{C}\setminus[0,\infty)$, the lower
bound $\nu^{-2}$ for the norm of the weighted resolvent of $h_\nu$. This in turn yields part~(ii) of Theorem~\ref{Thm.new}.
\begin{Proposition}\label{prop:lower-bound-Bessel-small-nu}
Let $\nu>0$. Then, for all $z\in\mathbb{C}\setminus[0,\infty)$,
\begin{equation}\label{eq:lower_bound_Bessel_small_nu}
\left\|r^{-1} (h_\nu - z)^{-1} r^{-1}\right\|_{L^2 \to L^2} \ge \frac{1}{\nu^2}.
\end{equation}
\end{Proposition}
\begin{proof}
Let $z\in\mathbb{C}\setminus[0,\infty)$.  Let
$\chi \in C_{c}^{\infty}([0,\infty))$ be such that
$\mathbf{1}_{[0,1]}\le \chi$, and let
$\eta\in C^\infty((0,\infty))$ be such that $\eta=0$ on
$(0,\frac12]$ and $\eta=1$ on $[1,\infty)$.  We set
$\eta_\sigma(r):=\eta(\frac{r}{\sigma})$.  For
$0<\sigma<1$, define
\begin{equation*}
  u_{\sigma}(r):=r^{\frac12}\chi(r)\eta_\sigma(r).
\end{equation*}
Since $\chi\eta_\sigma\in C^\infty_c((0,\infty))$, we have
$u_\sigma\in H^2((0,\infty))$ and
\begin{equation*}
  \partial_r^2u_\sigma
  =
  -\frac{u_\sigma}{4r^2}
  + r^{-\frac12}(\chi\eta_\sigma)'
  + r^{\frac12}(\chi\eta_\sigma)''.
\end{equation*}
This implies that $u_\sigma\in\mathcal{D}(h_\nu)$ and that
$(h_\nu-z)u_\sigma\in\mathcal{D}(r)$ with
\begin{equation}\label{eq:rhnu}
  r(h_\nu-z)u_{\sigma}
  =
  \nu^2\frac{u_\sigma}{r}
  - zru_\sigma
  + r^{\frac12}(\chi\eta_\sigma)'
  + r^{\frac32}(\chi\eta_\sigma)''.
\end{equation}
For $0<\sigma<1$ and $z\in\mathbb{C}\setminus[0,\infty)$, we set
\begin{equation*}
  v_{\sigma}(z):=r(h_\nu-z)u_{\sigma},
  \qquad \text{so that}\qquad
  u_\sigma=(h_\nu-z)^{-1}r^{-1}v_{\sigma}(z).
\end{equation*}
We claim that
  \begin{equation}\label{eq:Oep-Bessel}
    \frac{\|r^{-1}(h_\nu-z)^{-1}r^{-1}v_{\sigma}(z)\|_{L^{2}}}
        {\|v_{\sigma}(z)\|_{L^{2}}} = \frac{\|r^{-1}u_{\sigma}\|_{L^{2}}}
        {\|r(h_\nu-z)u_{\sigma}\|_{L^{2}}}
    \ge
   \frac{1}{\nu^2-C_{\chi,\eta}(1+|z|)(\ln\sigma)^{-1}},
  \end{equation}
for some positive constant $C_{\chi,\eta}$ depending on
$\chi$ and $\eta$ but independent of $\sigma$ and $z$.
To prove \eqref{eq:Oep-Bessel}, we first observe that
\begin{equation}\label{eq:ln-sigma}
  \|r^{-1}u_{\sigma}\|_{L^{2}}^2
  =
  \int_0^\infty r^{-1}(\chi\eta_\sigma)^2
  \ge
  \int_\sigma^1 r^{-1}
  =
  -\ln\sigma>0.
\end{equation}
Next we estimate each term from the right-hand side of
\eqref{eq:rhnu}.  First,
\begin{equation*}
  \|zru_\sigma\|^2_{L^2}
  \le
  |z|^2\int_0^\infty r^3(\chi\eta_\sigma)^2
  \le
  C_{\chi,\eta}|z|^2.
\end{equation*}
Moreover,
\begin{equation*}
  \big\|r^{\frac12}(\chi\eta_\sigma)'\big\|_{L^2}^2
  \le
  C_{\chi,\eta}
  + C_{\chi,\eta}\sigma^{-2}
  \int_{\frac\sigma2}^\sigma r
  \le
  C_{\chi,\eta},
\end{equation*}
and a similar computation gives
\begin{equation*}
\big\|r^{\frac32}(\chi\eta_\sigma)''\big\|_{L^2}^2\le C_{\chi,\eta}.
\end{equation*}
Letting
$A_{\sigma}(z):=zru_\sigma+r^{\frac12}(\chi\eta_\sigma)'
+r^{\frac32}(\chi\eta_\sigma)''$, we have shown that
$\|A_{\sigma}(z)\|_{L^2}\le C_{\chi,\eta}(1+|z|)$.
Using~\eqref{eq:rhnu} and~\eqref{eq:ln-sigma}, we can write
\begin{align*}
  \frac{\|r^{-1}u_{\sigma}\|_{L^{2}}}
    {\|r(h_\nu-z)u_{\sigma}\|_{L^{2}}}
  &\ge
  \frac{\|r^{-1}u_{\sigma}\|_{L^{2}}}
    {\nu^2\|r^{-1}u_\sigma\|_{L^2}
     +\|A_{\sigma}(z)\|_{L^2}} \\
  &=
  \frac{1}
    {\nu^2
     +\|A_{\sigma}(z)\|_{L^2}
      \|r^{-1}u_{\sigma}\|_{L^{2}}^{-1}} \\
  &\ge
  \frac{1}{\nu^2-C_{\chi,\eta}(1+|z|)(\ln\sigma)^{-1}}.
\end{align*}
This establishes \eqref{eq:Oep-Bessel}, which in turn yields that, for all $0<\sigma<1$ and $z\in\mathbb{C}\setminus[0,\infty)$,
\begin{equation*}
  \big\|r^{-1}(h_\nu-z)^{-1}r^{-1}\big\|_{L^2\to L^2}
  \ge
  \frac{1}{\nu^2-C_{\chi,\eta}(1+|z|)(\ln\sigma)^{-1}}.
\end{equation*}
Letting $\sigma\to0$ concludes the proof.
\end{proof}
\begin{Remark}
Together with Lemma~\ref{lem:Hardy-Bessel}, Proposition \ref{prop:lower-bound-Bessel-small-nu} implies that, for all $\nu>0$ and $z\in\mathbb{C}$ such that $\Re z\le0$, $z\neq0$,
\begin{equation}
\left\| r^{-1} (h_\nu - z)^{-1} r^{-1}  \right\|_{\sii \to \sii} 
= \frac{1}{\nu^2}.
 \end{equation}
\end{Remark}
Combining Proposition \ref{prop:lower-bound-Bessel-small-nu} and Theorem~\ref{thm:main_Bessel}, we deduce part~(ii) of Theorem~\ref{Thm.new}:
\begin{Corollary}\label{cor:lower_bound_Bessel}
 Let  $0<\nu\le\frac12$. One has
\begin{equation}\label{eq:equality_Bessel_small_nu}
\|A_\nu^\pm\|_{L^2\to L^2}=\sup_{z\in\mathbb{C}\setminus[0,\infty)}\left\|r^{-1} (h_\nu - z)^{-1} r^{-1}\right\|_{L^2 \to L^2} = \frac{1}{\nu^2}.
\end{equation}
\end{Corollary}
%

\section{Application to the multidimensional operators}\label{Sec.all}
%
\subsection{The Laplacian}\label{Sec.wave}
We now pass from the one-channel Bessel operators in $\sii((0,\infty))$
to the Laplacian in $\sii(\Real^d)$
by the usual partial wave decomposition.  As a consequence of
Theorem~\ref{thm:main_Bessel}, we obtain
Theorem~\ref{thm:main_Laplacian} for 
the $d$-dimensional Laplacian with $d\ge3$. 
Let us abbreviate
\begin{equation*}
N_d(z):= \left\||x|^{-1} (-\Delta - z)^{-1} |x|^{-1}\right\|_{L^2(\Real^d) \to L^2(\Real^d)}.
\end{equation*}
The main ingredient is the following proposition, which is a direct
consequence of the well-known expression of the Laplacian in
spherical coordinates
$(r,\sigma) \in \mathbb{R}_+ \times \mathbb{S}^{d-1}$,
with $\mathbb{R}_+:=(0,\infty)$.
Note also that~\eqref{eq:nonincreasing_Nd} establishes
Proposition~\ref{Prop.monotonicity} stated in the introduction.

\begin{Proposition}\label{prop:Laplacian->Bessel}
Let $d\ge3$. For all $z\in\mathbb{C}\setminus[0,\infty)$, we have
\begin{align}
N_d(z)=\sup_{\ell\ge0} \big \| r^{-1} \big( h_{\nu_{\ell,d}}- z\big)^{-1} r^{-1} \big\|_{L^2(\mathbb{R}_+) \to L^2(\mathbb{R}_+)} ,
\label{eq:Nd_Lambdanu0}
\end{align}
where $\nu_{\ell,d} := \ell + \frac{d}{2} - 1$.
In particular, the maps
\begin{equation}\label{eq:nonincreasing_Nd}
d\mapsto N_{2d+1}(z) \quad\text{ and }\quad  d\mapsto N_{2d+2}(z) 
\end{equation}
are non-increasing on $\mathbb{N}^*:=\{1,2,3,\dots\}$. 
\end{Proposition}
\begin{proof}
By using the spherical coordinates $(r,\sigma)$, 
let 
\begin{equation*}
U_1:L^2(\mathbb{R}^d)
\to 
L^2(\mathbb{R}_+,r^{d-1}\der r)\otimes L^2(\mathbb{S}^{d-1},\der\sigma)
\end{equation*}
be the unitary operator defined by 
$(U_1f)(r,\sigma) := f(r\sigma)$ and recall that
\begin{equation}\label{spherical}
U_1\Delta U_1^*=\partial_r^2+\frac{d-1}{r}\partial_r+\frac{1}{r^2}\Delta_{\mathbb{S}^{d-1}},
\end{equation}
in the sense of quadratic forms in $\sii(\Real^d)$
defined on $C_0^\infty(\Real^d)$, where $\Delta_{\mathbb{S}^{d-1}}$ stands for the Laplace--Beltrami operator on the sphere $\mathbb{S}^{d-1}$. Denote by $\lambda_\ell:=-\ell^2-\ell(d-2)$, $\ell\ge0$, the eigenvalues of $\Delta_{\mathbb{S}^{d-1}}$ and let $K_\ell$ be the associated eigenspaces. 

Let
\begin{equation}\label{U2}
U_2:L^2(\mathbb{R}_+,r^{d-1} \der r)\otimes L^2(\mathbb{S}^{d-1},\der\sigma)\to\bigoplus_{\ell=0}^{\infty}L^2(\mathbb{R}_+,r^{d-1}\der r)\otimes K_\ell
\end{equation}
be defined by $U_2(f\otimes\varphi)=\bigoplus_{\ell=0}^\infty(f\otimes\pi_{K_\ell}\varphi)$ and extended by linearity, where $\pi_{K_\ell}$ stands for the orthogonal projection onto $K_\ell$ in $L^2(\mathbb{S}^{d-1},\der\sigma)$. Then $U_2$ is a unitary operator and
\begin{equation*}
U_2\Big(\partial_r^2+\frac{d-1}{r}\partial_r+\frac{1}{r^2}\Delta_{\mathbb{S}^{d-1}}\Big)U_2^*=\bigoplus_{\ell=0}^\infty\Big(\partial_r^2+\frac{d-1}{r}\partial_r+\frac{\lambda_\ell}{r^2}\Big)\otimes 1_{K_\ell}.
\end{equation*}

Now let
\begin{equation*}
u_3: L^2(\mathbb{R}_+,r^{d-1}\der r)\to L^2(\mathbb{R}_+,\der r)
=\sii(\Real_+)
\end{equation*}
be the unitary operator defined by $(u_3f)(r)= r^{(d-1)/2}f$. Then
\begin{equation*}
u_3 \Big ( \partial_r^2+\frac{d-1}{r}\partial_r+\frac{\lambda_\ell}{r^2} \Big ) u_3^* = \partial_r^2-\frac{\nu_\ell^2-\frac14}{r^2} , \qquad \nu_{\ell,d} := \ell + \frac{d}{2} - 1,
\end{equation*}
for all $\ell\ge0$.

Finally, let 
\begin{equation}\label{eq:unitary_U}
U:=\Big(\bigoplus_{\ell=0}^\infty u_3\otimes1_{K_\ell}\Big)U_2U_1:L^2(\mathbb{R}^d)\to \bigoplus_{\ell=0}^{\infty}L^2(\mathbb{R}_+)\otimes K_\ell.
\end{equation}

Taking the Friedrichs extensions, we deduce from the previous identities that
\begin{equation*}
U(-\Delta)U^*=\bigoplus_{\ell=0}^\infty h_{\nu_{\ell,d}} \otimes 1_{K_\ell},
\end{equation*}
and therefore
\begin{equation}\label{therefore}
  |x|^{-1} (-\Delta - z)^{-1} |x|^{-1} = U \Big ( \bigoplus_{\ell=0}^\infty r^{-1} \big(h_{\nu_{\ell,d}} - z\big)^{-1} r^{-1} \otimes 1_{K_\ell} \Big)U^*.
\end{equation}
This yields
\begin{align}
N_d(z)&=\sup_{\ell\ge0} \big \| r^{-1} \big( h_{\nu_{\ell,d}} - z \big)^{-1} r^{-1} \otimes 1_{K_\ell} \big\|_{
L^2(\mathbb{R}_+)\otimes K_\ell\to L^2(\mathbb{R}_+) \otimes K_\ell}
\notag\\
&=\sup_{\ell\ge0} \big \| r^{-1} \big( h_{\nu_{\ell,d}} - z \big)^{-1} r^{-1} \big\|_{L^2(\mathbb{R}_+) \to L^2(\mathbb{R}_+)}. \label{eq:Nd_Lambdanu}
\end{align}
The fact that the maps in \eqref{eq:nonincreasing_Nd} are nonincreasing directly follows from the last equality since $\nu_{\ell,d} = \ell + \frac{d}{2} - 1$. 
\end{proof}

Theorem~\ref{thm:main_Laplacian} now follows by combining
Proposition~\ref{prop:Laplacian->Bessel} with
Theorem~\ref{thm:main_Bessel}.
\begin{proof}[Proof of Theorem~\ref{thm:main_Laplacian}]
For any $d\ge3$, we deduce from Proposition \ref{prop:Laplacian->Bessel} that
\begin{equation}\label{deduce}
\left\||x|^{-1} (-\Delta - z)^{-1} |x|^{-1}\right\|_{L^2(\Real^d) \to L^2(\Real^d)}\\
= \sup_{\ell\ge0} \big \| r^{-1} \big( h_{\nu_{\ell,d}} - z \big)^{-1} r^{-1} \big\|_{L^2(\mathbb{R}_+) \to L^2(\mathbb{R}_+)}.
\end{equation}
Now if $d=3$, then $\nu_{0,3}=\frac12$ and $\nu_{\ell,3}\ge\frac32>\frac{1}{\sqrt{3}}$ for $\ell\ge1$. Thus Theorem \ref{thm:main_Bessel} gives
\begin{align*}
 \sup_{\ell\ge0} \big \| r^{-1} \big( h_{\nu_{\ell,d}} - z \big)^{-1} r^{-1} \big\|_{L^2(\Real_+) \to L^2(\Real_+)}\le \max\Big(\frac{2}{\nu_{0,3}},\sup_{\ell\ge1}\frac{1}{\sqrt{\nu_{\ell,3}^2-\frac14}}\Big)=4.
\end{align*}
Likewise, if $d\ge4$, then Theorem \ref{thm:main_Bessel} gives
\begin{align*}
 \sup_{\ell\ge0} \big \| r^{-1} \big( h_{\nu_{\ell,d}} - z \big)^{-1} r^{-1} \big\|_{L^2(\Real_+) \to L^2(\Real_+)}
 \le \sup_{\ell\ge0}\frac{1}{\sqrt{\nu_{\ell,d}^2-\frac14}}=\frac{1}{\sqrt{\nu_{0,d}^2-\frac14}}=\frac{2}{\sqrt{(d-1)(d-3)}}.
\end{align*}
This concludes the proof of Theorem~\ref{thm:main_Laplacian}.
\end{proof}

Similarly, Proposition~\ref{prop:lower-bound-Bessel-small-nu} 
gives the lower bound in the lowest channel for any $z\in\mathbb{C}\setminus[0,\infty)$.
\begin{Proposition}\label{pro:opt3D}
  Let $d\ge3$ and let $z\in\mathbb{C}\setminus[0,\infty)$.
  Then
  \begin{equation}\label{eq:belowb}
    \left\|
      |x|^{-1}(-\Delta-z)^{-1}|x|^{-1}
    \right\|_{L^2(\Real^d)\to L^2(\Real^d)}
    \ge \left(\frac d2-1\right)^{-2}.
  \end{equation}
\end{Proposition}
\begin{proof}
Let $z\in\mathbb{C}\setminus[0,\infty)$. Using \eqref{deduce} and Proposition~\ref{prop:lower-bound-Bessel-small-nu}, we obtain 
\begin{align*}
\left\||x|^{-1} (-\Delta - z)^{-1} |x|^{-1}\right\|_{L^2(\Real^d) \to L^2(\Real^d)} &= \sup_{\ell\ge0} \big \| r^{-1} \big( h_{\nu_{\ell,d}} - z \big)^{-1} r^{-1} \big\|_{L^2(\mathbb{R}_+) \to L^2(\mathbb{R}_+)}\\
&\ge \sup_{\ell\ge0} \frac{1}{\nu_{\ell,d}^2}=\frac{1}{\nu_{0,d}^2}= \left(\frac d2-1\right)^{-2} .
\end{align*}
This concludes the proof of Proposition~\ref{pro:opt3D}.
\end{proof}

By combining Proposition~\ref{pro:opt3D} and Theorem~\ref{thm:main_Laplacian},
we deduce the sharp result~\eqref{eq:sharpdim3}
in the three-dimensional case.
In dimension $4$ we have the bounds
\begin{equation}\label{eq:sharpdim4}
  1
  \le
  \sup_{z\in\Com\setminus[0,\infty)}
  \left\|
    |x|^{-1}(-\Delta-z)^{-1}|x|^{-1}
  \right\|_{L^2(\Real^4)\to L^2(\Real^4)}
  \le
  \frac{2}{\sqrt3} 
  \approx 1.15.
\end{equation}
A similar consequence of the argument used in the proof of 
Proposition~\ref{prop:Laplacian->Bessel} is the following remark regarding the two dimensional case.
\begin{Remark}
Suppose that $d=2$. Restricting the weighted resolvent to the orthogonal of the vectors with zero angular momentum, we have the following uniform resolvent estimate:
\begin{equation*}
\sup_{z\in\mathbb{C}\setminus[0,\infty)} \left\||x|^{-1} (-\Delta - z)^{-1} |x|^{-1} \Pi^\perp \right\|_{L^2(\Real^2) \to L^2(\Real^2)}\le \frac{2}{\sqrt{3}},
\end{equation*}
where $\Pi$ is the orthogonal projection onto $U^*(L^2(\mathbb{R}_+)\otimes K_0)U$ (see the proof of Proposition \ref{prop:Laplacian->Bessel} for the notations, in particular \eqref{eq:unitary_U} for the definition of $U$) and $\Pi^\perp=1-\Pi$.
\end{Remark}

The lower bound given by Proposition~\ref{pro:opt3D} can be improved for large $d$, by the same argument, using the lower bound for $\|A_\nu^\pm\|_{L^2(\mathbb{R}_+)\to L^2(\mathbb{R}_+)}$ obtained in Theorem \ref{Thm.lower}.
\begin{Proposition}\label{pro:larged}
  Let $d\ge3$. 
  One has
  \begin{equation*}
  \sup_{z\in\mathbb{C}\setminus[0,\infty)}  \left\|
      |x|^{-1}(-\Delta-z)^{-1}|x|^{-1}
    \right\|_{L^2(\Real^d)\to L^2(\Real^d)}
    \ge \left(\frac d2-1\right)^{-1}\left(\frac{\pi^2}{12}+\frac18\left(\frac d2-1\right)^{-2}\right)^{\frac12} .
  \end{equation*}
\end{Proposition}

The estimates of Propositions~\ref{pro:opt3D} or~\ref{pro:larged} 
and Theorem~\ref{thm:main_Laplacian} do not match in the limit 
of large dimensions.  
To obtain the correct asymptotics, one can use the asymptotics of $\|A_\nu^\pm\|_{L^2(\mathbb{R}_+)\to L^2(\mathbb{R}_+)}$ obtained in Theorem \ref{thm:asymptotics}.
\begin{Proposition}\label{prop:asymptotics-d}
One has 
\begin{equation}\label{eq:aympt-d}
\sup_{z\in\Com\setminus[0,\infty)}
  \left\|
    |x|^{-1}(-\Delta-z)^{-1}|x|^{-1}
  \right\|_{L^2(\Real^d)\to L^2(\Real^d)}=\frac{2}{d} + O(d^{-3/2}),\qquad d\to\infty.
\end{equation}
\end{Proposition}
\begin{proof}
By \eqref{eq:main_Bessel} and \eqref{deduce}, we have
\begin{align*}
\sup_{z\in\Com\setminus[0,\infty)} \left\||x|^{-1} (-\Delta - z)^{-1} |x|^{-1}\right\|_{L^2(\Real^d) \to L^2(\Real^d)}
&= \sup_{\ell\ge0} \big \| A_{\nu_{\ell,d}}^\pm \big\|_{L^2(\mathbb{R}_+) \to L^2(\mathbb{R}_+)} .
\end{align*}
Using Theorem~\ref{thm:main_Laplacian}, we deduce that
\begin{align*}
\big \| A_{\nu_{0,d}}^\pm \big\|_{L^2(\mathbb{R}_+) \to L^2(\mathbb{R}_+)} \le \sup_{z\in\Com\setminus[0,\infty)} \left\||x|^{-1} (-\Delta - z)^{-1} |x|^{-1}\right\|_{L^2(\Real^d) \to L^2(\Real^d)}
&\le \frac{2}{\sqrt{(d-1)(d-3)}},
\end{align*}
and therefore, by \eqref{eq:asymptotics-bound},
\begin{align*}
\frac{1}{\nu_{0,d}} - \frac{C}{\nu_{0,d}^{3/2}} \le \sup_{z\in\Com\setminus[0,\infty)} \left\||x|^{-1} (-\Delta - z)^{-1} |x|^{-1}\right\|_{L^2(\Real^d) \to L^2(\Real^d)} &\le \frac{2}{\sqrt{(d-1)(d-3)}},
\end{align*}
for some positive constant $C$. Since $\nu_{0,d}=\frac{d-1}{2}$, this implies \eqref{eq:aympt-d}.
\end{proof}

Finally, the same argument as in the proof of Proposition \ref{pro:opt3D} 
combined with Proposition~\ref{prop:main_Bessel_Imaginary} also allows us to recover Simon's identity~\eqref{Simon} proved originally in~\cite{Simon_1992}.
\begin{proof}[Proof of~\eqref{Simon}]
One has
\begin{align}\label{eq:sharp_imaginary_Laplaciand}
\sup_{z\in\mathbb{C}\setminus[0,\infty)}\left\| \Im |x|^{-1} (-\Delta - z)^{-1} |x|^{-1} \right\|_{L^2 \to L^2} &=\sup_{z\in\mathbb{C}\setminus[0,\infty)} \sup_{\ell\ge0} 
\big \| r^{-1} \Im \big( h_{\nu_{\ell,d}} - z \big)^{-1} r^{-1} \big\|_{L^2 \to L^2} \notag \\
&= \frac{\pi}{4\nu_{0,d}} = \frac{\pi}{2(d-2)},
\end{align}
which establishes the desired result.
\end{proof}

\subsection{Inverse square potentials}\label{sec:inverse square}

The same partial wave computation applies without change to
the operator~$H_{d,c}$ introduced in~\eqref{inverse}.
The restriction $c > -c_d$
is exactly the condition to guarantee
that the lowest effective Bessel order be positive.

\begin{proof}[Proof of Theorem~\ref{thm:inverse square}]
  We use the notation and the unitary transform $U$ from the proof
  of Proposition~\ref{prop:Laplacian->Bessel}.  Since the
  multiplication operator $|x|^{-2}$ is radial, it preserves the
  spherical harmonic channels.  After the radial conjugation by
  $u_3$, the $\ell$-th channel becomes
  \begin{equation*}
    -\partial_r^2
    +
   \frac{\nu_{\ell,d,c}^2-\frac14}{r^2}
    =
    h_{\nu_{\ell,d,c}},
  \end{equation*}
  where
$   \nu_{\ell,d,c}^2
    =
    \nu_{\ell,d}^2+c.
$    
  Thus, for all $z\in\Com\setminus[0,\infty)$,
  \begin{align*}
    &\left\|
      |x|^{-1}(H_{d,c}-z)^{-1}|x|^{-1}
    \right\|_{L^2(\Real^d)\to L^2(\Real^d)}
    \\
    &\qquad =
    \sup_{\ell\ge0}
    \left\|
      r^{-1}(h_{\nu_{\ell,d,c}}-z)^{-1}r^{-1}
    \right\|_{L^2(\Real_+)\to L^2(\Real_+)} .
  \end{align*}
  The sequence $\ell\mapsto\nu_{\ell,d,c}$ is increasing, and the
  function~$\mathcal{C}$ defined in \eqref{eq:def-Cnu} is decreasing.
  Theorem~\ref{thm:main_Bessel} 
  therefore gives
  the desired bound~\eqref{eq:inverse square-bound}.

  Finally, if $0<\nu_{0,d,c}\le1/2$, the lower bound
  in Proposition~\ref{prop:lower-bound-Bessel-small-nu}, applied to
  the lowest channel, gives
  \begin{equation*}
    \sup_{z\in\Com\setminus[0,\infty)}
    \left\|
      |x|^{-1}(H_{d,c}-z)^{-1}|x|^{-1}
    \right\|
    \ge
    \frac{1}{\nu_{0,d,c}^2}.
  \end{equation*}
  This is exactly $\mathcal{C}(\nu_{0,d,c})$ in the range
  $0<\nu_{0,d,c}\le1/2$.
\end{proof}

Finally, we establish identity~\eqref{eq:inverse square-imaginary}
for the imaginary part of the weighted resolvent 
from Theorem~\ref{thm:imaginary-critical}.

\begin{proof}[Proof of \eqref{eq:inverse square-imaginary}]
  The same channel decomposition gives, for every
  $z\in\Com\setminus[0,\infty)$,
  \begin{align*}
    &|x|^{-1}\Im(H_{d,c}-z)^{-1}|x|^{-1} 
    =
    U\Big(
      \bigoplus_{\ell=0}^\infty
      r^{-1}\Im(h_{\nu_{\ell,d,c}}-z)^{-1}r^{-1}
      \otimes 1_{K_\ell}
    \Big)U^* .
  \end{align*}
  Therefore Proposition~\ref{prop:main_Bessel_Imaginary} yields
  \begin{align*}
    &\sup_{z\in\Com\setminus[0,\infty)}
    \left\|
      |x|^{-1}\Im(H_{d,c}-z)^{-1}|x|^{-1}
    \right\| 
    =
    \sup_{\ell\ge0}
    \frac{\pi}{4\nu_{\ell,d,c}}
    =
    \frac{\pi}{4\nu_{0,d,c}},
  \end{align*}
  because $\ell\mapsto\nu_{\ell,d,c}$ is increasing.
\end{proof}

\subsection{The Aharonov--Bohm operators}\label{sec:AB}
Now we establish Theorems~\ref{Thm.AB} 
and Theorem~\ref{thm:imaginary-critical} 
for the two-dimensional magnetic operator~$-\Delta_\alpha$
introduced in~\eqref{AB.operator}.
\begin{proof}[Proof of Theorem~\ref{Thm.AB}]
We proceed as in the proof of Theorem~\ref{thm:main_Laplacian}
in Section~\ref{Sec.wave}.
In polar coordinates $(r,\sigma) \in \Real_+\times\Sphere^1$,
the Aharonov--Bohm operator reads (\cf~\eqref{spherical})
\begin{equation*}
  U_1(-\Delta_\alpha)U_1^*
  =
  -\partial_r^2
  - \frac{1}{r}\partial_r
  + \frac{1}{r^2}(-i\partial_\sigma-\alpha)^2 .
\end{equation*}
The spectrum of $(-i\partial_\sigma-\alpha)^2$ 
in $\sii(\Sphere^1)$ with domain $H^1(\Sphere^1)$ 
is purely discrete.
The eigenvalues and normalised eigenfunctions read
$\{(m-\alpha)^2\}_{m\in\Int}$
and $\{\phi_m(\sigma) := (2\pi)^{-1/2} e^{im\sigma}\}_{m \in \Int}$.
In the definition~\eqref{U2} of the unitary transform~$U_2$, 
let us replace the eigenspaces~$K_\ell$
by the one dimensional subspaces spanned by $\phi_m$
and let sum over $m \in \Int$ instead of $\ell \in\Nat$.
Similarly, let us perform the same replacements 
in the definition~\eqref{eq:unitary_U}
of the unitary transform~$U$.
Then (\cf~\eqref{therefore})
\begin{equation*} 
  |x|^{-1} (-\Delta_\alpha - z)^{-1} |x|^{-1} 
  = U \Big ( \bigoplus_{m\in\Int} r^{-1} \big(h_{\nu_{m}(\alpha)} - z\big)^{-1} r^{-1} \otimes 1_{\{\phi_m\}} \Big)U^*,
\end{equation*}
where
$$
  \nu_m(\alpha) := |m-\alpha| \,.
$$
Consequently (\cf~\eqref{deduce}),  
\begin{equation}\label{deduce.bis}
\left\||x|^{-1} (-\Delta_\alpha - z)^{-1} |x|^{-1}\right\|_{L^2(\Real^2) \to L^2(\Real^2)}
= \sup_{m \in \Int} \big \| r^{-1} \big( h_{\nu_{m}(\alpha)} - z \big)^{-1} r^{-1} \big\|_{L^2 \to L^2}.
\end{equation}
This identity is an analogue of Proposition~\ref{prop:Laplacian->Bessel}.

To apply Theorem~\ref{thm:main_Bessel}, notice that,
for any $\alpha \not\in \Int$, there exists a unique $m_\alpha \in \Int$
such that $\alpha \in (m_\alpha,m_{\alpha+1})$. 
For all integer $m < m_\alpha$ or $m > m_{\alpha+1}$,
one has $\nu_m(\alpha) > 1 > \frac{1}{\sqrt{3}}$,
so the upper bound of~\eqref{eq:main_Bessel} 
with $\nu := \nu_m(\alpha)$
is less than $\frac{1}{\sqrt{1-\frac{1}{4}}} < 4$. 
If $\dist(\alpha,m_\alpha) < \dist(\alpha,m_{\alpha+1})$,
then $\nu_{m_{\alpha+1}}(\alpha) \in (\frac{1}{2},1)$,
so the upper bound of~\eqref{eq:main_Bessel} 
with $\nu := \nu_{m_{\alpha+1}}(\alpha)$
(respectively, with $\nu := \nu_{m_{\alpha}}(\alpha)$)
is less than
$\frac{2}{\frac{1}{2}} = 4$
or
$\frac{1}{\sqrt{\frac{1}{3}-\frac{1}{4}}} < 4$ 
(respectively, equals 
$
  \frac{1}{\dist(m_\alpha,\alpha)^2} > 4
$).
If $\dist(\alpha,m_\alpha) > \dist(\alpha,m_{\alpha+1})$,
the role of~$m_\alpha$ and $m_{\alpha+1}$ is exchanged
in the previous argument.
Finally, if $\dist(\alpha,m_\alpha) = \dist(\alpha,m_{\alpha+1})$,
then 
$
  \nu_{m_{\alpha}}(\alpha)
  = \frac{1}{2} = \nu_{m_{\alpha+1}}(\alpha)
$,
so the upper bound of~\eqref{eq:main_Bessel} 
with $\nu := \nu_{m_{\alpha+1}}(\alpha) = \nu_{m_{\alpha+1}}(\alpha)$
is given by $\frac{1}{\dist(m_\alpha,\alpha)^2} = 4$.
In summary, applying Theorem~\ref{thm:main_Bessel}, we get
\begin{equation*}
 \sup_{z \in \Com\setminus [0,\infty)} 
 \big \| r^{-1} \big( h_{\nu_{m}(\alpha)} - z \big)^{-1} r^{-1} 
 \big\|_{L^2 \to L^2}
 \leq \frac{1}{\dist(\alpha,\Int)^2}
\end{equation*}
Since this bound is independent of~$m$, 
together with equality~\eqref{deduce.bis} 
it implies the upper bound of Theorem~\ref{Thm.AB}.

To prove the lower bound of Theorem~\ref{Thm.AB},
we combine~\eqref{deduce.bis} and
Proposition~\ref{prop:lower-bound-Bessel-small-nu}
as follows:
\begin{align*} 
\sup_{z \in \Com\setminus [0,\infty)} 
\sup_{m\in\Int} 
\big \| r^{-1} \big( h_{\nu_{m}(\alpha)} - z \big)^{-1} r^{-1} 
\big\|_{L^2 \to L^2} 
&\geq 
\sup_{z \in \Com\setminus [0,\infty)} \big \| r^{-1} \big( h_{\nu_{m_*}(\alpha)} - z \big)^{-1} r^{-1} \big\|_{L^2 \to L^2} 
\\
& \geq 
\frac{1}{\nu_{m_*(\alpha)}^2}
\,,
\end{align*}
where $m_* \in \Int$ is such that $\dist(\alpha,\Int) = |m_*-\alpha|$. 
This together with equality~\eqref{deduce.bis} 
implies the lower bound of Theorem~\ref{Thm.AB}.
\end{proof}

Finally, we establish identity~\eqref{eq:AB-imaginary}
for the imaginary part of the weighted resolvent 
from Theorem~\ref{thm:imaginary-critical}.

\begin{proof}[Proof of \eqref{eq:AB-imaginary}]
Using the decomposition in the proof of Theorem~\ref{Thm.AB},
we have
\begin{align*}
  &\sup_{z\in\Com\setminus[0,\infty)}
  \left\|
    |x|^{-1}\Im(-\Delta_\alpha-z)^{-1}|x|^{-1}
  \right\| 
  =
  \sup_{m\in\Int}
  \sup_{z\in\Com\setminus[0,\infty)}
  \left\|
    r^{-1}\Im(h_{\nu_m(\alpha)}-z)^{-1}r^{-1}
  \right\|_{L^2\to L^2}.
\end{align*}
Proposition~\ref{prop:main_Bessel_Imaginary} gives
\begin{equation*}
  \sup_{m\in\Int}\frac{\pi}{4\nu_m(\alpha)}
  =
  \frac{\pi}{4\,\dist(\alpha,\Int)} ,
\end{equation*}
which proves \eqref{eq:AB-imaginary}.
\end{proof}

\section{The multiplier method for the Laplacian}\label{Sec.multipliers}
%
This final section records a direct route for proving
Theorem~\ref{thm:main_Laplacian}.  
It is kept because it represents an alternative way 
to establish the quantitative uniform resolvent estimate,
which admits a generalisation to potential perturbations 
which are not necessarily radial.

\subsection{Preliminaries}\label{Sec.pre}

We begin by justifying that for any
$z\in\mathbb{C}\setminus[0,\infty)$,
$|x|^{-1} (-\Delta - z)^{-1} |x|^{-1}$ indeed extends to a
bounded operator on $L^2 := \sii(\Real^d)$.
\begin{Lemma}\label{lm:bounded-op}
Let $d\ge3$ and $z\in\mathbb{C}\setminus[0,\infty)$. Then
\begin{equation*}
\left\| |x|^{-1} (-\Delta - z)^{-1} |x|^{-1}  \right\|_{\sii \to \sii} \le \Big( \frac{2}{d-2} \Big )^2\Big( 1 + \frac{|z|}{\dist(z,[0,\infty))} \Big).
 \end{equation*}
\end{Lemma}
\begin{proof}
It suffices to combine Hardy's inequality (valid for every $d \geq 3$)
\begin{equation}\label{Hardy} 
  \forall \psi \in H^1(\Real^d)
  \,, \qquad
  \|\nabla \psi\|_{\sii}^2
  \geq \left( \frac{d-2}{2} \right)^2
  \big\||x|^{-1} \psi\big\|_{\sii}
  \,,
\end{equation}
which implies
\begin{equation}\label{eq:csq_Hardy}
\big\||x|^{-1}(-\Delta)^{-\frac12}\big\|_{\sii \to \sii}=\big\|(-\Delta)^{-\frac12}|x|^{-1}\big\|_{\sii \to \sii}\le\frac{2}{d-2},
\end{equation}
and the estimate
\begin{equation*}
\left\| (-\Delta) (-\Delta - z)^{-1} \right\|_{\sii \to \sii} \le 1 + |z|\left\|(-\Delta - z)^{-1}\right\|_{\sii \to \sii} \le 1 + \frac{|z|}{\dist(z,[0,\infty))}.
\end{equation*}
\end{proof}

Next we establish another easy result showing that if $\Re z\le0$, $z\neq0$, the norm of the weighted resolvent $|x|^{-1} (-\Delta - z)^{-1} |x|^{-1}$ 
is bounded by the inverse of the constant in Hardy's inequality \eqref{Hardy}.
\begin{Lemma}\label{lem:Hardy}
Let $d\ge3$ and $z\in\mathbb{C}$ be such that $\Re z\le0$, $z\neq0$. Then
\begin{equation}\label{lem2.2} 
\left\| |x|^{-1} (-\Delta - z)^{-1} |x|^{-1}  \right\|_{\sii \to \sii} \le \Big( \frac{2}{d-2} \Big )^2.
 \end{equation}
\end{Lemma}
\begin{proof}
Writing, for all $\psi\in H^2(\Real^d)$,
\begin{equation*} 
\begin{aligned} 
  \|\nabla \psi\|_{\sii}^2
  \leq
  \|\nabla \psi\|_{\sii}^2 - \Re z \, \|\psi\|_{\sii}^2
  = \Re\langle\psi,(-\Delta-z)\psi\rangle_{\sii} 
  \leq \Big\|\frac{\psi}{|x|}\Big\|_{\sii} \, 
  \big\||x| (-\Delta-z)\psi\big\|_{\sii}
  \,,	
\end{aligned}
\end{equation*}
we deduce from Hardy's inequality~\eqref{Hardy} that
\begin{equation*}
\Big\|\frac{\psi}{|x|}\Big\|_{\sii}\le\Big( \frac{2}{d-2} \Big )^2\big\||x| (-\Delta-z)\psi\big\|_{\sii}.
\end{equation*}
Let $\varphi\in C_0^\infty(\Real^d\setminus\{0\})$. Applying the previous inequality to $\psi=(-\Delta-z)^{-1}|x|^{-1}\varphi$ (which belongs to $H^2(\Real^d)$), we obtain 
\begin{equation*}
\left\| |x|^{-1} (-\Delta - z)^{-1} |x|^{-1} \varphi \right\|_{\sii\to\sii} \le \Big( \frac{2}{d-2} \Big )^2 \|\varphi\|_{L^2}.
\end{equation*}
Since $C_0^\infty(\Real^d\setminus\{0\})$ is dense in~$\sii$, this concludes the proof.
\end{proof}
In order to estimate the norm of $|x|^{-1} (-\Delta - z)^{-1} |x|^{-1}$ for $\Re z>0$, $\Im z\neq0$, we will use the following proposition which will be proven thanks to the maximum modulus principle. 
\begin{Proposition}\label{prop:epsilon}
Let $d\ge 3$. The map
\begin{equation}\label{eq:map_epsilon0}
(0,\infty)\ni\varepsilon\mapsto\sup_{|\Im z|=\varepsilon \Re z}\left\| |x|^{-1} (-\Delta - z)^{-1} |x|^{-1} \right\|_{\sii \to \sii}
\end{equation}
is non-increasing and we have
\begin{align}\label{eq:sup_epsilon}
&\sup_{z\in\mathbb{C}\setminus[0,\infty)}
\left\|
  |x|^{-1} (-\Delta - z)^{-1} |x|^{-1} \right\|_{\sii \to \sii} =\lim_{\varepsilon\to0} \sup_{|\Im z|=\varepsilon \Re z}\left\|
  |x|^{-1} (-\Delta - z)^{-1} |x|^{-1}
  \right\|_{\sii \to \sii}.
\end{align}
\end{Proposition}

The proof of Proposition \ref{prop:epsilon} relies on the following well-known lemmata. 
\begin{Lemma}\label{Lem.decay}  
Let $\varphi,\psi \in C_0^\infty(\Real^d)$. Then 
\begin{equation*}
  \lim_{\stackrel[z \not\in [0,\infty)]{}{|z| \to \infty}}
  \left\langle
  \varphi, (-\Delta - z)^{-1} \psi 
  \right\rangle_{L^2} = 0.
\end{equation*}
\end{Lemma}
\begin{proof}
If $\dist\big(z,\sigma(-\Delta)\big) \to \infty$,
then the result immediately follows from~\eqref{sa}.   
To cover the other situations at the same time, 
let us consider $z=k^2$ with $\Im k > 0$.
Inspired by \cite[Sec.~3.1.1]{Dyatlov-Zworski},
we employ the representation formula
\begin{equation*} 
  \frac{1}{a^2 - k^2} = \int_0^\infty \frac{\sin(a t)}{a} 
  \, e^{ikt} \, \der t
\end{equation*}
valid with every $a > 0$.
We apply the formula 
with the help of the Fourier transform
as follows: 
\begin{equation*} 
\begin{aligned} 
  |k| \, \left| \left\langle
  \varphi, (-\Delta - z)^{-1} \psi 
  \right\rangle_{L^2}\right|
 & = |k| \, 
  \left| \left\langle\hat\varphi, (|\cdot|^2 - k^2)^{-1} \hat\psi \right\rangle \right|\\
  &= 
  \left|
  \int_{\Real^d} \overline{\hat\varphi(\xi)} \, \hat\psi(\xi)   
  \int_0^\infty \frac{\sin(|\xi| t)}{|\xi|} 
  \, k \, e^{ikt} \, \der t \, \der\xi
  \right|
  \\
  &= 
  \left|
  \int_0^\infty
  \int_{\Real^d} \overline{\hat\varphi(\xi)} \, \hat\psi(\xi) \,   
  \cos(|\xi| t)
  \, e^{ikt} \, \der \xi \, \der t
  \right|
  \,,
\end{aligned}
\end{equation*}
where the second line follows by an integration by parts in~$t$  
and Fubini's theorem.
Note that $\hat\varphi, \hat\psi$ belong to the Schwartz space
$\mathscr{S}(\Real^d)$.
Splitting the integration in~$t$ over $(0,1)$ and $(1,\infty)$,
the first integral is bounded by a constant independent of~$k$:
\begin{equation*} 
\begin{aligned}  
  \left|
  \int_0^1
  \int_{\Real^d} \overline{\hat\varphi(\xi)} \, \hat\psi(\xi) \,  
  \cos(|\xi| t)
  \, e^{ikt} \, \der \xi \, \der t
  \right|
  &\leq \|\hat\varphi\|_{\sii(\Real^d)} \, \|\hat\psi\|_{\sii(\Real^d)} 
  \int_0^1 \der t
  &= \|\varphi\|_{\sii(\Real^d)} \, \|\psi\|_{\sii(\Real^d)} 
  \,.
\end{aligned}
\end{equation*}
For the second integral, we integrate by parts (twice)
with respect to~$|\xi|$ as follows:
\begin{equation*} 
\begin{aligned}  
  \left|
  \int_1^\infty
  \int_{\Real^d} \overline{\hat\varphi(\xi)} \, \hat\psi(\xi) \,  
  \cos(|\xi| t)
  \, e^{ikt} \, \der \xi \, \der t
  \right|
  &= \left|
  \int_1^\infty
  \int_{\Real^d} 
  \left(\frac{\xi}{|\xi|}\cdot\nabla\right)^2
  \left[ \overline{\hat\varphi(\xi)} \, \hat\psi(\xi) \right]  
  \frac{\sin(|\xi| t)}{t^2}
  \, e^{ikt} \, \der \xi \, \der t
  \right|
  \\
  &\leq \|\hat\varphi\|_{H^2(\Real^d)} \, \|\hat\psi\|_{H^2(\Real^d)}
  \int_1^\infty \frac{\der t}{t^2}
  = \|\hat\varphi\|_{H^2(\Real^d)} \, \|\hat\psi\|_{H^2(\Real^d)}
  \,.
\end{aligned}
\end{equation*}
In summary, there exists a positive constant~$C$
independent of~$z$ such that 
\begin{equation*}
   \left| \left\langle
  \varphi, (-\Delta - z)^{-1} \psi 
  \right\rangle_{L^2}\right| 
  \leq \frac{C}{\sqrt{|z|}} 
\end{equation*}
for every $z \in \Com \setminus [0,\infty)$.
(In fact,    
$ C \leq
\|\hat\varphi\|_{H^2(\Real^d)} \, \|\hat\psi\|_{H^2(\Real^d)}
$.)
This, in particular, implies the claim of the lemma.
\end{proof}

For any $0<\varepsilon<1$, we set
\begin{equation}\label{eq:def_Seps}
S_\varepsilon:=\mathbb{C}\setminus\{z\in\mathbb{C}\,|\,|\Im z|\le\varepsilon \, \Re z\}.
\end{equation}
\begin{Lemma}\label{Lem:zero}  
Let $\varphi,\psi \in C_0^\infty(\Real^d)$. Then 
\begin{equation}\label{eq:lemzero}
  \lim_{\stackrel[z\in S_\varepsilon]{}{z \to 0}}
  \left\langle
  \varphi, (-\Delta - z)^{-1} \psi 
  \right\rangle_{L^2} =   \left\langle
  \phi, (-\Delta )^{-1} \psi 
  \right\rangle_{L^2} .
\end{equation}
\end{Lemma}
\begin{proof}
Let $z\in S_\varepsilon$ and write $z_1:=\Re z$, $z_2:=\Im z$. We have
\begin{align}
\left\langle
  \varphi, (-\Delta - z)^{-1} \psi 
  \right\rangle_{L^2} = \int_{\Real^d}\big(|\xi|^2-z\big)^{-1}\overline{\hat\varphi(\xi)}\hat\psi(\xi)\,\mathrm{d}\xi. \label{eq:zero0}
\end{align}
In order to apply Lebesgue's dominated convergence theorem, we bound $||\xi|^2-z|^{-1}$ as follows: If $z_1\le0$, then
\begin{equation}\label{eq:zero1}
\big||\xi|^2-z\big|^{-1}=\big(\big(|\xi|^2-z_1\big)^2+z_2^2\big)^{-\frac12}\le|\xi|^{-2}.
\end{equation}
If $0<z_1<\varepsilon^{-1}z_2$, we have
\begin{align*}
\big||\xi|^2-z\big|^2&=|\xi|^4-2z_1|\xi|^2+z_1^2+z_2^2\notag\\
&\ge|\xi|^4-2z_1|\xi|^2+(1+\varepsilon^2)z_1^2\notag\\
&=\Big[\big(1+\varepsilon^2\big)^{-\frac12}|\xi|^2-\big(1+\varepsilon^2\big)^{\frac12}z_1\Big]^2 + \varepsilon^2\big(1+\varepsilon^2\big)^{-1}|\xi|^4\notag\\
&\ge\varepsilon^2\big(1+\varepsilon^2\big)^{-1}|\xi|^4,
\end{align*}
and therefore
\begin{equation}
\big||\xi|^2-z\big|^{-1}\le\big(1+\varepsilon^2\big)^{-\frac12}\varepsilon^{-1}|\xi|^{-2}. \label{eq:zero2}
\end{equation}
Using \eqref{eq:zero1}, \eqref{eq:zero2} and the fact that $\xi\mapsto|\xi|^{-1}\hat\varphi(\xi)$, $\xi\mapsto|\xi|^{-1}\hat\psi(\xi)$ are square integrable (by Hardy's inequality~\eqref{Hardy}, since $\hat\varphi$, $\hat\psi$ belong to the Schwartz space $\mathscr{S}(\Real^d)$), we can apply Lebesgue's dominated convergence theorem in the integral in \eqref{eq:zero0}. This gives \eqref{eq:lemzero}.
\end{proof}

Now we are ready to prove Proposition \ref{prop:epsilon}. 
\begin{proof}[Proof of Proposition \ref{prop:epsilon}]
First we show that
\begin{align}\label{eq:sup_epsilon0}
&\sup_{z\in\mathbb{C}\setminus[0,\infty)}
\left\|
  |x|^{-1} (-\Delta - z)^{-1} |x|^{-1} \right\|_{\sii \to \sii} =\lim_{\varepsilon\to0} \sup_{z\in S_\varepsilon}\left\|
  |x|^{-1} (-\Delta - z)^{-1} |x|^{-1}
  \right\|_{\sii \to \sii}.
\end{align}
Since $S_{\varepsilon_2}\subset S_{\varepsilon_1}$ if $0<\varepsilon_1<\varepsilon_2<1$, the map 
\begin{equation}\label{eq:map_epsilon}
(0,\infty)\ni\varepsilon\mapsto\sup_{z\in S_\varepsilon}\left\| |x|^{-1} (-\Delta - z)^{-1} |x|^{-1} \right\|_{\sii \to \sii}
\end{equation}
is non-increasing. This implies
\begin{align*}
&\lim_{\varepsilon\to0} \sup_{z\in S_\varepsilon}\left\|
  |x|^{-1} (-\Delta - z)^{-1} |x|^{-1}
  \right\|_{\sii \to \sii}\le\sup_{z\in\mathbb{C}\setminus[0,\infty)}
\left\|
  |x|^{-1} (-\Delta - z)^{-1} |x|^{-1} \right\|_{\sii \to \sii}.
\end{align*}
To prove the converse inequality, we note that, for all $z_0\in\mathbb{C}\setminus[0,\infty)$ we have $z_0\in S_\varepsilon$ for $\varepsilon>0$ small enough and therefore
\begin{align*}
&\left\|
  |x|^{-1} (-\Delta - z_0)^{-1} |x|^{-1}
  \right\|_{\sii \to \sii}\le\lim_{\varepsilon\to0}\sup_{z\in S_\varepsilon}\left\|
  |x|^{-1} (-\Delta - z)^{-1} |x|^{-1}
  \right\|_{\sii \to \sii}.
\end{align*}
This implies \eqref{eq:sup_epsilon0}.

It remains to establish that, for all $\varepsilon>0$,
\begin{align}\label{eq::>epsilon}
&\sup_{|\Im z|=\varepsilon\Re z}\left\|
  |x|^{-1} (-\Delta - z)^{-1} |x|^{-1}
  \right\|_{\sii \to \sii}=\sup_{z\in S_\varepsilon}\left\|
  |x|^{-1} (-\Delta - z)^{-1} |x|^{-1}
  \right\|_{\sii \to \sii}.
\end{align}
 Let $\varepsilon>0$ and $\varphi,\psi\in C_0^\infty(\Real^d\setminus\{0\})$. The map
\begin{equation}
z\mapsto\left\langle\varphi,|x|^{-1}(-\Delta-z)^{-1}|x|^{-1}\psi\right\rangle_{L^2}
\end{equation}
is analytic on $S_\varepsilon$ and continuous up to the boundary (the continuity at $z=0$ follows from Lemma \ref{Lem:zero}). Moreover since $\varphi,\psi\in C_0^\infty(\Real^d\setminus\{0\})$, it follows from Lemma \ref{Lem.decay} that
\begin{equation}
\left|\left\langle\varphi,|x|^{-1}(-\Delta-z)^{-1}|x|^{-1}\psi\right\rangle_{L^2}\right|
\xrightarrow[|z|\to\infty]{}
0 .
\end{equation}
Therefore, by the maximum modulus principle,
\begin{align*}
\sup_{z\in S_\varepsilon}\left|\left\langle\varphi,|x|^{-1}(-\Delta-z)^{-1}|x|^{-1}\psi\right\rangle_{L^2}\right|=\sup_{|\Im z|=\varepsilon\Re z}\left|\left\langle\varphi,|x|^{-1}(-\Delta-z)^{-1}|x|^{-1}\psi\right\rangle_{L^2}\right|.
\end{align*}
Since $C_0^\infty(\Real^d\setminus\{0\})$ is dense in~$L^2$, 
this implies \eqref{eq::>epsilon}.
\end{proof}
%

\subsection{The alternative proof of Theorem~\ref{thm:main_Laplacian}}
%
{Now we are in a position 
to give an alternative multiplier proof of the
free upper bound~\eqref{eq:main_Laplacian}.
The argument is based on the following identity.
\begin{Lemma}\label{lm:multipliers}
Let $d \geq 2$.
Let $z\in\mathbb{C}\setminus[0,\infty)$ be such that $\Re z>0$. 
Let $\psi\in \mathscr{S}(\mathbb{R}^d)$ and let $f:=(-\Delta-z)\psi$. Then
\begin{align}
&\|\nabla\psi^-\|^2+\frac{|z_2|}{\sqrt{z_1}}\int_{\Real^d}|x|
\Big(
\big|\nabla\psi^-\big|^2-\frac{d-1}{2}\frac{|\psi|^2}{|x|^2}
\Big)
\notag\\
&=\Re\int_{\mathbb{R}^d}\Big((d-1)\bar\psi f + \frac{|z_2|}{\sqrt{z_1}}|x|\bar\psi f+ 2 f^- x\cdot\overline{\nabla\psi^-}\Big), \label{eq:identity_multipliers}
\end{align}
where 
\begin{equation}
z_1:=\Re z, \quad z_2:= \Im z, \quad \psi^-(x):=e^{-i\sqrt{z_1}\mathrm{sgn}(z_2)|x|}\psi(x),
\end{equation}
and likewise for $f^-$.
\end{Lemma}
\begin{proof}
The identity~\eqref{eq:identity_multipliers} in a more general 
context of electromagnetic perturbations of the Laplacian
and~$\psi$ being merely $H^1(\mathbb{R}^d)$ is due to~\cite[Lem.~4.1]{CFK} 
(see also \cite[Eq.~(37)]{CK4} for an abridged proof).
We give a proof in the present context to make the paper self-contained.
The painful regularisation procedure of~\cite{CFK}, 
originally due to~\cite{CK2}, is not needed here 
because $\psi\in \mathscr{S}(\mathbb{R}^d)$.

Consider the equation 
\begin{equation}\label{i0}
  -\Delta\psi = z\psi + f
  \,,
\end{equation}
where $\psi,f\in \mathscr{S}(\mathbb{R}^d)$.
First, multiplying~\eqref{i0} by 
$\bar\phi := x\cdot\nabla\bar\psi + \frac{d}{2}\bar\psi$,
integrating over~$\mathbb{R}^d$, 
taking twice the real part of the obtained identity
and integrating by parts, we obtain  
\begin{equation}\label{i1}
  2 \int_{\Real^d} |\nabla \psi|^2 
  = -2 \, z_2 \, \Im \int_{\Real^d} \psi \, x\cdot\nabla\bar\psi 
  + 2 \, \Re \int_{\Real^d} f \, x \cdot\nabla\bar\psi
  + d \, \Re \int_{\Real^d} f \, \bar\psi
  \,.
\end{equation}
Noticing that $\bar\phi = iA\bar\psi$, 
where $A := \frac{1}{2} (x \cdot (-i\nabla) + (-i\nabla) \cdot x)$
is the dilation operator,
\eqref{i1}~can be considered as a consequence 
of the virial identity $2(-\Delta) = i[-\Delta,A]$.
Second, multiplying~\eqref{i0} by~$\bar\psi$,
integrating over~$\mathbb{R}^d$, 
taking the real part of the obtained identity
and integrating by parts, we get
\begin{equation}\label{i2}
  \int_{\Real^d} |\nabla \psi|^2  
  = z_1 \int_{\Real^d} |\psi|^2 
  + \Re \int_{\Real^d} f \, \bar\psi
  \,.
\end{equation}
Third, multiplying~\eqref{i0} by $|x| \bar\psi$,
integrating over~$\mathbb{R}^d$,
taking the real and imaginary part of the obtained identity
and integrating by parts, 
we respectively get
\begin{equation}\label{i3}
  \int_{\Real^d} |x| \, |\nabla\psi|^2 
  - \frac{d-1}{2} \, \int_{\Real^d}\frac{|\psi|^2}{|x|}
  = z_1  \int_{\Real^d} |x| \, |\psi|^2 
  + \Re \int_{\Real^d} |x| \, \bar\psi \, f
\end{equation}
and 
\begin{equation}\label{i4}
  \Im  \int_{\Real^d} \bar\psi \, \frac{x}{|x|} \cdot \nabla\psi
  = z_2 \int_{\Real^d} |x| \, |\psi|^2
  + \Im \int_{\Real^d} |x| \, \bar\psi \, f
  \,.
\end{equation}
Note that the singular term in~\eqref{i3} is well defined 
due to the weighted Hardy inequality
\begin{equation}\label{Hardy.weighted}
  \int_{\Real^d} |x| \, |\nabla\psi|^2
  \geq \left(\frac{d-1}{2}\right)^2 
  \int_{\Real^d} \frac{|\psi|^2}{|x|} 
\end{equation}
valid for $d \geq 2$.
Taking the clever sum
\begin{equation}\label{clever.sum}
  \eqref{i1} - \eqref{i2} + \frac{|z_2|}{\sqrt{z_1}} \, \eqref{i3} 
  - 2 \, \sqrt{z_1} \, \sgn(z_2) \, \eqref{i4}
  \,,
\end{equation}
we arrive at the desired identity~\eqref{eq:identity_multipliers}. 
\end{proof}

We now give the alternative multiplier proof of 
Theorem~\ref{thm:main_Laplacian}.
\begin{proof}[Alternative proof of Theorem~\ref{thm:main_Laplacian}]
Estimate \eqref{eq:main_Laplacian} for $d=3$ follows from the pointwise bound $
  |G_z(x,y)| \leq G_0(x,y)
$,
where~$G_z$ denotes the integral kernel of 
the resolvent $(-\Delta - z)^{-1}$
(see, \eg, \cite[Sec.~2]{FKV}) and the consequence of Hardy's inequality recalled in \eqref{eq:csq_Hardy}.

Now we prove \eqref{eq:main_Laplacian} for $d \geq 4$. 
By Lemma \ref{lem:Hardy} we know that, if $z_1=\Re z\le 0$ (and $z\neq0$), 
then the explicit bound~\eqref{lem2.2} holds.
Now we consider $z\in\mathbb{C}\setminus[0,\infty)$ with $z_1>0$, $\psi\in \mathscr{S}(\mathbb{R}^d)$ and we set $f=(-\Delta-z)\psi$. 
If $f=0$ then $\psi=0$. In what follows we assume that $f\neq0$.
With the notations from Lemma \ref{lm:multipliers}, we deduce from \eqref{eq:identity_multipliers} 
(and the weighted Hardy inequality~\eqref{Hardy.weighted} 
showing that the second term on the left-hand side of \eqref{eq:identity_multipliers} is non-negative) that
\begin{align}
\|\nabla\psi^-\|&\le\Re\int_{\mathbb{R}^d}\Big((d-1)\bar\psi f + \frac{|z_2|}{\sqrt{z_1}}|x|\bar\psi f+ 2 f^- x\cdot\overline{\nabla\psi^-}\Big) \notag\\
&=\Re\int_{\mathbb{R}^d}|x|f^-\Big((d-1)\frac{\overline{\psi^-}}{|x|} + \frac{|z_2|}{\sqrt{z_1}}\overline{\psi^-}+ 2 \frac{x}{|x|}\cdot\overline{\nabla\psi^-}\Big)\notag\\
&\le \||x|f\|_{L^2}\Big\|(d-1)\frac{\overline{\psi^-}}{|x|} + \frac{|z_2|}{\sqrt{z_1}}\overline{\psi^-}+ 2 \frac{x}{|x|}\cdot\overline{\nabla\psi^-}\Big\|_{L^2}. \label{eq:estim_multipliers}
\end{align}
Next we compute
\begin{align}
&\Big\|(d-1)\frac{\overline{\psi^-}}{|x|} + \frac{|z_2|}{\sqrt{z_1}}\overline{\psi^-}+ 2 \frac{x}{|x|}\cdot\overline{\nabla\psi^-}\Big\|_{L^2}^2\notag\\
&=(d-1)^2\Big\|\frac{\psi}{|x|}\Big\|_{L^2}^2+\frac{z_2^2}{z_1}\|\psi\|_{L^2}^2+4\Big\|\frac{x}{|x|}\cdot\overline{\nabla\psi^-}\Big\|_{L^2}^2\notag\\
&\quad+2(d-1)\frac{|z_2|}{\sqrt{z_1}}\int_{\mathbb{R}^d}\frac{|\psi|^2}{|x|}+2(d-1)\int_{\mathbb{R}^d}\frac{x}{|x|^2}\cdot\nabla|\psi|^2+2\frac{|z_2|}{\sqrt{z_1}}\int_{\mathbb{R}^d}\frac{x}{|x|}\cdot\nabla|\psi|^2.
\end{align}
Integrating by parts in the last two integrals 
and using $(d-1)^2-2(d-1)(d-2)=-(d-1)(d-3)$, we obtain
\begin{align*}
&\Big\|(d-1)\frac{\overline{\psi^-}}{|x|} + \frac{|z_2|}{\sqrt{z_1}}\overline{\psi^-}+ 2 \frac{x}{|x|}\cdot\overline{\nabla\psi^-}\Big\|_{L^2}^2=-(d-1)(d-3)\Big\|\frac{\psi}{|x|}\Big\|_{L^2}^2+\frac{z_2^2}{z_1}\|\psi\|_{L^2}^2+4\Big\|\frac{x}{|x|}\cdot\overline{\nabla\psi^-}\Big\|_{L^2}^2.
\end{align*}
Inserting this into \eqref{eq:estim_multipliers} yields
\begin{align*}
\|\nabla\psi^-\|_{L^2}^4&\le \||x|f\|_{L^2}^2\Big(-(d-1)(d-3)\Big\|\frac{\psi}{|x|}\Big\|_{L^2}^2+\frac{z_2^2}{z_1}\|\psi\|_{L^2}^2+4\Big\|\frac{x}{|x|}\cdot\overline{\nabla\psi^-}\Big\|_{L^2}^2\Big).
\end{align*}
Now, we have
\begin{equation*}
|z_2| \, \|\psi\|_{\sii}^2
=\Big|\Im\int_{\mathbb{R}^d}\bar\psi f\Big|\le\Big\|\frac{\psi}{|x|}\Big\|_{L^2}\||x|f\|_{L^2},
\end{equation*}
and therefore we arrive at
\begin{align*}
\|\nabla\psi^-\|_{L^2}^4&\le \||x|f\|_{L^2}^2\Big(-(d-1)(d-3)\Big\|\frac{\psi}{|x|}\Big\|_{L^2}^2+\frac{|z_2|}{z_1}\Big\|\frac{\psi}{|x|}\Big\|_{L^2}\||x|f\|_{L^2}+4\|\nabla\psi^-\|_{L^2}^2\Big).
\end{align*}

Computing the discriminant of the previous quadratic expression in $\|\nabla\psi^-\|_{L^2}^2$, it follows that
\begin{align*}
16\||x|f\|_{L^2}^4+4\Big(-(d-1)(d-3)\Big\|\frac{\psi}{|x|}\Big\|_{L^2}^2+\frac{|z_2|}{z_1}\Big\|\frac{\psi}{|x|}\Big\|_{L^2}\Big)\||x|f\|_{L^2}^2\ge0.
\end{align*}
Since we have assumed that $f\neq0$, this yields
\begin{align}\label{eq:discrim2}
(d-1)(d-3)\Big\|\frac{\psi}{|x|}\Big\|_{L^2}^2-\frac{|z_2|}{z_1}\Big\|\frac{\psi}{|x|}\Big\|_{L^2}\||x|f\|_{L^2}-4\||x|f\|_{L^2}^2\le0.
\end{align}
The discriminant of the previous quadratic expression in $\||x|^{-1}\psi\|_{L^2}$ reads 
\begin{equation*}
(z_1^{-2}z_2^2+16(d-1)(d-3))\||x|f\|_{L^2}^2,
\end{equation*}
and therefore \eqref{eq:discrim2} implies
\begin{align}\label{eq:main_estim1}
\Big\|\frac{\psi}{|x|}\Big\|_{L^2}\le\frac{\frac{|z_2|}{z_1}+\sqrt{\frac{z_2^2}{z_1^2}+16(d-1)(d-3)}}{2(d-1)(d-3)}\||x|f\|_{L^2}.
\end{align}

This inequality holds for any $\psi\in\mathscr{S}(\mathbb{R}^d)$, with $f=(-\Delta-z)\psi$. Let now $\varphi\in\mathscr{S}(\mathbb{R}^d)$ and let $\eta\in C_0^\infty(\mathbb{R}^d\setminus\{0\})$. Observe that $(-\Delta-z)^{-1}|x|^{-1}\eta\varphi\in\mathscr{S}(\mathbb{R}^d)$. Indeed, with $\mathcal{F}$ the usual Fourier transform, and using that $\eta$ is supported away from $0$ and that $z\notin[0,\infty)$, we have
\begin{equation*}
\varphi\in\mathscr{S}(\mathbb{R}^d)\,\Rightarrow\,|x|^{-1}\eta\varphi\in\mathscr{S}(\mathbb{R}^d)\,\Rightarrow\,\mathcal{F}(|x|^{-1}\eta\varphi)\in\mathscr{S}(\mathbb{R}^d)\,\Rightarrow\,(|\xi|^2-z)^{-1}\mathcal{F}(|x|^{-1}\eta\varphi)\in\mathscr{S}(\mathbb{R}^d).
\end{equation*}
Applying then \eqref{eq:main_estim1} to $\psi=(-\Delta-z)^{-1}|x|^{-1}\eta\varphi$ (and recalling from Lemma \ref{lm:bounded-op} that $|x|^{-1}(-\Delta-z)^{-1}|x|^{-1}$ extends to a bounded operator on $L^2$) gives
\begin{align}\label{eq:main_estim2}
\big\||x|^{-1}(-\Delta-z)^{-1}|x|^{-1}\eta\varphi\big\|_{L^2}\le\frac{\frac{|z_2|}{z_1}+\sqrt{\frac{z_2^2}{z_1^2}+16(d-1)(d-3)}}{2(d-1)(d-3)}\|\eta\varphi\|_{L^2},
\end{align}
for any $\varphi\in\mathscr{S}(\mathbb{R}^d)$ and $\eta\in C_0^\infty(\mathbb{R}^d\setminus\{0\})$. Since the set $\{\varphi\in\mathscr{S}(\mathbb{R}^d) \, | \, \exists \eta\in C_0^\infty(\mathbb{R}^d\setminus\{0\}), \, \eta\varphi=\varphi\}$ is dense in $L^2(\mathbb{R}^d)$ (as it contains $C_0^\infty(\mathbb{R}^d\setminus\{0\})$, inequality \eqref{eq:main_estim2} leads to
\begin{align}\label{eq:main_estim3}
\big\||x|^{-1}(-\Delta-z)^{-1}|x|^{-1}\big\|_{L^2\to L^2}\le\frac{\frac{|z_2|}{z_1}+\sqrt{\frac{z_2^2}{z_1^2}+16(d-1)(d-3)}}{2(d-1)(d-3)}.
\end{align}
Applying now Proposition \ref{prop:epsilon}, we deduce that
\begin{align}\label{eq:main_estim4}
\sup_{z_1>0\, , z_2\neq0}\big\||x|^{-1}(-\Delta-z)^{-1}|x|^{-1}\big\|_{L^2\to L^2}&\le\lim_{\varepsilon\to0}\sup_{z_2=\varepsilon z_1}\frac{\varepsilon+\sqrt{\varepsilon^2+16(d-1)(d-3)}}{2(d-1)(d-3)}\notag\\
&=\frac{2}{\sqrt{(d-1)(d-3)}}.
\end{align}
This concludes the alternative proof of \eqref{eq:main_Laplacian}.
\end{proof}
\begin{Remark}\label{rk:scaling}
Similarly as in Remark \ref{rk:scaling-hnu},
by scaling invariance, the norm $\| |x|^{-1} (-\Delta - z)^{-1} |x|^{-1}\|_{\sii\to\sii}$ is constant along the ray $e^{i \arg(z)} (0,\infty)$. More precisely, given any positive number~$\tau$, and defining the unitary transform $U_\tau$ on $\sii$ by 
\begin{equation*}
  (U_\tau \psi)(x) := \tau^{d/2} \, \psi(\tau x) 
  \,,
\end{equation*}
we have
\begin{align*} 
  \left\| |x|^{-1} (-\Delta -  z)^{-1} |x|^{-1}\right\|_{\sii\to \sii} &= \left\|
  U_\tau
  |x|^{-1} (-\Delta - z)^{-1} |x|^{-1}
  U_\tau^{-1}
  \right\|_{\sii\to \sii} 
 \notag  \\
  &=
  \left\|
  |x|^{-1} (-\Delta - \tau^2 z)^{-1} |x|^{-1}
  \right\|_{\sii \to \sii}
  \,,
\end{align*}
for any $\tau>0$. We did not use this property in our proof of 
Theorem~\ref{thm:main_Laplacian}.
\end{Remark}
\begin{Remark}
The map
\begin{equation*}
\mathbb{C}\setminus[0,\infty)\ni z\mapsto |x|^{-1} (-\Delta - z)^{-1} |x|^{-1}
\in \mathcal{B}(L^2)
\end{equation*}
is analytic and, as proved in \cite{Kato-Yajima_1989}, the limits
\begin{equation*}
|x|^{-1} (-\Delta - \lambda \pm i 0)^{-1} |x|^{-1} := \lim_{\varepsilon\to0^+} |x|^{-1} (-\Delta - \lambda \pm i \varepsilon)^{-1} |x|^{-1}
\end{equation*}
exist in the operator norm, for any $\lambda>0$. Moreover, the maps
\begin{equation*}
(0,\infty)\ni \lambda\mapsto |x|^{-1} (-\Delta - \lambda \pm i0)^{-1} |x|^{-1}\in \mathcal{B}(L^2)
\end{equation*}
are continuous. However, $|x|^{-1} (-\Delta - z)^{-1} |x|^{-1}$ does \emph{not} converge to $|x|^{-1} (-\Delta )^{-1} |x|^{-1}$ in $\mathcal{B}(L^2)$ when $z\in\mathbb{C}\setminus[0,\infty)$, $z\to0$. Indeed, if $z$ is any fixed complex number in $\mathbb{C}\setminus[0,\infty)$, then by Remark \ref{rk:scaling},
\begin{equation*}
\begin{aligned}
&  \left\|
  |x|^{-1} (-\Delta-\tau^2 z)^{-1} |x|^{-1} -  |x|^{-1} (-\Delta)^{-1} |x|^{-1}
  \right\|_{\sii \to \sii}  \\
& =  \left\|
  |x|^{-1} (-\Delta- z)^{-1} |x|^{-1} -  |x|^{-1} (-\Delta)^{-1} |x|^{-1}
  \right\|_{\sii \to \sii} .
\end{aligned}  
\end{equation*}
Hence if the left-hand side tends to $0$ as $\tau\to0$, this implies that the right-hand side vanishes for any $z\in\mathbb{C}\setminus[0,\infty)$, which is certainly not true. Similarly, $|x|^{-1} (-\Delta - \lambda \pm i0)^{-1} |x|^{-1}$ does not converge to $|x|^{-1} (-\Delta )^{-1} |x|^{-1}$ in $\mathcal{B}(L^2)$.
\end{Remark}

Applying Theorem \ref{thm:main_Laplacian} to a vector of the form $|x|(-\Delta-z)\psi$ with $\psi$ regular enough 
(\eg, $\psi\in C_0^\infty(\mathbb{R}^d)$) and then letting $\Im z\to0$, we obtain, for instance in the case $d\ge4$, that
\begin{equation}\label{prop:real.pre}
\Big\|\frac{\psi}{|x|}\Big\|_{\sii}\le\frac{2}{\sqrt{(d-1)(d-3)}}
\, \big\||x|(-\Delta-\lambda)\psi\big\|_{L^2} , \quad \lambda>0, \quad \psi\in C_0^\infty(\Real^d).
\end{equation}
For $\lambda\le0$, Hardy's inequality gives a better estimate.
The following proposition shows that even for positive~$\lambda$,
the constant on the right-hand side of~\eqref{prop:real.pre} 
can be improved.
\begin{Proposition}\label{prop:real}
Let $d\ge3$. For every $\lambda>0$ and $\psi\in C_0^\infty(\Real^d)$,
one has
\begin{equation}\label{eq:bound_real_prop}
\Big\|\frac{\psi}{|x|}\Big\|_{\sii(\Real^d)}
\le C_d \, \big\||x|(-\Delta-\lambda)\psi\big\|_{\sii(\Real^d)}  
\qquad \mbox{with} \qquad
C_d :=
\begin{cases}
  4 & \mbox{if} \quad d=3 \,, 
  \\
  1 & \mbox{if} \quad d=4 \,, 
  \\
  \displaystyle
  \frac{1}{\sqrt{d(d-4)}} & \mbox{if} \quad d \geq 5 \,.
\end{cases}
\end{equation}
\end{Proposition}

The proof of Proposition~\ref{prop:real} is based on the following well-known virial identity (see, \eg, \cite[Lem.~4.6]{CFK}).
\begin{Lemma}\label{Lem.virial}
Let $\lambda \in \Real$, $\psi\in C_0^\infty(\Real^d)$ 
and $f:=(-\Delta-\lambda)\psi$. Then
\begin{equation}\label{virial}  
  2 \int_{\Real^d} |\nabla \psi|^2 
  = d \, \Re \int_{\Real^d} f \, \overline{\psi}
  + 2 \, \Re \int_{\Real^d} f \, x \cdot \overline{\nabla \psi}
  \,.
\end{equation}
\end{Lemma}
\begin{proof}
\eqref{virial}~is a special case of the previously established~\eqref{i1}
(for whose validity the hypothesis $z \in \Com \setminus [0,\infty)$
of Lemma~\ref{lm:multipliers} is not needed).  
\end{proof}

Now we are ready to prove Proposition~\ref{prop:real}.
\begin{proof}[Proof of Proposition~\ref{prop:real}]
Using the notations of Lemma \ref{Lem.virial}, and starting with the virial identity \eqref{virial}, we estimate the right-hand side as follows:
\begin{equation*} 
\begin{aligned} 
  \left|
  d \, \Re \int_{\Real^d} f \, \overline{\psi}
  + 2 \, \Re \int_{\Real^d} f \, x \cdot \overline{\nabla \psi}
  \right|
  &=
  \left|
  \Re \int_{\Real^d}
  |x| \, f \left(
  d \, \frac{\overline{\psi}}{|x|}
  + 2 \, \frac{x}{|x|} \cdot \overline{\nabla \psi}
  \right) 
  \right|
  \\
  &\leq 
  \big\| |x| f \big\|_{\sii(\Real^d)} \,
  \left\| d \, \frac{\overline{\psi}}{|x|}
  + 2 \, \frac{x}{|x|} \cdot \overline{\nabla \psi} 
  \right\|_{\sii(\Real^d)} 
  \,.
\end{aligned}  
\end{equation*}
Integrating by parts, we have
\begin{equation*} 
\begin{aligned} 
  \left\| d \, \frac{\overline{\psi}}{|x|}
  + 2 \, \frac{x}{|x|} \cdot \overline{\nabla \psi} 
  \right\|_{\sii(\Real^d)}^2
  &= d^2 \,  \left\| \frac{\psi}{|x|} \right\|_{\sii(\Real^d)}^2
  + 4 \, \left\| \nabla \psi \right\|_{\sii(\Real^d)}^2
  + 2 \, d \int_{\Real^d} \frac{x}{|x|^2} \cdot \nabla |\psi|^2
  \\
  &=
  d^2 \,  \left\| \frac{\psi}{|x|} \right\|_{\sii(\Real^d)}^2
  + 4 \, \left\| \nabla \psi \right\|_{\sii(\Real^d)}^2
  - 2\, d \, (d-2) \left\| \frac{\psi}{|x|} \right\|_{\sii(\Real^d)}^2
  \\
  &= 4 \, \left\| \nabla \psi \right\|_{\sii(\Real^d)}^2
  - d \, (d-4) \left\| \frac{\psi}{|x|} \right\|_{\sii(\Real^d)}^2
  \,.
\end{aligned}  
\end{equation*}
Consequently, \eqref{virial}~implies
\begin{equation}\label{virial.implies}  
  4 \, \left\| \nabla \psi \right\|_{\sii(\Real^d)}^4
  \leq
  \big\| |x| f \big\|_{\sii(\Real^d)}^2
  \left(
  4 \, \left\| \nabla \psi \right\|_{\sii(\Real^d)}^2
  - d \, (d-4) \left\| \frac{\psi}{|x|} \right\|_{\sii(\Real^d)}^2
  \right)  
  \,.
\end{equation}

\medskip
\noindent
\fbox{$d=3$}
Estimate \eqref{eq:bound_real_prop} follows from the pointwise bound $
  |G_z(x,y)| \leq G_0(x,y)
$,
where~$G_z$ denotes the integral kernel of 
the resolvent $(-\Delta - z)^{-1}$
(see, \eg, \cite[Sec.~2]{FKV}) and the consequence of Hardy's inequality recalled in \eqref{eq:csq_Hardy}. 
Anyway, let us show how to deduce it from~\eqref{virial.implies}
with the help of the Hardy inequality~\eqref{Hardy}:
\begin{equation*} 
  4 \, \left\| \nabla \psi \right\|_{\sii(\Real^3)}^4
  \leq
  \big\| |x| f \big\|_{\sii(\Real^3)}^2
  \left(
  4 \, \left\| \nabla \psi \right\|_{\sii(\Real^3)}^2
  + 3 \left\| \frac{\psi}{|x|} \right\|_{\sii(\Real^3)}^2
  \right)  
  \leq 16 \, \big\| |x| f \big\|_{\sii(\Real^3)}^2 \,
  \left\| \nabla \psi \right\|_{\sii(\Real^3)}^2
  \,.
\end{equation*}
Dividing by $\left\| \nabla \psi \right\|_{\sii(\Real^3)}^2$ 
and applying the Hardy inequality~\eqref{Hardy} once more,
we conclude with
\begin{equation*} 
  \left\| \frac{\psi}{|x|} \right\|_{\sii(\Real^3)}^2  
  \leq 16 \, \big\| |x| f \big\|_{\sii(\Real^3)}^2  
  \,,
\end{equation*}
which proves \eqref{eq:bound_real_prop}.

\medskip
\noindent
\fbox{$d=4$}
In four dimensions, \eqref{virial.implies} 
and the Hardy inequality~\eqref{Hardy}
directly imply
\begin{equation*} 
  \left\| \frac{\psi}{|x|} \right\|_{\sii(\Real^4)}^2  
  \leq
  \left\| \nabla \psi \right\|_{\sii(\Real^4)}^2
  \leq
  \big\| |x| f \big\|_{\sii(\Real^4)}^2
  \,,
\end{equation*}
which gives \eqref{eq:bound_real_prop}.

\medskip
\noindent
\fbox{$d \geq 5$}
In these higher dimensions, 
we regard~\eqref{virial.implies} as the quadratic inequality 
\begin{equation*} 
  4 \, A^2 - 4 \, \big\| |x| f \big\|_{\sii(\Real^3)}^2 \, A
  - d \, (d-4) \, \big\| |x| f \big\|_{\sii(\Real^3)}^2    
  \left\| \frac{\psi}{|x|} \right\|_{\sii(\Real^d)}^2
  \leq 0
\end{equation*}
for the unknown $A := \left\| \nabla \psi \right\|_{\sii(\Real^d)}^2$
Since the discriminant reads
\begin{equation*} 
  D :=
  16 \, \big\| |x| f \big\|_{\sii(\Real^d)}^2 
  \left(
  \big\| |x| f \big\|_{\sii(\Real^d)}^2 
  -  d \, (d-4) 
  \left\| \frac{\psi}{|x|} \right\|_{\sii(\Real^d)}^2
  \right)
  \,,
\end{equation*}
estimate \eqref{eq:bound_real_prop}
follows by requiring $D \geq 0$.
\end{proof}

\begin{Remark}
Using in particular the maximum modulus
(or the Phragm\'en--Lindel\"of) principle, one obtains
\begin{equation*}
\sup_{z\in\mathbb{C}\setminus[0,\infty)}\||x|^{-1} (-\Delta - z)^{-1} |x|^{-1} \|_{L^2 \to L^2}=\sup_{\lambda>0}\||x|^{-1} (-\Delta - \lambda\pm i0)^{-1} |x|^{-1} \|_{L^2 \to L^2}.
\end{equation*}
See the proof of Theorem \ref{thm:main_Laplacian} above for a similar argument. It is therefore tempting to think that Proposition~\ref{prop:real} should imply that 
$\||x|^{-1} (-\Delta - \lambda\pm i0)^{-1} |x|^{-1} \|_{L^2 \to L^2}\le C_d$, with $C_d$ the constant given by Proposition~\ref{prop:real}. This is however wrong since, as mentioned in the introduction, we have the lower bound $\||x|^{-1} (-\Delta - \lambda\pm i0)^{-1} |x|^{-1} \|_{L^2 \to L^2}\ge\frac{\pi}{2(d-2)}$ and $\frac{\pi}{2(d-2)}> C_d$ for $d\ge5$. The fact that Proposition~\ref{prop:real} does not imply a bound on $\||x|^{-1} (-\Delta - \lambda\pm i0)^{-1} |x|^{-1} \|_{L^2 \to L^2}$ is due to the lack of integrability of vectors of the form $(-\Delta - \lambda\pm i0)^{-1} |x|^{-1}\varphi$, even if $\varphi$ belongs to $C_0^\infty(\Real^d\setminus\{0\})$.
\end{Remark}


\subsection*{Acknowledgements}
The first author was partially supported
by the Progetto Ricerca Scientifica 2024
``Wave dynamics in heterogeneous media'' of Sapienza University,
and by the Gruppo Nazionale per l'Analisi Matematica, 
la Probabilit\`{a} e le loro Applicazioni (GNAMPA)
The last author was supported
by the grant no.~26-21940S
of the Czech Science Foundation.

%

\begin{thebibliography}{10}

\bibitem{Barcelo-Vega-Zubeldia_2013}
J.~A. Barcel{\'o}, L.~Vega, and M.~Zubeldia, \emph{The forward problem for the
  electromagnetic {H}elmholtz equation with critical singularities}, Adv. Math.
  \textbf{240} (2013), 636--671.

\bibitem{Bouclet-Mizutani_2018}
J.-M. Bouclet and H.~Mizutani, \emph{Uniform resolvent and {S}trichartz
  estimates for {S}chr{\"o}dinger equations with critical singularities},
  Trans. Amer. Math. Soc. \textbf{370} (2018), no.~10, 7293--7333.

\bibitem{Bui-DAncona-Duong-Li-Ly_2017}
T.~A. Bui, P.~D'Ancona, X.~T. Duong, J.~Li, and F.~K. Ly, \emph{Weighted
  estimates for powers and smoothing estimates of {S}chr{\"o}dinger operators
  with inverse square potentials}, J. Differential Equations \textbf{262}
  (2017), no.~3, 2771--2807.

\bibitem{Burq-Planchon-Stalker-TahvildarZadeh_2003}
N.~Burq, F.~Planchon, J.~G. Stalker, and A.~S. Tahvildar-Zadeh,
  \emph{{S}trichartz estimates for the wave and {S}chr{\"o}dinger equations
  with the inverse-square potential}, J. Funct. Anal. \textbf{203}
  (2003), 519--549. 

\bibitem{Burq-Planchon-Stalker-TahvildarZadeh_2004}
N.~Burq, F.~Planchon, J.~G. Stalker, and A.~S. Tahvildar-Zadeh,
  \emph{{S}trichartz estimates for the wave and {S}chr{\"o}dinger equations
  with potentials of critical decay}, Indiana Univ. Math. J. \textbf{53}
  (2004), no.~6, 1665--1680.
  
\bibitem{CK}
C.~Cazacu and D.~Krej{\v{c}}iř{\'{i}}k,
\emph{The Hardy inequality and the heat equation with magnetic field in
  any dimension},
  Comm. Partial Differential Equations \textbf{41} (2016), 1056--1088.   

\bibitem{CFK}
L.~Cossetti, L.~Fanelli, and D.~Krej\v{c}i\v{r}\'ik, \emph{Absence of
  eigenvalues of {D}irac and {P}auli {H}amiltonians via the method of
  multipliers}, Comm. Math. Phys. \textbf{379} (2020), 633--691.

\bibitem{CK2}
L.~Cossetti and D.~Krej\v{c}i\v{r}\'ik, \emph{Absence of eigenvalues of
  non-selfadjoint {R}obin {L}aplacians on the half-space}, Proc. London. Math.
  Soc. \textbf{121} (2020), 584--616.

\bibitem{CK4}
\bysame, \emph{The virial theorem and the method of multipliers in spectral
  theory}, Mathematical Models for Interacting Dynamics on Networks, Trends in
  Mathematics, Birkh{\"a}user Cham, 2025.

\bibitem{DAncona_2015}
P.~D'Ancona, \emph{{K}ato smoothing and {S}trichartz estimates for wave
  equations with magnetic potentials}, Comm. Math. Phys. \textbf{335} (2015),
  1--16.

\bibitem{Derezinski-Richard_2017}
J.~Derezi\'nski and S.~Richard, \emph{On {S}chr{\"o}dinger operators with
  inverse square potentials on the half--line}, Ann. Henri Poincar\'e
  \textbf{18} (2017), 869--928.

\bibitem{Dyatlov-Zworski}
S.~Dyatlov and M.~Zworski, \emph{Mathematical theory of scattering resonances},
  Amer. Math. Soc., 2019.

\bibitem{FKV2}
L.~Fanelli, D.~Krej\v{c}i\v{r}\'{\i}k, and L.~Vega, \emph{Absence of
  eigenvalues of two dimensional magnetic {S}chr{\"o}dinger operators}, J.
  Funct. Anal. \textbf{275} (2018), 2453--2472.

\bibitem{FKV}
\bysame, \emph{Spectral stability of {S}chr{\"o}dinger operators with
  subordinated complex potentials}, J. Spectr. Theory \textbf{8} (2018),
  575--604.

\bibitem{Fanelli-Zhang-Zheng_2023}
L.~Fanelli, J.~Zhang, and J.~Zheng, \emph{Uniform resolvent estimates for
  critical magnetic {S}chr{\"o}dinger operators in 2D}, Int. Math. Res. Not.
  IMRN \textbf{2023}, no.~20, 17656--17703.
  
\bibitem{Frank_2011}
R.~L. Frank, \emph{Eigenvalue bounds for {S}chr{\"o}dinger operators with
  complex potentials}, Bull. Lond. Math. Soc. \textbf{43} (2011), 745--750.  
  
  \bibitem{GR}
 I.~S.~ Gradshteyn and I.~M.~Ryzhik, \emph{Table of integrals, series, and products.}, Amsterdam: Elsevier/Academic Press, 8th edition, 2015.

\bibitem{HK2}
M.~Hansmann and D.~Krej\v{c}i\v{r}{\'\i}k, \emph{The abstract
  {B}irman-{S}chwinger principle and spectral stability}, J. Anal. Math.
  \textbf{148} (2022), 361--398.

\bibitem{Kato_1966}
T.~Kato, \emph{Wave operators and similarity for some non-selfadjoint
  operators}, Math. Ann. \textbf{162} (1966), 258--279.

\bibitem{Kato-Yajima_1989}
T.~Kato and K.~Yajima, \emph{Some examples of smooth operators and the
  associated smoothing effect}, Rev. Math. Phys. \textbf{1} (1989), 481--496.

\bibitem{Kenig-Ruiz-Sogge_1987}
C.~E. Kenig, A.~Ruiz, and C.~D. Sogge, \emph{Uniform sobolev inequalities and
  unique continuation for second order constant coefficient differential
  operators}, Duke Math. J. \textbf{55} (1987), 329--347.

\bibitem{K13}
D.~Krej\v{c}i\v{r}\'{\i}k, \emph{Complex magnetic fields: {A}n improved
  {H}ardy-{L}aptev-{W}eidl inequality and quasi-selfadjointness}, SIAM J.
  Math. Anal. \textbf{51} (2019), 790--807.

\bibitem{Laptev-Weidl_1999}
A.~Laptev and T.~Weidl, \emph{{H}ardy inequalities for magnetic {D}irichlet
  forms}, Oper. Theory Adv. Appl. \textbf{108} (1999), 299--305.

\bibitem{Mizutani-Zhang-Zheng_2020}
H.~Mizutani, J.~Zhang, and J.~Zheng, \emph{Uniform resolvent estimates for
  {S}chr{\"o}dinger operator with an inverse square potential}, J. Funct.
  Anal. \textbf{278} (2020), no.~4, 108350.
  
\bibitem{Mochizuki_2010a}  
K.~Mochizuki, 
\emph{Resolvent estimates for magnetic
Schr\"odinger operators and their
applications to related evolution equations},
Rend. Istit. Mat. Univ. Trieste \textbf{42} (2010), 143--164.

\bibitem{Mochizuki_2010b}  
K.~Mochizuki,
\emph{Uniform resolvent estimates for magnetic
Schr\"odinger operators and smoothing
effects for related evolution equations}
Publ. RIMS Kyoto Univ. \textbf{46} (2010), 741--754.  

\bibitem{Pankrashkin-Richard_2011}
K.~Pankrashkin and S.~Richard, \emph{Spectral and scattering theory for the
  {A}haronov-{B}ohm operators}, Rev. Math. Phys. \textbf{23} (2011), 53--81.

\bibitem{PBM}
A.~P.~Prudnikov, Yu.~A.~Brychkov O.~I.~ Marichev, Integrals and Series, Vol. 2, Special Functions, Gordon \& Breach Sci. Publ., New York, 1986. 

\bibitem{RS1}
M. Reed, B. Simon,
\emph{Methods of modern mathematical physics. IV Functional Analysis},
Academic Press, 
New York-London, 1980.

\bibitem{Szego}
G.~Szeg\"o, Orthogobal Polynomials, 4th edition, 
Amer. Math. Soc., Colloquium Publications, 1975.

\bibitem{Simon_1992}
B.~Simon, \emph{Best constants in some operator smoothness estimates}, J.
  Funct. Anal. \textbf{107} (1992), no.~3, 66--71.
  
\end{thebibliography}
\bibliographystyle{amsplain}
\providecommand{\bysame}{\leavevmode\hbox to3em{\hrulefill}\thinspace}
\providecommand{\MR}{\relax\ifhmode\unskip\space\fi MR }
\providecommand{\MRhref}[2]{%
  \href{http://www.ams.org/mathscinet-getitem?mr=#1}{#2}
}
\providecommand{\href}[2]{#2}

\end{document}